\documentclass[11pt]{amsart}
\usepackage{amssymb,mathrsfs,graphicx,extpfeil}
\usepackage{epsfig}
\usepackage{indentfirst, latexsym, amssymb, enumerate,amsmath,graphicx}
\usepackage{float}
\usepackage{colortbl}
\usepackage{epsfig,subfigure}
\allowdisplaybreaks[4]
\usepackage{amsmath}
\title[Hilbert Expansion of the R-VMB system]{Global Hilbert expansion for the relativistic Vlasov-Maxwell-Boltzmann system}

\author[Y. Guo]{Yan Guo}
\address[Y. Guo]{{\newline Division of Applied Mathematics, Brown University
Providence, RI 02912, USA}}
\email{yan\_guo@brown.edu}

\author[Q. H. Xiao]{Qinghua Xiao}
\address[Qinghua Xiao]{\newline Innovation Academy for Precision Measurement Science and Technology,\newline Chinese Academy of Sciences, Wuhan 430071, China}
\email{xiaoqh@apm.ac.cn}

\newtheorem{theorem}{Theorem}[section]
\newtheorem{lemma}{Lemma}[section]
\newtheorem{corollary}{Corollary}[section]
\newtheorem{proposition}{Proposition}[section]

\newtheorem{remark}{Remark}[section]

\def\charf {\mbox{{\text 1}\kern-.30em {\text l}}}

\begin{document}

\date{\today}

\subjclass{76Y05; 35Q20} \keywords{Hilbert expansion; Relativistic Vlasov-Maxwell-Boltzmann system; $L^2-L^{\infty}$ framework}

\thanks{\textbf{Acknowledgment.} The work of Yan Guo is supported by NSF grant DMS-1810868, the work of Qinghua Xiao is supported by grants from Youth Innovation Promotion Association, CAS (No. 2017379), the National Natural Science Foundation of China under contract 11871469, the State Scholarship Fund of China and Hubei Chenguang Talented Youth Development Foundation.}


\begin{abstract} Consider the relativistic Vlasov-Maxwell-Boltzmann system describing the dynamics of an electron gas in the presence of a fixed ion background. Thanks to recent works \cite{Germain-Masmoudi-ASENS-2014, Guo-Ionescu-Pausader-JMP-2014} and \cite{Deng-Ionescu-Pausader-ARMA-2017}, we establish the global-in-time validity of its Hilbert expansion and derive the limiting relativistic Euler-Maxwell system as the mean free path goes to zero. Our method is based on the $L^2-L^{\infty}$ framework and the Glassey-Strauss Representation of the electromagnetic field,  with auxiliary $H^1$ estimates and $W^{1,\infty}$ estimates to control the characteristic curves and corresponding $L^{\infty}$ norm.
\end{abstract}

\maketitle
\thispagestyle{empty}

\tableofcontents

\section{Introduction}

\setcounter{equation}{0}
\subsection{Relativistic Vlasov-Maxwell-Boltzmann system}
The relativistic Vlasov-Maxwell-Boltzmann system is a fundamental and complete model describing the dynamics of a dilute collisional plasma appearing in  nuclear fusion and the interior of stars, etc. Correspondingly, the relativistic Euler-Maxwell system, the foundation of the two-fluid theory in plasma physics, describes the dynamics of two compressible ion and electron fluids interacting with their own self-consistent electromagnetic field.  It is also the origin of many celebrated dispersive PDE such as NLS, KP, KdV, Zaharov, etc, as various scaling limits and approximations of such a fundamental model. Since the ion mass is far larger than the electron mass in a plasma, the dynamics of ions is negligible for simplification sometimes. In this special case, the plasma can be approximately described by the one-species relativistic Vlasov-Maxwell-Boltzmann system in the mesoscopic level and treated as a single fluid in the macroscopic level. It has been an important open question if the general  Euler-Maxwell system can be derived rigorously from its kinetic counter-part, the Vlasov-Maxwell-Boltzmann system, as the mean field path goes to zero.

In this paper, we are able to answer this question in the affirmative in the special case when the ions form a constant background, and the relativistic Euler-Maxwell system describes the dynamics of the electron gas. The relativistic Vlasov-Boltzmann system can be written as:
\begin{equation}\label{main1}
\begin{split}
& \partial_t F^{\varepsilon} + c\hat{p}\cdot \nabla_x F^{\varepsilon}-e_- \Big(E^{\epsilon}+\hat{p}\times B^{\varepsilon} \Big)\cdot\nabla_p F^{\varepsilon} = \frac{1}{\varepsilon}Q(F^{\varepsilon},F^{\varepsilon}),
 \end{split}
\end{equation}
which is coupled with the Maxwell system
\begin{align}\label{main2}
\begin{aligned}
& \partial_tE^{\varepsilon}-  c\nabla_x \times B^{\varepsilon} =4\pi e_-\int_{\mathbb R^3} \hat{p}F^{\varepsilon} dp, \\
 &\partial_tB^{\varepsilon}+ c\nabla_x \times E^{\varepsilon}=0,\\
& \nabla_x\cdot E^{\varepsilon}=4\pi e_-\Big(\bar{n} -\int_{\mathbb R^3}  F^{\varepsilon} dp\Big), \\
& \nabla_x\cdot B^{\varepsilon}=0.
\end{aligned}
\end{align}
Here $\varepsilon$ is the  Knudsen
number (the mean free path), $F^{\varepsilon}= F^{\varepsilon}(t,x,p)$ is the number density function for electrons at time $t\geq0$, position $x=(x_1,x_2,x_3) \in \mathbb R^3$ and momentum $p=(p_1,p_2,p_3)\in \mathbb R^3$. $p^0=\sqrt{m^2c^2 + |p|^2}$ is the energy of an electron and $\hat{p}=\frac{p}{p^0} $. The constants $-e_-$ and $m$ are the magnitude of the electrons' charge and rest mass, respectively. $c$ is the speed of light, and $E(t,x), B(t,x)$ are the electromagnetic fields.

 Corresponding to \eqref{main1}-\eqref{main2}, at the
hydrodynamic level, the electron gas obeys the relativistic Euler-Maxwell system, which is an
important 1-fluid model for a plasma:
\begin{equation}
\begin{cases} \label{rem}
 \frac{1}{c}\partial_t(n u^0) + \nabla_x\cdot(n u) =0,\vspace{0.5ex}\\
 \frac{1}{c}\partial_t[(e+P)u^0u] + \nabla_x\cdot[(e+P)u\otimes u]+c^2\nabla_xP+ce_-n[u^0E+u\times B]=0, \vspace{0.5ex}\\
\partial_tE-c\nabla_x\times B =4\pi e_- \frac{nu}{c},\vspace{0.5ex}\\
\partial_tB+c\nabla_x\times E =0,\vspace{0.5ex}\\
\nabla_x\cdot E=4\pi e_-(\bar{n}-\frac{1}{c}nu^0),\vspace{0.5ex}\\
\nabla_x\cdot B=0,
\end{cases}
\end{equation}
where $n$ is the particle number density, $u=(u_1, u_2, u_3)$, $u^0=\sqrt{|u|^2+c^2} $, $P$ is the pressure,
$e$  is the total energy which is the sum of internal energy  and the energy in the rest frame.

 The purpose of
this article is to rigorously prove that solutions of the relativistic Vlasov-Maxwell-Boltzmann system
\eqref{main1}-\eqref{main2} converge to solutions of the relativistic Euler-Maxwell system \eqref{rem} globally in time, as the Knudsen
number $\varepsilon$ tends to zero:

\begin{equation}\label{limit}
\lim_{\varepsilon\rightarrow0}\sup_{0\leq t\leq \frac{1}{2}|\ln\varepsilon|^{\frac{1}{3}}}\left(\left\|(F^{\varepsilon}-\mathbf{M})(t)\right\|_{H^1}+\|(E^{\varepsilon}-E_0)(t)\|_{H^1}
+\|(B^{\varepsilon}-B_0)(t)\|_{H^1}\right) = 0.
\end{equation}
Namely, the solution $(F^{\varepsilon}, E^{\varepsilon}, B^{\varepsilon})$ to the relativistic Vlasov-Maxwell-Boltzmann system
converges to $(\mathbf{M}, E_0, B_0)$  in the $H^1$ norm. The macroscopic variables $n_0, u, T_0$ of the Maxwellian $\mathbf{M}$ \eqref{maxwell}, and $E_0, B_0$ satisfy the relativistic Euler-Maxwell system \eqref{rem}.

Our contribution is a step forward to derive two-fluid models for describing a plasma from kinetic theory. On the other hand, due to complexity of different scalings, such a derivation for a general two-fluid model with both ions and electrons remains a major open problem.

Now we briefly explain the strategy of our proof. Detailed explanations will be followed after the statement of Theorem \ref{result}.
In fact, the main part of our proof is constructing solutions to equations of remainder terms $F_R^{\varepsilon}, E_R^{\varepsilon}$ and $B_R^{\varepsilon}$ in the Hilbert expansion \eqref{expan} for the relativistic Vlasov-Maxwell-Boltzmann system in a $H^1-W^{1,\infty}$ framework. First, it is simple to get the $L^2$ estimate of the remainder terms. Under a priori assumptions \eqref{assump} for the $W^{1,\infty}$ estimates of the remainder terms, we can proceed characteristics estimates and further obtain the $L^{\infty}$ estimate of $F_R^{\varepsilon}$ to close the $L^2-L^{\infty}$ energy estimates. To justify the assumptions \eqref{assump}, the essential and delicate part is the  $W^{1,\infty}$ estimates of the electromagnetic fields $F_R^{\varepsilon}, E_R^{\varepsilon}$ via the Glassey-Strauss Representation. We succeed to bound the  $W^{1,\infty}$ estimates of the electromagnetic fields through the $W^{1,\infty}$  estimates of $F_R^{\varepsilon}$ properly. Then we modify the corresponding proof in \cite{Guo-Jang-CMP-2010} and bound the $W^{1,\infty}$ norm of $F_R^{\varepsilon}$ with its $H^1$ norm. Finally, we derive the $H^1$ norm estimates of the remainder terms to close the energy estimates and verify the assumptions \eqref{assump} for $t\in[0, \frac{1}{2}|\ln\varepsilon|^{\frac{1}{3}}]$.

The relativistic Boltzmann collision operator $Q(\cdot, \cdot)$ in (\ref{main1}) takes the form of
\begin{equation} \label{collision1}
 Q(F,G)=\frac{c}{p^0}\int_{\mathbb R^3}\frac{dq}{q^0}\int_{\mathbb R^3}\frac{dq'}{q'^{0}}\int_{\mathbb R^3}\frac{dp'}{p'^{0}}W[F(p')G(q')-F(p)G(q)].
  \end{equation}
Here the ``transition rate" $W=W(p,q|p',q')$ is defined as
\begin{equation} \label{transition}
W=s\sigma(g, \theta)\delta(p^0+q^0-p'^{0}-q'^{0}) \delta^{(3)}(p+q-p'-q').
  \end{equation}
  In a pioneering work of Glassey and Strauss \cite{Glassey-Strauss-PRIMS-1993}, the collision operator $Q$ in \eqref{collision1} was represented as follows:
\begin{equation*} \label{collision3}
\begin{split}
 Q(F,G)&=\int_{\mathbb R^3}dq\int_{\mathbb S^2}d\omega \frac{s\sigma(g,\theta)}{p^0q^0}\mathbb B(p,q,\omega)[F(p')G(q')-F(p)G(q)],
\end{split}
  \end{equation*}
 where the kernel $\mathbb B(p,q,\omega)$ is
\begin{equation*} \label{kernel}
\mathbb B(p,q,\omega)=\frac{(p^0+q^0)^2|\omega\cdot(p^0q-q^0p)|}{[(p^0+q^0)^2-(\omega\cdot(p+q))^2]^2}.
  \end{equation*}
Denote the four-momentums $p^{\mu} = (p^0, p)$ and $q^{\mu} = (q^0, q)$.  We use the Einstein convention that repeated up-down indices are
summed, and we raise and lower indices using the Minkowski metric $g_{\mu\nu}
\overset{\text{def}}{=}
diag(-1, 1, 1, 1)$. The Lorentz inner product is then given by
\begin{equation*}
p^{\mu}q_{\mu}\overset{\text{def}}{=}-p^0q^0 +
\sum_{i=1}^3 p_iq_i .
\end{equation*}
 The quantity $s$ in (\ref{transition}) is the square of the energy in the ``center of momentum", $p+q=0$, and is given as
$$s=s(p,q)=-(p^{\mu}+q^{\mu})(p_{\mu}+q_{\mu})=2(p^0q^0-p\cdot q+m^2c^2)\geq4m^2c^2.$$
And the relativistic momentum $g$ in (\ref{transition}) is denoted as
$$g=g(p,q)=\sqrt{-(p^{\mu}-q^{\mu})(p_{\mu}-q_{\mu})}=\sqrt{2(p^0q^0-p\cdot q-m^2c^2)}\geq0.$$

The condition for elastic collision is then given by
\begin{equation} \label{conservations}
p^0+q^0=p'^{0}+q'^{0},\qquad p+q=p'+q',
  \end{equation}
where $p'$ and $q'$ are the post collisional momentums given as
\begin{align*} \label{p-momentum2}
& p'=p+a(p,q,\omega)\omega,\qquad q'=q-a(p,q,\omega)\omega,\\
&a(p,q,\omega)=\frac{2(p^{0}+q^{0})[\omega\cdot(p^0q-q^0p)]}{(p^0+q^0)^2-(\omega\cdot(p+q))^2}.
  \end{align*}
 The Jacobian for the transformation
$(p, q) \rightarrow (p', q')$ in these variables \cite{Glassey-Strauss-TTSP-1991} is
\begin{equation*} \label{Jacobian1}
 \frac{\partial(p', q')}{\partial(p,q)}=-\frac{p'^{0} q'^{0}}{p^0q^0}.
  \end{equation*}
The relativistic differential cross section $\sigma(g,\theta)$ measures the interactions between particles. See \cite{Dudynske-Ekiel-Maria-CMP-1988,Dudynske-Ekiel-Maria-JTP-2007} for a physical discussion of general assumptions. We use the following hypothesis.\newline
{\it Hypothesis on the collision kernel.} We consider the ``hard ball" condition
\begin{equation*} \label{cross-section}
\sigma(g,\theta)=\mbox{constant}.
 \end{equation*}
This condition is used throughout the rest of the article. In fact,
without loss of generality,  we will use the normalized condition
$\sigma(g,\theta)=1$ for simplicity. The Newtonian limit, as $c\rightarrow\infty$, in this situation is the Newtonian
hard-sphere Boltzmann collision operator \cite{Strain-SIAM-2010}.

\subsection{Hilbert expansion}
We consider the Hilbert expansion for small Knudsen number $\varepsilon$,
\begin{align}\label{expan}
\begin{aligned}
&F^{\varepsilon}=\sum_{n=0}^{2k-1}\varepsilon^nF_n+\varepsilon^kF^{\varepsilon}_R,\\
&E^{\varepsilon}=\sum_{n=0}^{2k-1}\varepsilon^n E_n+\varepsilon^kE^{\varepsilon}_R,\\
&B^{\varepsilon}=\sum_{n=0}^{2k-1}\varepsilon^n B_n+\varepsilon^kB^{\varepsilon}_R,\qquad (k\geq 2).
\end{aligned}
\end{align}
To determine the coefficients $F_0(t, x, p), \ldots , F_{2k-1}(t, x, p); E_0(t, x)$, $ \ldots,$
$E_{2k-1}(t, x)$; $B_0(t, x), \ldots,$
$B_{2k-1}(t, x)$, we plug the formal expansion \eqref{expan} into the rescaled equations \eqref{main1}-\eqref{main2} to have

\begin{align}
& \partial_t \Big(\sum_{n=0}^{2k-1}\varepsilon^nF_n+\varepsilon^kF^{\varepsilon}_R\Big) + c\hat{p} \cdot \nabla_x \Big(\sum_{n=0}^{2k-1}\varepsilon^nF_n+\varepsilon^kF^{\varepsilon}_R\Big)\nonumber\\
&-e_-\Big[\Big(\sum_{n=0}^{2k-1}\varepsilon^n E_n+\varepsilon^kE^{\varepsilon}_R\Big)+\hat{p} \times \Big(\sum_{n=0}^{2k-1}\varepsilon^n B_n+\varepsilon^kB^{\varepsilon}_R\Big)\Big]\cdot\nabla_p \Big(\sum_{n=0}^{2k-1}\varepsilon^nF_n+\varepsilon^kF^{\varepsilon}_R\Big)\nonumber\\
 &= \frac{1}{\varepsilon}Q\Big(\sum_{n=0}^{2k-1}\varepsilon^nF_n+\varepsilon^kF^{\varepsilon}_R,\sum_{n=0}^{2k-1}\varepsilon^nF_n+\varepsilon^kF^{\varepsilon}_R\Big)),\nonumber\\
& \partial_t\Big(\sum_{n=0}^{2k-1}\varepsilon^n E_n+\varepsilon^kE^{\varepsilon}_R\Big)-  c\nabla_x \times \Big(\sum_{n=0}^{2k-1}\varepsilon^n B_n+\varepsilon^kB^{\varepsilon}_R\Big)\nonumber\\
 &=4\pi e_-\int_{\mathbb R^3} \hat{p}\Big(\sum_{n=0}^{2k-1}\varepsilon^nF_n+\varepsilon^kF^{\varepsilon}_R\Big) dp,\label{expan1}\\
 &\partial_t\Big(\sum_{n=0}^{2k-1}\varepsilon^n B_n+\varepsilon^kB^{\varepsilon}_R\Big)+ c\nabla_x \times \Big(\sum_{n=0}^{2k-1}\varepsilon^n E_n+\varepsilon^kE^{\varepsilon}_R\Big)=0,\nonumber\\
& \nabla_x\cdot \Big(\sum_{n=0}^{2k-1}\varepsilon^n E_n+\varepsilon^kE^{\varepsilon}_R\Big)=4\pi e_-\Big(\bar{n} -\int_{\mathbb R^3}  \Big(\sum_{n=0}^{2k-1}\varepsilon^nF_n+\varepsilon^kF^{\varepsilon}_R\Big) dp\Big), \nonumber\\
& \nabla_x\cdot \Big(\sum_{n=0}^{2k-1}\varepsilon^n B_n+\varepsilon^kB^{\varepsilon}_R\Big)=0.\nonumber
\end{align}
Now we equate the coefficients on both sides of equation \eqref{expan1} in front of different powers of
the parameter $\varepsilon$ to obtain:
\begin{align}
\frac{1}{\varepsilon}:& Q(F_0,F_0)=0,\nonumber\\
\varepsilon^0:&\partial_tF_0+c\hat{p}\cdot\nabla_x F_0-e_-\Big(E_0+\hat{p} \times B_0 \Big)\cdot\nabla_pF_0=Q(F_1,F_0)+Q(F_0,F_1),\nonumber\\
 & \partial_tE_0-  c\nabla_x \times B_0 =4\pi e_-\int_{\mathbb R^3} \hat{p}F_0 dp, \nonumber\\
 &\partial_tB_0+ c\nabla_x \times E_0=0,\nonumber\\
& \nabla_x\cdot E_0=4\pi e_-\Big(\bar{n} -\int_{\mathbb R^3}  F_0 dp\Big), \nonumber\\
& \nabla_x\cdot B_0=0,\nonumber\\
&\ldots\ldots\label{expan2}\\
\varepsilon^n:& \partial_tF_n+c\hat{p}\cdot \nabla_xF_n-e_-\Big(E_n+\hat{p} \times B_n \Big)\cdot\nabla_pF_0-e_-\Big(E_0+\hat{p} \times B_0 \Big)\cdot\nabla_pF_n\nonumber\\
 &=\sum_{\substack{i+j=n+1\\i,j\geq0}}Q(F_i,F_j)+e_-\sum_{\substack{i+j=n\\i,j\geq1}}\Big(E_i+\hat{p} \times B_i \Big)\cdot\nabla_pF_j, \nonumber\\
&\partial_tE_n-c\nabla_x \times B_n=4\pi e_-\int_{\mathbb R^3} \hat{p}F_n dp, \nonumber\\
 &\partial_t B_n+ c\nabla_x \times E_n=0,\nonumber\\
& \nabla_x\cdot E_n=-4\pi e_-\int_{\mathbb R^3}  F_n dp, \nonumber\\
& \nabla_x\cdot  B_n=0,\nonumber\\
&\ldots\ldots\nonumber\\
\varepsilon^{2k-1}:& \partial_tF_{2k-1}+c\hat{p}\cdot\nabla_x F_{2k-1}-e_-\Big(E_{2k-1}+\hat{p} \times B_{2k-1} \Big)\cdot\nabla_pF_0\nonumber\\
&-e_-\Big(E_0+\hat{p} \times B_0 \Big)\cdot\nabla_pF_{2k-1}\nonumber\\
 &=\sum_{\substack{i+j=2k\\i,j\geq1}}Q(F_i,F_j)+e_-\sum_{\substack{i+j=2k-1\\i,j\geq1}}\Big(E_i+\hat{p} \times B_i \Big)\cdot\nabla_pF_j, \nonumber\\
&\partial_tE_{2k-1}-c\nabla_x \times B_{2k-1}=4\pi e_-\int_{\mathbb R^3} \hat{p}F_{2k-1} dp, \nonumber\\
 &\partial_t B_{2k-1}+ c\nabla_x \times E_{2k-1}=0,\nonumber\\
& \nabla_x\cdot E_{2k-1}=-4\pi e_-\int_{\mathbb R^3}  F_{2k-1} dp, \nonumber\\
& \nabla_x\cdot  B_{2k-1}=0.\nonumber
\end{align}
The remainder terms $F_R^{\varepsilon}, E_R^{\varepsilon}$ and $B_R^{\varepsilon}$ satisfy the following equations:
\begin{align}
& \partial_tF_R^{\varepsilon}+c\hat{p}\cdot\nabla_x F_R^{\varepsilon}-e_-\Big(E_R^{\varepsilon}+\hat{p} \times B_R^{\varepsilon} \Big)\cdot\nabla_pF_0\nonumber\\
 &-e_-\Big(E_0+\hat{p} \times B_0 \Big)\cdot\nabla_pF_R^{\varepsilon}-\frac{1}{\varepsilon}[Q(F_R^{\varepsilon},F_0)+Q(F_0,F_R^{\varepsilon})]\nonumber\\
 &=\varepsilon^{k-1}Q(F_R^{\varepsilon},F_R^{\varepsilon})+\sum_{i=1}^{2k-1}\varepsilon^{i-1}[Q(F_i, F_R^{\varepsilon})+Q(F_R^{\varepsilon}, F_i)]+\varepsilon^ke_-\Big(E_R^{\varepsilon}+\hat{p} \times B_R^{\varepsilon}\Big)\cdot\nabla_pF_R^{\varepsilon}\nonumber\\
 &+\sum_{i=1}^{2k-1}\varepsilon^ie_-\Big[\Big(E_i+\hat{p} \times B_i \Big)\cdot\nabla_pF_R^{\varepsilon}+ \Big(E_R^{\varepsilon}+\hat{p} \times B_R^{\varepsilon} \Big)\cdot\frac{\nabla_pF_i}{\sqrt{\mathbf{M}}}\Big]+\varepsilon^{k}A,
\nonumber\\
&\partial_tE_R^{\varepsilon}-c\nabla_x \times B_R^{\varepsilon}=4\pi e_-\int_{\mathbb R^3} \hat{p}F_R^{\varepsilon} dp, \label{remain}\\
 &\partial_t B_R^{\varepsilon}+ c\nabla_x \times E_R^{\varepsilon}=0,\nonumber\\
& \nabla_x\cdot E_R^{\varepsilon}=-4\pi e_-\int_{\mathbb R^3}  F_R^{\varepsilon} dp, \nonumber\\
& \nabla_x\cdot  B_R^{\varepsilon}=0,\nonumber
\end{align}
where
\begin{align*}
\begin{aligned}
A=\sum_{\substack{i+j\geq 2k+1\\2\leq i,j\leq2k-1}}\varepsilon^{i+j-2k-1}Q(F_i,F_j)+\sum_{\substack{i+j\geq 2k\\1\leq i,j\leq2k-1}}\varepsilon^{i+j-2k}e_-\Big(E_i+\hat{p} \times B_i \Big)\cdot\nabla_pF_j.
\end{aligned}
\end{align*}
From the first equation in \eqref{expan2}, we can obtain that $F_0$ should be a local Maxwellian $\mathbf{M}=F_0$:
\begin{equation}\label{maxwell}
F_0(t,x,p)= \frac{n_0 \gamma}{4 \pi m^3
c^3 K_2(\gamma)} \exp\Big\{\frac{u^{\mu}p_{\mu} }{k_BT_0}\Big\},
\end{equation}
where
$K_j(\gamma) (j=0, 1, 2, \ldots)$ are the modified second order Bessel functions:
\begin{equation*}\label{defini}
	 K_j(\gamma)=\frac{(2^j)j!}{(2j)!}\frac{1}{\gamma^j}\int_{\lambda=\gamma}^{\lambda=\infty}e^{-\lambda}(\lambda^2-\gamma^2)^{j-1/2}d\lambda,
\quad(j\geq0),
	\end{equation*}
 and
 $\gamma$ is a dimensionless variable defined as
\begin{equation}\label{temp}
\gamma=\frac{mc^2}{k_BT_0},
\end{equation}
$k_B$ is Boltzmann's constant. Here $n_0(t,x), u^{\mu}(t,x)$ and $T_0(t,x)$ are the number density, four-velocity and temperature. As proved in Proposition $3.3$ of \cite{Speck-Strain-CMP-2011}, it holds that
\begin{align}\label{moments3}
\begin{split}
&  I^\alpha[\mathbf{M}] = c \int_{\mathbb R^3}  \frac{p^\alpha}{p^0}\mathbf{M}\,dp=n_0u^{\alpha},\\
& T^{\alpha \beta}[\mathbf{M}] =
c \int_{\mathbb R^3}     \frac{ p^\alpha
p^\beta }{p^0} \mathbf{M}\,   d p=  \frac{e_0+P_0}{c^2}u^{\alpha}u^{\beta}+P_0 g^{\alpha\beta},
\end{split}
\end{align}
where $e_0(t,x)$ is the total energy and $P_0(t,x)$ is the pressure satisfying
 \begin{align*}
&   P_0 =  \frac{m nc^2}{\gamma} = \frac{k_B}{m}\rho T  \, , \\
& e_0= \frac{n_0 m c^2}{  K_2(\gamma)}  \left[ K_3 \left(\gamma
\right) - \frac{1}{\gamma}  K_2
\left(\gamma   \right) \right].
\end{align*}
Noting the relationship  $p^0=\sqrt{m^2c^2 + |p|^2}$ and
\begin{align*}
&-\int_{\mathbb R^3}  p^0\Big(E_0+\hat{p} \times B_{0} \Big)\cdot\nabla_p\mathbf{M} dp\\
&\hspace{0.5cm} =\int_{\mathbb R^3}E_{0}\cdot \hat{p}\mathbf{M} dp=\frac{n_0u\cdot E_{0}}{c},
\end{align*}
 we project the second equation in \eqref{expan2} onto the five collision invariants, $1, p, p^0$, for the Boltzmann operator $Q$ and use \eqref{moments3} to have
\begin{equation}
\begin{cases} \label{order0}
 \frac{1}{c}\partial_t(n_0 u^0) + \nabla_x\cdot(n_0 u) =0,\vspace{0.5ex}\\
 \frac{1}{c}\partial_t[(e_0+P_0)u^0u] + \nabla_x\cdot[(e_0+P_0)u\otimes u]\vspace{0.5ex}\\
\hspace{2cm}+c^2\nabla_xP_0+ce_-n_0[u^0E_0+u\times B_0]=0, \vspace{0.5ex}\\
\frac{1}{c}\partial_t[(e_0+P_0)(u^0)^2-c^2|u|^2P_0] + \nabla_x\cdot[(e_0+P_0)u^0u]+ce_-n_0 u\cdot E_0=0.
\end{cases}
\end{equation}
Then, under the conditions
\begin{equation}\label{cdn}
e_0+P_0=n_0h(n_0),\qquad\mbox{and}\qquad P_0'(n_0)=n_0h'(n_0)>0,
\end{equation}
which were indicated in \cite{Guo-Ionescu-Pausader-JMP-2014}, we can further use the equations of $\varepsilon^0$ power in \eqref{expan2} to obtain equations of $n_0, u$ and $E_0, B_0$:
\begin{equation}
\begin{cases} \label{order01}
 \frac{1}{c}\partial_t(n_0 u^0) + \nabla_x\cdot(n_0 u) =0,\vspace{0.5ex}\\
  \frac{1}{c}\partial_t[(e_0+P_0)u^0u] + \nabla_x\cdot[(e_0+P_0)u\otimes u]\vspace{0.5ex}\\
\hspace{2cm}+c^2\nabla_xP_0+ce_-n_0[u^0E_0+u\times B_0]=0, \vspace{0.5ex}\\
\partial_tE_0-c\nabla_x\times B_0 =4\pi e_- \frac{n_0u}{c},\vspace{0.5ex}\\
\partial_tB_0+c\nabla_x\times E_0 =0,\vspace{0.5ex}\\
\nabla_x\cdot E_0=4\pi e_-(\bar{n}-\frac{1}{c}n_0u^0),\vspace{0.5ex}\\
\nabla_x\cdot B_0=0.
\end{cases}
\end{equation}
In fact, we use $\eqref{order0}_1$ to have
\begin{equation}\label{order02}
\begin{aligned}
&\displaystyle  \frac{n_0u^0}{c}\partial_t\Big(\frac{e_0+P_0}{n_0}u\Big) + n_0 u\cdot \nabla_x\Big(\frac{e_0+P_0}{n_0} u\Big)\\
&\hspace{2cm}+c^2\nabla_xP_0+ce_-n_0[u^0E_0+u\times B_0]=0, \\
&\displaystyle \frac{n_0u^0}{c}\partial_t\Big(\frac{e_0+P_0}{n_0}u^0\Big)-c\partial_tP_0 +n_0\cdot \nabla_x\cdot\Big(\frac{e_0+P_0}{n_0}u\Big)+ce_-n_0 u\cdot E_0=0.
\end{aligned}
\end{equation}
We further employ $u^0{\eqref{order02}}_2-u\cdot{\eqref{order02}}_1$ to obtain
\begin{equation}\label{order3}
u^0\Big[n_0\partial_t\Big(\frac{e_0+P_0}{n_0}\Big)-\partial_tP_0\Big]+u\cdot\Big[n_0\nabla_x\Big(\frac{e_0+P_0}{n_0}\Big)-\nabla_xP_0\Big]=0.
\end{equation}
\eqref{order3} automatically holds under the conditions \eqref{cdn}. Namely,  the third equation in \eqref{order0}, the energy equation, can be expressed by the first and second equations in \eqref{order0}, and \eqref{order01} holds true.

Here and below, we assume that $[n_0(t,x), u(t,x), E_0(t,x), B_0(t,x)]$ is a global smooth solution to the relativistic Euler-Maxwell 1-fluid system \eqref{order01} constructed in \cite{Guo-Ionescu-Pausader-JMP-2014} with $n_0(t,x)-\bar{n}$, $u(t,x)$, $E_0(t,x)$ and $B_0(t,x)$ sufficiently small and satisfying
\begin{align}\label{decay}
&\sup_{t\in[0,\infty]}\|[n_0(t)-\bar{n}, u(t), E_0(t), B_0(t)]\|_{H^{N_0}}+\sup_{t\in[0,\infty]}\Big[(1+t)^{\beta_0}\\
&\times(\sup_{|\rho|\leq \bar{N}}\|D_x^{\rho}(n_0(t)-\bar{n})\|_{\infty}+\sup_{|\rho|\leq4}
\|[D_x^{\rho}u(t), D_x^{\rho}E_0(t), D_x^{\rho}B_0(t)]\|_{\infty})\Big]\lesssim \bar{\varepsilon}_0 , \nonumber
\end{align}
where the constant $N_0$ is sufficiently large, $\bar{N}\geq3$, $\beta_0=101/100$ and $\bar{\varepsilon}_0$ is a small positive constant.

For $\mathbf{M}$ given in \eqref{maxwell}, where $(n_0, u, T_0)(t,x)$ is part of a solution to the relativistic Euler-Maxwell 1-fluid system \eqref{order01} constructed in \cite{Guo-Ionescu-Pausader-JMP-2014}, we define the linearized collision operator $Lf$ and nonlinear collision operator ${ \Gamma} ( f_1, f_2 )$ :
\begin{equation*}
\begin{aligned}
{L}f &=-\frac{1}{\sqrt{\mathbf{M}}}[Q( \sqrt{\mathbf{M}} f, \mathbf{M} )+Q( \mathbf{M},\sqrt{\mathbf{M}}f  )]=\nu(\mathbf{M})f-K_{\mathbf{M}}(f), \\
\Gamma ( f_1, f_2 ) &=\frac{1}{\sqrt{\mathbf{M}}}Q( \sqrt{\mathbf{M}} f_1, \sqrt{\mathbf{M}}f_2 ),
\end{aligned}
\end{equation*}
where the collision frequency $\nu=\nu(t, x, p)$ is defined as
\begin{equation*}
 \nu=\Gamma^{loss} ( 1, \sqrt{\mathbf{M}} )=\int_{\mathbb R^3}dq\int_{\mathbb S^2}d\omega \frac{s}{p^0q^0}\mathbb B(p,q,\omega)\sqrt{\mathbf{M}}(q),
  \end{equation*}
and $K_{\mathbf{M}}(f)$ takes the following form:
\begin{align}\label{Koper}
  K_{\mathbf{M}}(f)=&\int_{\mathbb{R}^3}dq\int_{\mathbb S^2}d\omega \frac{s}{p^0q^0}\mathbb B(p,q,\omega)\sqrt{\mathbf{M}}(q)[\sqrt{\mathbf{M}}(q')f(p')+\sqrt{\mathbf{M}}(p')f(q')]\nonumber\\
  &-\int_{\mathbb{R}^3}dq\int_{\mathbb S^2}d\omega \frac{s}{p^0q^0}\mathbb B(p,q,\omega)\sqrt{\mathbf{M}}(q)\sqrt{\mathbf{M}}(p)f(q)\\
  =&K_2(f)-K_1(f).\nonumber
\end{align}
 Then from Lemma $3.1$ in \cite{Strain-CMP-2010} and Lemma 6 in \cite{Guo-Strain-CMP-2012}, we can similarly obtain
 \begin{equation*}\label{frequ}
 \nu(t, x, p)\approx 1,\qquad\mbox{and}\qquad |D_p^{\rho}\nu(t, x, p)|\lesssim (p^0)^{-1},\quad |\rho|>0.
 \end{equation*}

Note that the null space of the linearized operator $L$ is given by
\[\mathcal {N}=\mbox{span}\left\{\sqrt{\mathbf{M}}, p_i\sqrt{\mathbf{M}}(1\leq i\leq3), p^{0}\sqrt{\mathbf{M}}\right\}.
\]
Let $\bf P$ be the orthogonal projection from $L^2_p
$ onto $\mathcal {N}$. Given $f(t, x, p)$, one can
express ${\bf P}f$ as a linear combination of the basis in $\mathcal{N}$:
$${\bf P}f = \{a_f(t, x) + b_f(t, x) \cdot p + c_f(t, x)p^{0}\} \sqrt{\mathbf{M}}.$$
Then we have \cite{Glassey-Strauss-PRIMS-1993}
$$\langle Lf,f\rangle\geq \delta_0\|{\{\bf I-P\}}f\|^2,$$
for some constant $\delta_0>0$. Define $f^{\varepsilon}$ as
\begin{equation}\label{L2}
F^{\varepsilon}_R= \sqrt{\mathbf{M}}f^{\varepsilon}.
\end{equation}
We further introduce a global Maxwellian
$$J_{M}= \frac{n_M \gamma_M}{4 \pi m^3
c^3 K_2(\gamma_M)} \exp\{- \frac{ p^0 }{k_BT_M}  \},$$
and define
\begin{equation}\label{Linfty}
F^{\varepsilon}_R= (1+|p|)^{-\beta}\sqrt{J_{M}}h^{\varepsilon}=\frac{\sqrt{J_{M}}}{w(|p|)}h^{\varepsilon},
\end{equation}
with $w(|p|)\equiv(1+|p|)^{\beta}$ for some $\beta\geq 8$.
Here $n_M, T_M=mc^2/(k_B \gamma_M)$ satisfy the conditions in \eqref{cdn} and
\begin{equation}\label{temp}
T_M<\sup_{t\in [0,\frac{1}{2}|\ln\varepsilon|^{\frac{1}{3}}],x\in \mathbb R^3} T_0(t,x)< 2T_M.
\end{equation}

\begin{remark}
Since the presence of physical constants do not cause essential
mathematical difficulties,  we will normalize all constants in the relativistic Vlasov-Maxwell-Boltzmann system \eqref{main1}, \eqref{main2} and in all related quantities to be one.
\end{remark}

\subsection{Notations} Throughout the paper, $C$ denotes a generic positive constant which may change line by line. The notation $A \lesssim B$ implies that there exists a positive constant $C$  such
that $A \leq CB$ holds uniformly over the range of parameters. The
notation $ A \approx B$ means $\frac{1}{\bar{C}}A\leq B\leq \bar{C}A$ for some constant $\bar{C}>1$. We use
standard notations to denote the Sobolev spaces $W^{ k,2}(\mathbb R^3_x)$ (or $W^{ k,2}(\mathbb R^3_x\times\mathbb R^3_p)$) and $W^{ k,\infty}(\mathbb R^3_x)$ (or $W^{ k, \infty}(\mathbb R^3_x\times\mathbb R^3_p)$) with corresponding
norms $\|\cdot\|_{H^k}$ and  $\|\cdot\|_{W^{k,\infty}}$, respectively. We also use the standard notations $\|\cdot\|$ and $\|\cdot\|_{\infty}$ to denote the
$L^2$ norm and $L^{\infty}$ norm in both $(x, p)\in \mathbb R^3\times \mathbb R^3$ and $x\in \mathbb R^3$, respectively.  The standard $L^2(\mathbb R^3\times \mathbb R^3)$ inner product is denoted as $\langle \cdot, \cdot \rangle $ for simplicity. $D_x$ and $D_p$ are used as any space and momentum derivative, respectively.

\subsection{Main results} We now state our main results.
\begin{theorem}\label{result} Let $F_0 =\mathbf{M}$  as in \eqref{maxwell} and let $[n_0(0, x), u(0, x), E_0(0,x), B_0(0,x)]$ satisfies the same assumptions in Theorem $2.2$ of \cite{Guo-Ionescu-Pausader-JMP-2014}
and $[n_0(t,x), u(t,x), E_0(t,x)$, $B_0(t,x)]$ be a corresponding global solution. Then for the
remainder terms $F^{\varepsilon}_R$,$E^{\varepsilon}_R$ and $B^{\varepsilon}_R$
in \eqref{expan}, there exists an $\varepsilon_0 > 0$  such
that for $0 \leq \varepsilon\leq\varepsilon_0$,
\begin{eqnarray}
&&\sup_{0\leq t\leq \frac{1}{2}|\ln\varepsilon|^{\frac{1}{3}}}
\Big(\varepsilon^{\frac{7}{4}}\Big\|\frac{(1+|p|)^{\beta}F^{\varepsilon}_R}{\sqrt{J_M}}(t)\Big\|_{\infty}
+\varepsilon^{3}\Big\|\frac{(1+|p|)^{\beta}F^{\varepsilon}_R}{\sqrt{J_M}}(t)\Big\|_{W^{1,\infty}}\Big)\nonumber\\
&&\hspace{0.5cm}+\sup_{ 0\leq t\leq \frac{1}{2}|\ln\varepsilon|^{\frac{1}{3}}}
\Big(\varepsilon^{\frac{11}{4}}\|[E^{\varepsilon}_R(t),B^{\varepsilon}_R(t)]\|_{\infty}+\varepsilon^{4}
\|[E^{\varepsilon}_R(t),B^{\varepsilon}_R(t)]\|_{W^{1,\infty}}\Big)
\nonumber\\
&&\hspace{0.5cm}
+\sup_{ 0\leq t\leq \frac{1}{2}|\ln\varepsilon|^{\frac{1}{3}}}\Big[\Big\|\frac{F^{\varepsilon}_R}{\sqrt{\mathbf{M}}}(t)\Big\|
+\|[E^{\varepsilon}_R(t),B^{\varepsilon}_R(t)]\|\nonumber\\
&&\hspace{0.5cm}+\sqrt{\varepsilon}\Big(\Big\|\frac{F^{\varepsilon}_R}{\sqrt{\mathbf{M}}}(t)\Big\|_{H^1}
+\|[E^{\varepsilon}_R(t),B^{\varepsilon}_R(t)]\|_{H^1}\Big)\Big]\label{theo1}\\
&&\lesssim\varepsilon^{\frac{3}{2}}\Big\|\frac{(1+|p|)^{\beta}F^{\varepsilon}_R}{\sqrt{J_M}}(0)\Big\|_{\infty}
+\varepsilon^{\frac{3}{2}}\Big\|\frac{(1+|p|)^{\beta}F^{\varepsilon}_R}{\sqrt{J_M}}(0)\Big\|_{W^{1,\infty}}
+\varepsilon^{^{\frac{3}{2}}}\|[E^{\varepsilon}_R(0),B^{\varepsilon}_R(0)]\|_{\infty}
\nonumber\\
&&\hspace{0.5cm}+\varepsilon^{^{\frac{3}{2}}}\|[E^{\varepsilon}_R(0),B^{\varepsilon}_R(0)]\|_{W^{1,\infty}}
+\Big\|\frac{F^{\varepsilon}_R}{\sqrt{\mathbf{M}}}(0)\Big\|+\|[E^{\varepsilon}_R(0),B^{\varepsilon}_R(0)]\|\nonumber\\
&&\hspace{0.5cm}+\sqrt{\varepsilon}\Big(\Big\|\frac{F^{\varepsilon}_R}{\sqrt{\mathbf{M}}}(0)\Big\|_{H^1}
+\|[E^{\varepsilon}_R(0),B^{\varepsilon}_R(0)]\|_{H^1}\Big)+1.\nonumber
\end{eqnarray}
\end{theorem}

\begin{remark}
 We require $ k \geq 5$ for the expansion in \eqref{expan}. This requirement is  more restricted than the case of the Hilbert expansion for the Boltzmann equation \cite{Guo-JJ-KRM-2009, Guo-JJ-CPAM-2010, Speck-Strain-CMP-2011} and more relaxed than that for the Vlasov-Poisson-Boltzmann system \cite{Guo-Jang-CMP-2010}. Moreover, our uniform estimates lead to the relativistic Euler-Maxwell limit \eqref{limit}. Here we point out that the estimate \eqref{theo1} is not accurate and is written in a form for simplicity.  Explicit estimates corresponding to \eqref{theo1} can be found in Section \ref{mr}.

\end{remark}

Our result guarantees that the Hilbert expansion for the relativistic Vlasov-Maxwell-Boltzmann system is valid for any time if $\varepsilon$ is chosen sufficiently small. Similar result was also obtained in \cite{Guo-Jang-CMP-2010} for the Hilbert expansion of the Vlasov-Poisson-Boltzmann system.
 Due to shock formations in the pure compressible
Euler flow, as illustrated in \cite{Sideris-CMP-1985,Christodoulou-2007}, corresponding Hilbert expansion, acoustic limit  for the Boltzmann equation or relativistic Boltzmann equation \cite{Caflisch-CPAM-1980, Guo-JJ-KRM-2009, Guo-JJ-CPAM-2010, Speck-Strain-CMP-2011} is only valid local in time. Different from the pure compressible Euler fluids where shock waves may develop even for smooth irrotational initial data with small amplitude, the electromagnetic interaction in the two-fluid models \cite{Guo-CMP-1998,Guo-Ionescu-Pausader-JMP-2014,Guo-Ionescu-Pausader-Annals-2016} could create stronger dispersive effects, enhance linear decay rates, and prevent formation of shock waves with small amplitude.

Our method is based on an $L^2-L^{\infty}$ framework originated in \cite{Guo-ARMA-2010}. In Section \ref{l2}, we first establish the basic $L^2$ estimate for remainders $(f^{\varepsilon}, E^{\varepsilon}_ R, B^{\varepsilon}_ R)$.
 In order to close the energy estimate \eqref{L2ener0} in Proposition \ref{L2ener}, we need $L^{\infty}$ estimate of $h^{\varepsilon}$ along the curved characteristic given by
\begin{align}\label{chara1}
\frac{d X(\tau; t, x, p)}{d\tau}&=\hat{P}(\tau; t, x, p),\nonumber\\
\frac{d P(\tau; t, x, p)}{d\tau}&=-E^{\varepsilon}(\tau, X(\tau; t, x, p))-\hat{P}(\tau; t, x, p)\times B^{\varepsilon}(\tau, X(\tau; t, x, p)),
\end{align}
with  $X(t; t, x, p)=x$, $P(t; t, x, p)=p$.

In Section \ref{cl}, we study the characteristics and $L^{\infty}$ estimate of $h^{\varepsilon}$ under the crucial bootstrap assumptions:
\begin{align}\label{assump}
  &\sup_{t\in[0,T]}\varepsilon^{2}\|h^{\varepsilon}(t)\|_{\infty}\leq \varepsilon^{\frac{1}{4}},\nonumber\\
 &\sup_{t\in[0,T]}\varepsilon^{3}\Big(\|\nabla_xh^{\varepsilon}(t)\|_{\infty}+\|\nabla_ph^{\varepsilon}(t)\|_{\infty}\Big)\leq \varepsilon^{\frac{1}{8}},\\
 &\sup_{t\in[0,T]}\varepsilon^{4}\|[E^{\varepsilon}_ R(t), B^{\varepsilon}_ R(t)]\|_{W^{1,\infty}}\leq \varepsilon^{\frac{1}{8}},\nonumber
  \end{align}
 for given $T\in [0,\frac{1}{2}|\ln\varepsilon|^{\frac{1}{3}}]$.
We can modify the $L^{\infty}$ estimate of $h^{\varepsilon}$ in Subsection 5.1 of \cite{Guo-Jang-CMP-2010} and obtain Proposition
\ref{linfty}. Combing Proposition \ref{L2ener},   Proposition \ref{linfty} and the $L^{\infty}$ estimate of the electromagnetic field, we finally close the $L^2-L^{\infty}$ energy estimates. The rest of the paper is   mainly devoted to the proof of \eqref{assump}.

In Section \ref{w1eb}, we estimate the $W^{1,\infty}$ norm of $E^{\varepsilon}_R, B^{\varepsilon}_R$ in term of the  $W^{1,\infty}$ norm of $f^{\varepsilon}$, via Glassey-Strauss Representation (see Theorem 3 and Theorem 4 in \cite{Glassey-Strauss-ARMA-1986}). In the previous work \cite{Guo-Jang-CMP-2010}, the electric field $E^{\varepsilon}_R$ was calculated via elliptic estimate of the Poisson system:
$$E^{\varepsilon}_R=-\nabla_x\phi^{\varepsilon}_R,\qquad \triangle\phi^{\varepsilon}_R=\rho^{\varepsilon}_R,\quad \rho^{\varepsilon}_R\in L^{\infty}.$$
It is important to note that the relativistic nature of bounded velocity $\hat{p}$ and the strong weight $(1+|p|)^k$ play key roles. We remark that the $W^{1,\infty}$ estimate \eqref{TWEB1} in Proposition \ref{TWEB} fails for the classical Vlasov-Maxwell-Boltzmann system with unbounded velocity. Our method relies crucially on the Glassey-Strauss representation, which fails for the non-relativistic case. Once the relativistic 1-fluid model is derived, it is then possible to study the classical (non-relativistic limit) as the speed of light goes to $\infty$. It is possible to extend our result to the 2D case, with a similar Glassey-Strauss formulation in $2D$ \cite{Glassey-Schaeffer-CMP-1997}.

In Section \ref{w1h}, we establish the $W^{1,\infty}$ estimate for $h^{\varepsilon}$.
In \cite{Guo-Jang-CMP-2010}, such an estimate relies on a second integration by parts
\begin{eqnarray}
&&D_xh^{\varepsilon}(t,x,v)\sim\ldots+\frac{1}{\varepsilon^2}
\int_0^t\int_0^s\exp\Big\{-\frac{1}{\varepsilon}\int_s^t\nu(\tau)d\tau-\frac{1}{\varepsilon}\int_{s_1}^s\nu(\tau)d\tau\Big\}\nonumber\\
&&\hspace{0.5cm}\times\int_{\mathbb{R}^3\times\mathbb{R}^3}l_{M,w}(V(s),v')l_{M,w}(V(s_1),v'')D_xh^{\varepsilon}(s_1, X(s_1), v'')|\frac{dv'}{dy}| dv''yds_1ds+\ldots\label{fe}\\
&&\hspace{0.5cm}\sim\ldots-\frac{1}{\varepsilon^2}
\int_0^t\int_0^{s-\kappa\varepsilon}\exp\Big\{-\frac{1}{\varepsilon}\int_s^t\nu(\tau)d\tau-\frac{1}{\varepsilon}\int_{s_1}^s\nu(\tau)d\tau\Big\}\nonumber\\
&&\hspace{0.5cm}\times\int_{\hat{B}}l_{N}(V(s),v')l_{N}(V(s_1),v'')h^{\varepsilon}(s_1, X(s_1), v'')D_x(|\frac{dv'}{dy}|)dydv''ds_1ds+\ldots,\nonumber
\end{eqnarray}
 where $l_{M,w}$ is the kernel corresponding to $K_{M,w}$ in \cite{Guo-Jang-CMP-2010}, $l_N$ is a smooth approximation of $l_{M,w}$ with compact support, $y=X(s_1)$, and $\hat{B}=\{|y-X(s)|\leq C(s-s_1)N, |v''|\leq 3N\}$.
 Thanks to the elliptic regularity of the Poisson system, the electric field $E^{\varepsilon}=-\nabla_x\phi^{\varepsilon}$ is bounded almost in $W^{2,\infty}$ with the auxiliary $W^{1,\infty}$ bound for the distribution function, and $\|D_x(|\frac{dv'}{dy}|)\|_{L^2_{\hat{B}}}$ can also be bounded by $C_N/(\varepsilon^4)$. Unfortunately, in the presence of the magnetic field, the Maxwell system is hyperbolic, and such a gain of regularity and bound of $\|D_x(|\frac{dv'}{dy}|)\|_{L^2_{\hat{B}}}$ are impossible. Instead, we estimate the first expression in \eqref{fe}
    by the $H^1$ norm of $f^{\varepsilon}$ and $W^{1,\infty}$ norm of $(E^{\varepsilon}, B^{\varepsilon})$.

     We turn to an auxiliary $H^1$ estimate in Section \ref{h1f}.
  Even though such an $H^1$ estimate in Proposition \ref{TH1} leads to a loss of $\varepsilon^{-\frac{1}{2}}$, it is still sufficient for the $W^{1,\infty}$ estimate of $h^{\varepsilon}$ in the crucial bootstrap assumptions \eqref{assump}.
Noting the fact
 $$\|\nabla_pf^{\varepsilon}\|\lesssim \|\nabla_p{\bf P}f^{\varepsilon}\|+\|\nabla_p{\{\bf I-\bf P\}}f^{\varepsilon}\|,$$
 we estimate the norm $\|\nabla_p{\{\bf I-\bf P\}}f^{\varepsilon}\|$ instead of a direct estimation of $\|\nabla_pf^{\varepsilon}\|$ to  avoid estimates, such as the estimate related to the transport term
 \begin{equation}\label{transp}
 \langle (D_p\hat{p})\cdot \nabla_x f^{\varepsilon}, \nabla_pf^{\varepsilon}\rangle\lesssim \|D_pf^{\varepsilon}\|\|D_xf^{\varepsilon}\|,
 \end{equation}
which may lead to exponential growth of the $H^1$ norm of remainders. In fact, we proceed energy estimate after employing momentum differentiation $D_p$ to the equation of ${\{\bf I-\bf P\}}f^{\varepsilon}$ via micro projection onto the equation of $f^{\varepsilon}$. Thanks to the dissipation of $\frac{1}{\varepsilon}\nabla_p(Lf^{\varepsilon})$, instead of \eqref{transp}, we can obtain
 $$\langle (D_p\hat{p})\cdot \nabla_x f^{\varepsilon}, \nabla_p{\{\bf I-\bf P\}}f^{\varepsilon}\rangle\lesssim \frac{\kappa}{\varepsilon}\|D_p{\{\bf I-\bf P\}}f^{\varepsilon}\|^2+\frac{\varepsilon}{\kappa}\|D_xf^{\varepsilon}\|^2$$
for some small positive constant $\kappa>0$. Then $(\kappa/\varepsilon)\|D_p{\{\bf I-\bf P\}}f^{\varepsilon}\|^2$ can be absorbed by the dissipation term and  the factor $\varepsilon$ of $(\varepsilon/\kappa)\|D_xf^{\varepsilon}\|^2$ can kill the increase of time for $t\in [0,\frac{1}{2}|\ln\varepsilon|^{\frac{1}{3}}]$.

In Section \ref{mr}, we finally verify \eqref{assump} and close the energy estimates via a continuity argument for
$ 0\leq t\leq \frac{1}{2}|\ln\varepsilon|^{\frac{1}{3}}$. The proof of our main result Theorem \ref{result} is given.

Appendix 3 is devoted to construction of coefficients $F_{n}(t,x,p), E_{n}(t,x)$, $B_{n}(t,x), (1\leq n\leq 2k-1)$ in the Hilbert expansion \eqref{expan} and estimation of their regularities (see Theorem \ref{fn}). Due to the complex nature of the relativistic Boltzmann equation, the construction and estimates are delicate.

We remark that the time decay rate $(1+t)^{-\beta_0} (\beta_0>1)$ of  the solutions \cite{Guo-Ionescu-Pausader-JMP-2014}  $n_0(t,x)$, $u(t,x)$, $E_0(t,x)$ and $B_0(t,x)$ in \eqref{decay}  is important. As in \eqref{L200}, \eqref{Hx1000} and \eqref{Hx1001}, thanks to the time decay of $\|\nabla_xn_0\|_{\infty}$, $\|\nabla_xu\|_{\infty}$, $\|\nabla_xE_0\|_{\infty}$ and $\|\nabla_xB_0\|_{\infty}$, we obtain estimates such as $(1+t)^{-\beta_0}\|f^{\varepsilon}\|^2$, which is bound by Gr\"{o}nwall's inequality.


We briefly review works related to the relativistic Vlasov-Maxwell-Boltzmann system. For the two species relativistic Vlasov-Maxwell-Boltzmann system with $\varepsilon = 1$, global smooth solutions near a Maxwellian
was constructed \cite{Guo-Strain-CMP-2012} in a periodic box with new momentum regularity estimates. This result was extended to the whole space in \cite{Ma-Ma-APC-2015} and \cite{Liu-Xiao-KRM-2016} for short range interaction. There are more studies about the non-relativistic  Vlasov-Maxwell-Boltzmann system. In \cite{Guo-Invent-2003}, global classical solutions near a Maxwellian
were constructed in a periodic box. Since then, many works have been done in global existence and asymptotic stability of solutions, and hydrodynamic limits, such as \cite{Strain-CMP-2006, Jang-ARMA-2009,Duan-Strain-CPAM-2011, Arsenio-SaintR-CRMASP-2013, DLYZ-KRM-2013, HXXZ-JMP-2013, Lei-Zhao-JDE-2016, DLYZ-CMP-2017,Arsenio-SaintR-EMSMM-2019}.

\section{$L^2$ estimate for $f^{\varepsilon}$}\label{l2}

\setcounter{equation}{0}

In this section, we derive the $L^2$ energy estimates for the remainder $f^{\varepsilon}= \frac{F^{\varepsilon}_ R}{\sqrt{\mathbf{M}}}$
.

To perform the $L^2$ energy estimate, we first use \eqref{L2} to rewrite \eqref{remain} as
\begin{align}
& \partial_tf^{\varepsilon}+\hat{p}\cdot\nabla_xf^{\varepsilon}+\frac{u^0 }{2T_0}\hat{p}\sqrt{\mathbf{M}}\cdot E_R^{\varepsilon}-\frac{u \sqrt{\mathbf{M}}}{2 T_0}\cdot\Big(\hat{p} \times B_R^{\varepsilon} \Big)\nonumber\\
&-\Big(E_0+\hat{p} \times B_0 \Big)\cdot\nabla_pf^{\varepsilon}+\frac{Lf^{\varepsilon}}{\varepsilon}\nonumber\\
 &=-\frac{f^{\varepsilon}}{\sqrt{\mathbf{M}}}\Big[\partial_t+\hat{p}\cdot\nabla_x-\Big(E_0+\hat{p} \times B_0 \Big)\cdot\nabla_p\Big]\sqrt{\mathbf{M}}+\varepsilon^{k-1}\Gamma(f^{\varepsilon},f^{\varepsilon})\nonumber\\
 &+\sum_{i=1}^{2k-1}\varepsilon^{i-1}\big[\Gamma( \frac{F_i}{\sqrt{\mathbf{M}}}, f^{\varepsilon})
 +\Gamma(f^{\varepsilon},  \frac{F_i}{\sqrt{\mathbf{M}}})\big]+\varepsilon^k\Big(E_R^{\varepsilon}+\hat{p} \times B_R^{\varepsilon}\Big)\cdot\nabla_pf^{\varepsilon}\nonumber\\
 &-\varepsilon^k\frac{1 }{2 T_0}\Big(u^0\hat{p}-u\Big)\cdot\Big(E_R^{\varepsilon}+\hat{p} \times B_R^{\varepsilon}\Big)f^{\varepsilon}\label{L20}\\
 &+\sum_{i=1}^{2k-1}\varepsilon^i\Big[\Big(E_i+\hat{p} \times B_i \Big)\cdot\nabla_pf^{\varepsilon}
 + \Big(E_R^{\varepsilon}+\hat{p} \times B_R^{\varepsilon} \Big)\cdot\frac{\nabla_pF_i}{\sqrt{\mathbf{M}}}\Big]\nonumber\\
 &-\sum_{i=1}^{2k-1}\varepsilon^i\Big[\Big(E_i+\hat{p} \times B_i \Big)\cdot\frac{1}{2 T_0}\Big(u^0\hat{p}-u\Big)f^{\varepsilon}\Big]+\varepsilon^{k}\bar{A},
\nonumber
\end{align}
and
\begin{align}
&\partial_tE_R^{\varepsilon}-\nabla_x \times B_R^{\varepsilon}=\int_{\mathbb R^3} \hat{p}\sqrt{\mathbf{M}}f^{\varepsilon} dp, \nonumber\\
 &\partial_t B_R^{\varepsilon}+ \nabla_x \times E_R^{\varepsilon}=0,\label{L201}\\
& \nabla_x\cdot E_R^{\varepsilon}=-\int_{\mathbb R^3}  F_R^{\varepsilon} dp, \qquad \nabla_x\cdot  B_R^{\varepsilon}=0,\nonumber
\end{align}
where $\bar{A}=\frac{A}{\sqrt{\mathbf{M}}}$.

\begin{proposition}\label{L2ener}
For the remainders $(f^{\varepsilon}, E^{\varepsilon}_ R, B^{\varepsilon}_ R)$, it holds that
\begin{align}
&\frac{d}{dt}\Big(\Big\|\sqrt{\frac{2T_0}{u^0}}f^{\varepsilon}(t)\Big\|^2+\|[E^{\varepsilon}_ R(t), B^{\varepsilon}_ R(t)]\|^2\Big)\nonumber\\
&\hspace{0.5cm}+\Big(\frac{\delta_0T_M}{2\varepsilon}-C\varepsilon^{k-2}\|h^{\varepsilon}\|_{\infty}\Big)\|{\{\bf I-P\}}f^{\varepsilon}\|^2\nonumber\\
&\hspace{0.5cm}\lesssim \left[(1+t)^{-\beta_0}+\varepsilon^{k}\|h^{\varepsilon}\|_{\infty}+\varepsilon  {\mathcal I}_1(t)\right]
\left(\|f^{\varepsilon}\|^2+\|(E^{\varepsilon}_ R, B^{\varepsilon}_ R)\|^2\right)\label{L2ener0}\\
& \hspace{0.5cm}+\left[\varepsilon^{\frac{11}{4}}\|h^{\varepsilon}\|_{\infty}
+\varepsilon^{k}{\mathcal I}_2(t)\right]\|f^{\varepsilon}\|,\nonumber
\end{align}
where ${\mathcal I}_1$ and ${\mathcal I}_2$ have the following form:
\begin{align*}
{\mathcal I}_1(t)&=\sum_{i=1}^{2k-1}\varepsilon^{i-1} (1+t)^{i}\Big[1+\sum_{i=1}^{2k-1}\varepsilon^{i-1} (1+t)^{i}\Big],\\
{\mathcal I}_2(t)&=\sum_{1+2k\leq i+j\leq 4k-2}\varepsilon^{i+j-2k-1} (1+t)^{i+j}.
\end{align*}
\end{proposition}
\begin{proof}
We take the $L^2$ inner product with $\frac{2T_0}{u^0}f^{\varepsilon}$ on both sides of \eqref{L20} to have
\begin{align}
& \frac{1}{2}\frac{d}{dt}\Big\|\sqrt{\frac{2T_0}{u^0}}f^{\varepsilon}(t)\Big\|^2+\Big\langle\hat{p}\cdot E_R^{\varepsilon}\sqrt{\mathbf{M}}, f^{\varepsilon}\Big\rangle
+\frac{2\delta_0}{\varepsilon}\Big\|\sqrt{\frac{T_0}{u^0}}{\{\bf I-P\}}f^{\varepsilon}\Big\|^2\nonumber\\
&\leq-\Big\langle\frac{f^{\varepsilon}}{\sqrt{\mathbf{M}}}\Big[\partial_t+\hat{p}\cdot\nabla_x-\Big(E_0+\hat{p} \times B_0 \Big)\cdot\nabla_p\Big]\sqrt{\mathbf{M}},\frac{2T_0}{u^0}f^{\varepsilon}\Big\rangle\nonumber\\
&+\frac{1}{2}\Big\langle\Big[\Big(\partial_t+\hat{p}\cdot\nabla_x\Big)\Big(\frac{2T_0}{u^0}\Big)\Big]f^{\varepsilon}, f^{\varepsilon}\Big\rangle\nonumber\\
 &+ \Big\langle u \cdot\Big(\hat{p} \times B_R^{\varepsilon} \Big)\sqrt{\mathbf{M}},f^{\varepsilon}\Big\rangle+\varepsilon^{k-1}\Big\langle\Gamma(f^{\varepsilon},f^{\varepsilon}),\frac{2T_0}{u^0}f^{\varepsilon}\Big\rangle\nonumber\\
 &+\sum_{i=1}^{2k-1}\varepsilon^{i-1}\Big\langle\big[\Gamma( \frac{F_i}{\sqrt{\mathbf{M}}}, f^{\varepsilon})+\Gamma(f^{\varepsilon}, \frac{F_i}{\sqrt{\mathbf{M}}})\big],\frac{2T_0}{u^0}f^{\varepsilon}\Big\rangle\nonumber\\
 &-\varepsilon^k\Big\langle \Big(u^0\hat{p}-u\Big)\cdot\Big(E_R^{\varepsilon}+\hat{p} \times B_R^{\varepsilon}\Big)f^{\varepsilon},f^{\varepsilon}\Big\rangle\label{remain0}\\
 &+\sum_{i=1}^{2k-1}\varepsilon^i\Big\langle \Big[\Big(E_i+\hat{p} \times B_i \Big)\cdot\nabla_pf^{\varepsilon}+ \Big(E_R^{\varepsilon}+\hat{p} \times B_R^{\varepsilon} \Big)\cdot\frac{\nabla_pF_i}{\sqrt{\mathbf{M}}}\Big],\frac{2T_0}{u^0}f^{\varepsilon}\Big\rangle\nonumber\\
 &-\sum_{i=1}^{2k-1}\varepsilon^i\Big\langle \Big[\Big(E_i+\hat{p} \times B_i \Big)\cdot\frac{1 }{ T_0}\Big(u^0\hat{p}-u\Big)f^{\varepsilon}\Big],\frac{2T_0}{u^0}f^{\varepsilon}\Big\rangle+\varepsilon^{k}\Big\langle\bar{A},\frac{2T_0}{u^0}f^{\varepsilon}\Big\rangle.
\nonumber
\end{align}
On the other hand, from \eqref{L201}, we have
\begin{align}\label{L2max}
\frac{1}{2}\frac{d}{dt}\left(\|E_R^{\varepsilon}(t)\|^2+\|B_R^{\varepsilon}(t)\|^2\right)=\Big\langle\hat{p}\cdot E_R^{\varepsilon}\sqrt{\mathbf{M}}, f^{\varepsilon}\Big\rangle.
\end{align}
Noting \eqref{temp}, we combine \eqref{remain0} and \eqref{L2max} to have
\begin{align}
& \frac{1}{2}\frac{d}{dt}\Big(\Big\|\sqrt{\frac{2T_0}{u^0}}f^{\varepsilon}(t)\Big\|^2+\|[E^{\varepsilon}_ R(t), B^{\varepsilon}_ R(t)]\|^2\Big)+\frac{\delta_0T_M}{\varepsilon}\|{\{\bf I-P\}}f^{\varepsilon}\|^2\nonumber\\
&\leq-\Big\langle\frac{f^{\varepsilon}}{\sqrt{\mathbf{M}}}\Big[\partial_t+\hat{p}\cdot\nabla_x-\Big(E_0+\hat{p} \times B_0 \Big)\cdot\nabla_p\Big]\sqrt{\mathbf{M}},\frac{2T_0}{u^0}f^{\varepsilon}\Big\rangle\nonumber\\
&+\frac{1}{2}\Big\langle\Big[\Big(\partial_t+\hat{p}\cdot\nabla_x\Big)\Big(\frac{2T_0}{u^0}\Big)\Big]f^{\varepsilon},f^{\varepsilon}\Big\rangle\nonumber\\
 &+ \Big\langle u \cdot\Big(\hat{p} \times B_R^{\varepsilon} \Big)\sqrt{\mathbf{M}},f^{\varepsilon}\Big\rangle+\varepsilon^{k-1}\Big\langle\Gamma(f^{\varepsilon},f^{\varepsilon}),\frac{2T_0}{u^0}f^{\varepsilon}\Big\rangle\nonumber\\
 &+\sum_{i=1}^{2k-1}\varepsilon^{i-1}\Big\langle\big[\Gamma( \frac{F_i}{\sqrt{\mathbf{M}}}, f^{\varepsilon})+\Gamma(f^{\varepsilon}, \frac{F_i}{\sqrt{\mathbf{M}}})\big],\frac{2T_0}{u^0}f^{\varepsilon}\Big\rangle\nonumber\\
 &-\varepsilon^k\Big\langle \Big(u^0\hat{p}-u\Big)\cdot\Big(E_R^{\varepsilon}+\hat{p} \times B_R^{\varepsilon}\Big)f^{\varepsilon},f^{\varepsilon}\Big\rangle\label{L202}\\
 &+\sum_{i=1}^{2k-1}\varepsilon^i\Big\langle \Big[\Big(E_i+\hat{p} \times B_i \Big)\cdot\nabla_pf^{\varepsilon}+ \Big(E_R^{\varepsilon}+\hat{p} \times B_R^{\varepsilon} \Big)\cdot\frac{\nabla_pF_i}{\sqrt{\mathbf{M}}}\Big],\frac{2T_0}{u^0}f^{\varepsilon}\Big\rangle\nonumber\\
 &-\sum_{i=1}^{2k-1}\varepsilon^i\Big\langle \Big[\Big(E_i+\hat{p} \times B_i \Big)\cdot\frac{1} { T_0}\Big(u^0\hat{p}-u\Big)f^{\varepsilon}\Big],\frac{2T_0}{u^0}f^{\varepsilon}\Big\rangle+\varepsilon^{k}\Big\langle\bar{A},\frac{2T_0}{u^0}f^{\varepsilon}\Big\rangle.
\nonumber
\end{align}
Now we estimate the terms in the right hand of \eqref{L202}. For brevity, we only treat the first, fourth and fifth terms since the remained terms can be estimated in the same way as Proposition $3.1$ in \cite{Guo-Jang-CMP-2010}.
Note that $\frac{1}{\sqrt{\mathbf{M}}}[\partial_t+\hat{p}\cdot\nabla_x-(E_0+\hat{p} \times B_0 )\cdot\nabla_p]\sqrt{\mathbf{M}}$ is a linear  polynomial of $p$, $(1+|p|)|f^{\varepsilon}|\leq (1+|p|)^{-7}|h^{\varepsilon}|$ from \eqref{temp}, and
$$\Big(\int_{|p|\geq\sqrt{\frac{\kappa}{\varepsilon}}}(1+|p|)^{-7\times2}dp\Big)^{1/2}\lesssim\Big(\frac{\varepsilon}{\kappa}\Big)^{\frac{11}{4}},$$
where $\kappa$ is a sufficiently small positive constant.
 Then the first term  can be estimated as follows:
\begin{align}
&-\Big\langle\frac{f^{\varepsilon}}{\sqrt{\mathbf{M}}}\Big[\partial_t+\hat{p}\cdot\nabla_x-\Big(E_0+\hat{p} \times B_0 \Big)\cdot\nabla_p\Big]\sqrt{\mathbf{M}},\frac{2T_0}{u^0}f^{\varepsilon}\Big\rangle\nonumber\\
&\hspace{0.5cm}\lesssim \left(\|\nabla_xn_0\|_{\infty}+\|\nabla_xu\|_{\infty}+\|\nabla_xE_0\|_{\infty}+\|\nabla_xB_0\|_{\infty}\right)
\|f^{\varepsilon}\|^2\nonumber\\
&\hspace{0.5cm}+(\|\nabla_xn_0\|+\|\nabla_xu\|+\|\nabla_xE_0\|+\|\nabla_xB_0\|)\nonumber\\
&\hspace{0.5cm}\times\|(1+|p|){\{\bf I-P\}}f^{\varepsilon}\|_{L^{\infty}_xL^2_p}\|{\{\bf I-P\}}f^{\varepsilon}\|\label{L200}\\
&\hspace{0.5cm}\lesssim (1+t)^{-\beta_0}\|f^{\varepsilon}\|^2+\|{\{\bf I-P\}}f^{\varepsilon}\|\nonumber\\
&\hspace{0.5cm}\times\|(1+|p|){\{\bf I-P\}}f^{\varepsilon}({\bf 1}_{|p|\geq\sqrt{\frac{\kappa}{\varepsilon}}}+{\bf 1}_{|p|\leq\sqrt{\frac{\kappa}{\varepsilon}}})\|_{L^{\infty}_xL^2_p}\|{\{\bf I-P\}}f^{\varepsilon}\|\nonumber\\
&\hspace{0.5cm}\lesssim (1+t)^{-\beta_0}\|f^{\varepsilon}\|^2+\frac{\kappa}{\varepsilon}\|{\{\bf I-P\}}f^{\varepsilon}\|^2+\varepsilon^{\frac{11}{4}}\|h^{\varepsilon}\|_{\infty}\|f^{\varepsilon}\|.\nonumber
\end{align}
By Theorem 2 in \cite{Guo-Strain-CMP-2012}, we can estimate the fourth term  as
\begin{align*}
&\varepsilon^{k-1}\Big\langle\Gamma(f^{\varepsilon},f^{\varepsilon}),\frac{2T_0}{u^0}f^{\varepsilon}\Big\rangle\nonumber\\
&=\varepsilon^{k-1}\Big\langle\Gamma(f^{\varepsilon},f^{\varepsilon}),\frac{2T_0}{u^0}{\{\bf I-P\}}f^{\varepsilon}\Big\rangle\nonumber\\
&\lesssim \varepsilon^{k-1}\|h^{\varepsilon}\|_{\infty} \|f^{\varepsilon}\|\|{\{\bf I-P\}}f^{\varepsilon}\|\\
&\lesssim \varepsilon^{k-2}\|h^{\varepsilon}\|_{\infty} \|{\{\bf I-P\}}f^{\varepsilon}\|^2+\varepsilon^{k}\|h^{\varepsilon}\|_{\infty} \|f^{\varepsilon}\|^2.\nonumber
\end{align*}
For the fifth term, we use Theorem 2 in \cite{Guo-Strain-CMP-2012} and \eqref{growth0} in Theorem \ref{fn} to obtain
\begin{align*}
&\sum_{i=1}^{2k-1}\varepsilon^{i-1}\Big\langle\Big[\Gamma( \frac{F_i}{\sqrt{\mathbf{M}}}, f^{\varepsilon})+\Gamma(f^{\varepsilon}, \frac{F_i}{\sqrt{\mathbf{M}}})\big],\frac{2T_0}{u^0}f^{\varepsilon}\Big\rangle\nonumber\\
&\hspace{0.5cm}=\sum_{i=1}^{2k-1}\varepsilon^{i-1}\Big\langle\big[\Gamma( \frac{F_i}{\sqrt{\mathbf{M}}}, f^{\varepsilon})+\Gamma(f^{\varepsilon},  \frac{F_i}{\sqrt{\mathbf{M}}})\big],\frac{2T_0}{u^0}{\{\bf I-P\}}f^{\varepsilon}\Big\rangle\nonumber\\
&\hspace{0.5cm}\lesssim\sum_{i=1}^{2k-1}\varepsilon^{i-1}\|F_i\|_{L^{\infty}_xL^2_p}\|f^{\varepsilon}\|\|{\{\bf I-P\}}f^{\varepsilon}\|\\
&\hspace{0.5cm}\lesssim \frac{\kappa}{\varepsilon}\|{\{\bf I-P\}}f^{\varepsilon}\|^2+\varepsilon\Big(\sum_{i=1}^{2k-1}\varepsilon^{i-1} (1+t)^{i}\Big)^2\|f^{\varepsilon}\|^2.\nonumber
\end{align*}

\end{proof}

\section{Characteristics and $L^{\infty}$ estimates}\label{cl}

\setcounter{equation}{0}

In this Section, we first establish the characteristic estimates under the a priori assumptions \eqref{assump}.
Then the $L^{\infty}$ estimate of $h^{\varepsilon}$ will be presented.

\subsection{Characteristics  estimates}
In this subsection, we will provide several characteristics estimates for the relativistic Vlasov-Maxwell-Boltzmann system. To this aim, we need a uniform $W^{1,\infty}$ estimate for the electro-magnetic field $[E^{\varepsilon}(t,x), B^{\varepsilon}(t,x)]$.

We first note that, under the a priori assumptions \eqref{assump},
 \begin{align}\label{uniEB}
& \sup_{t\in[0,T]}\|[E^{\varepsilon}(t), B^{\varepsilon}(t)]\|_{W^{1,\infty}}\nonumber\\
&\hspace{0.5cm}\leq  \sum_{i=0}^{2k-1}\varepsilon^i \|[E_i(t), B_i(t)]\|_{W^{1,\infty}}+\varepsilon^k\|[E^{\varepsilon}_ R(t), B^{\varepsilon}_ R(t)]\|_{W^{1,\infty}}\lesssim 1,
 \end{align}
 for $T\in [0,\frac{1}{2}|\ln\varepsilon|^{\frac{1}{3}}]$ and $k\geq5$.
With the uniform estimate \eqref{uniEB}, we can study the characteristics defined in \eqref{chara1}.

\begin{lemma}\label{charc} Let  $T\in [0,\frac{1}{2}|\ln\varepsilon|^{\frac{1}{3}}]$ and assumptions in \eqref{assump} hold. Then there exists a small constant $\bar{T}\in [0,T]$ such that for $\tau\in[0, \bar{T}]$,
\begin{align}
|D_pX(\tau)|&\lesssim \frac{|t-\tau|}{p^0},\label{Xp}\\
\frac{d^2 D_pX(\tau; t, x, p)}{d\tau^2}&\lesssim \frac{1}{(p^0)^2},\label{Xptt}\\
\frac{1}{2(p^0)^5}|t-\tau|^3&\leq\left|\det\left(\frac{\partial X(\tau)}{\partial p}\right)\right|\leq\frac{2}{(p^0)^5}|t-\tau|^3.\label{detXp}
\end{align}
\end{lemma}
\begin{proof} We first prove \eqref{Xp}.
For $i, j=1,2,3$, we differentiate \eqref{chara1} w.r.t. $p_i$  to have
\begin{align}\label{xpi1}
\frac{d \partial_{p_i}X_j(\tau; t, x, p)}{d\tau}&=\Big(\frac{(P^0)^2\partial_{p_i}P_j-P_j(P\cdot\partial_{p_i}P)}{(P^0)^3}\Big)(\tau; t, x, p),
\end{align}
and
\begin{align}\label{ppi1}
&\frac{d \partial_{p_i}P_j(\tau; t, x, p)}{d\tau}\nonumber\\
&=-[\nabla_xE_j^{\varepsilon}\cdot\partial_{p_i}X](\tau, X(\tau; t, x, p))\nonumber\\
&\hspace{0.5cm}-\Big[\hat{P}(\tau; t, x, p)\times [\nabla_xB^{\varepsilon}\cdot\partial_{p_i}X](\tau, X(\tau; t, x, p))\Big]_j\\
&\hspace{0.5cm}-\Big[\Big(\frac{(P^0)^2\partial_{p_i}P-P(P\cdot\partial_{p_i}P)}{(P^0)^3}\Big)(\tau; t, x, p)\times B^{\varepsilon}(\tau, X(\tau; t, x, p))\Big]_j.\nonumber
\end{align}
Note that $\partial_{p_i}P_j(t; t, x, p)=\delta_{ij}$.
We integrate \eqref{ppi1} over $[t, \tau]$ and use \eqref{uniEB} to get
\begin{align*}
 \|\partial_{p_i}P_j(\tau)\|_{\infty}&\leq \delta_{ij}+C|\tau-t|\Big(\sup_{\tau\in[0,\bar{T}]}\|\partial_{p_i}X(\tau)\|_{\infty}
 +\frac{\sup_{\tau\in[0,\bar{T}]}\|\partial_{p_i}P(\tau)\|_{\infty}}{\inf_{\tau\in[0,\bar{T}]}\|P^0(\tau)\|_{\infty}}\Big).
\end{align*}
Then, for sufficiently small $\bar{T}$, it holds that
\begin{align}\label{ppi3}
 \|\partial_{p_i}P_j(\tau)\|_{\infty}&\leq \frac{5}{4}\delta_{ij}+C|\tau-t|\sup_{\tau\in[0,\bar{T}]}\|\partial_{p_i}X(\tau)\|_{\infty}.
\end{align}
Noting that $\partial_{p_i}X_j(t; t, x, p)=0$, and for some $\bar{\tau} $ between $\tau$ and $t$,
\begin{align}\label{p0tt}
P^0(\tau)=&p^0+(\tau-t)\frac{d P}{d\tau}(\bar{\tau}; t, x, p)
\leq p^0+C|\tau-t|,
\end{align}
we integrate \eqref{xpi1} over $[t, \tau]$ and use \eqref{p0tt} to obtain
\begin{align}\label{xpi2}
\|\partial_{p_i}X_j(\tau)\|_{\infty}\leq \frac{2|t-\tau|}{\inf_{\tau\in[0,\bar{T}]}\|P^0(\tau)\|_{\infty}}\lesssim \frac{|t-\tau|}{p^0},
\end{align}
for $\bar{T}$ small enough. Inserting \eqref{xpi2} into \eqref{ppi3} yields
\begin{align}\label{ppij}
 \|\partial_{p_i}P_j(\tau)\|_{\infty}\leq 2,
\end{align}
for sufficiently small $\bar{T}$. Moreover, from \eqref{chara1}, one has
\begin{align}
\frac{d^2 \partial_{p_i}X_j(\tau)}{d\tau^2}=&\partial_{p_i}\Big(\frac{1}{P^0}\frac{dP_j}{d\tau}
-\frac{P_j}{(P^0)^3}\frac{dP}{d\tau}\cdot P\Big)(\tau)\nonumber\\
=&\partial_{p_i}\Big[\frac{-E^{\varepsilon}_j-[\hat{P}\times B^{\varepsilon}]_j}{P^0}+\frac{P_jP\cdot E^{\varepsilon}}{(P^0)^3}\Big](\tau)\nonumber\\
=&\Big\{\frac{-\nabla_xE^{\varepsilon}_j\cdot \partial_{p_i}X-[\hat{P}\times (\nabla_xB^{\varepsilon}\cdot\partial_{p_j}X)]_j-[\partial_{p_i}\hat{P}\times B^{\varepsilon}]_j}{P^0}\label{ijtt}\\
&+\frac{\partial_{p_i}P_jP\cdot E^{\varepsilon}+P_j\partial_{p_i}P\cdot E^{\varepsilon}+P_jP\cdot (\nabla_xE^{\varepsilon}\cdot \partial_{p_i} X)}{(P^0)^3}\nonumber\\
&+\frac{E^{\varepsilon}_j+[\hat{P}\times B^{\varepsilon}]_j}{(P^0)^3}P\cdot\partial_{p_i}P-\frac{3P_jP\cdot E^{\varepsilon}(P\cdot\partial_{p_i}P)}{(P^0)^5}\Big\}(\tau)\nonumber.
\end{align}
Then we use \eqref{p0tt}, \eqref{xpi2} and \eqref{ppij} to bound $\frac{d^2 \partial_{p_i}X_j(\tau)}{d\tau^2}$ by $\frac{C}{(p^0)^2}(1+|t-\tau|)$. \eqref{Xptt} follows for small $\bar{T}$.

Finally, we prove \eqref{detXp}. Expand $\partial_{p_i}X_j(\tau)$ at $t$ to have
\begin{align}
\partial_{p_i}X_j(\tau)=&\partial_{p_i}X_j(t)+\frac{d \partial_{p_i}X_j(\tau; t, x, p)}{d\tau}\Big|_{\tau=t}(\tau-t)+\frac{(\tau-t)^2}{2}\frac{d^2 \partial_{p_i}X_j(\bar{\tau})}{d\tau^2}\nonumber\\
=&(\tau-t)\frac{(p^0)^2\delta_{ij}-p_ip_j}{(p^0)^3}+\frac{(\tau-t)^2}{2}\Big\{\frac{-\nabla_xE^{\varepsilon}_i\cdot \partial_{p_j}X}{P^0}\label{pxij}\\
&-\frac{[\hat{P}\times (\nabla_xB^{\varepsilon}\cdot\partial_{p_j}X)]_j+[\partial_{p_j}\hat{P}\times B^{\varepsilon}]_j}{P^0}\nonumber\\
&+\frac{\partial_{p_i}P_jP\cdot E^{\varepsilon}+P_j\partial_{p_i}P\cdot E^{\varepsilon}+P_jP\cdot (\nabla_xE^{\varepsilon}\cdot \partial_{p_i} X)}{(P^0)^3}\nonumber\\
&+\frac{E^{\varepsilon}_j+[\hat{P}\times B^{\varepsilon}]_j}{(P^0)^3}P\cdot\partial_{p_i}P-\frac{3P_jP\cdot E^{\varepsilon}(P\cdot\partial_{p_i}P)}{(P^0)^5}\Big\}(\bar{\tau})\nonumber.
\end{align}
Then, we have
\begin{align}
\left|\det\left(\frac{\partial X(\tau)}{\partial p}\right)\right|=&|\tau-t|^3
\left|\det\left(\frac{\bf I}{(p^0)^3}-\frac{p\otimes p}{(p^0)^5}+\frac{\tau-t}{2}\frac{d^2 \partial_{p_i}X_j(\bar{\tau})}{d\tau^2}\right)\right|\nonumber\\
=&|\tau-t|^3\left|\frac{1}{(p^0)^5}+\frac{\tau-t}{2}\sum_{ i,j=1}^3\frac{\delta_{ij}+p_ip_j}{(p^0)^4}
\frac{d^2 \partial_{p_i}X_j(\bar{\tau})}{d\tau^2}\right.\nonumber\\
&\left.+\frac{O(1)|\tau-t|^2}{p^0}\Big\|\frac{d^2 \partial_{p}X(\bar{\tau})}{d\tau^2}\Big\|_{\infty}^2+O(1)|\tau-t|^3\Big\|\frac{d^2 \partial_{p}X(\bar{\tau})}{d\tau^2}\Big\|_{\infty}^3\right|. \nonumber
\end{align}
Here $\bf I$ is the $3\times3$ identity matrix. By \eqref{p0tt} and \eqref{ijtt}, we can estimate $\left|\det\left(\frac{\partial X(\tau)}{\partial p}\right)\right|$ as
\begin{equation}\label{tt1}
|\tau-t|^3\left|[1+O(1)|\tau-t|^2]^2\frac{1}{(p^0)^5}+\frac{\tau-t}{2}\sum_{ i,j=1}^3\frac{\delta_{ij}+p_ip_j}{(p^0)^4}
\frac{d^2 \partial_{p_i}X_j(\bar{\tau})}{d\tau^2}\right|.
\end{equation}
Inserting the expression of $\partial_{p_i}X_j(\tau)$ in \eqref{pxij} into \eqref{tt1}, after a long computation, we can further
estimate $\left|\det\left(\frac{\partial X(\tau)}{\partial p}\right)\right|$ as
\begin{align}\label{tt2}
\begin{aligned}
&\frac{|\tau-t|^3}{(p^0)^5}\Big|[1+O(1)|\tau-t|^2]^2+\frac{(\tau-t)}{2p^0}\Big[\hat{P}\cdot E^{\varepsilon}\\
&-\nabla_x\cdot E^{\varepsilon}+\sum_{ i,j=1}^3\hat{P}_i\hat{P}_j\partial_{x_i}E^{\varepsilon}_j+\hat{P}\cdot(\nabla_x\times B^{\varepsilon})\Big](\bar{\tau})\Big|.
\end{aligned}
\end{align}
Then we combine \eqref{tt2} and \eqref{p0tt}, and choose $\bar{T}$ sufficiently small to obtain \eqref{detXp}.

\end{proof}

\subsection{ $L^{\infty}$  estimate of $h^{\varepsilon}$}  As in \cite{Guo-Jang-CMP-2010}, we will establish the $L^{\infty}$ norm of $h^{\varepsilon}$ by the characteristic method. Note that, from \eqref{remain} and \eqref{Linfty}, $h^{\varepsilon}$ satisfies
 \begin{align}
& \partial_th^{\varepsilon}+\hat{p}\cdot\nabla_xh^{\varepsilon}-\Big(E^{\varepsilon}+\hat{p} \times B^{\varepsilon} \Big)\cdot\nabla_ph^{\varepsilon}+\frac{\nu(\mathbf{M})h^{\varepsilon}}{\varepsilon}+\frac{\bar{K}(h^{\varepsilon})}{\varepsilon}\nonumber\\
 &=\sum_{i=1}^{2k-1}\varepsilon^{i-1}\frac{w}{\sqrt{J_{M}}}[Q( F_i, \frac{\sqrt{J_{M}}}{w}h^{\varepsilon})
 +Q(\frac{\sqrt{J_{M}}}{w}h^{\varepsilon}, F_i)]\nonumber\\
 &+\varepsilon^{k-1}\frac{w}{\sqrt{J_{M}}}Q(\frac{\sqrt{J_{M}}}{w}h^{\varepsilon},\frac{\sqrt{J_{M}}}{w}h^{\varepsilon})
 +\Big(E^{\varepsilon}+\hat{p} \times B^{\varepsilon} \Big)\cdot\nabla_p\Big(\frac{\sqrt{J_{M}}}{w}\Big)h^{\varepsilon}
 \nonumber\\
 &+\Big(E_R^{\varepsilon}+\hat{p} \times B_R^{\varepsilon}\Big)\cdot\frac{w}{\sqrt{J_{M}}}\nabla_p\Big(\mathbf{M}+\sum_{i=1}^{2k-1}\varepsilon^iF_i \Big)+\varepsilon^{k}\frac{w}{\sqrt{J_{M}}}A,
\nonumber
\end{align}
 where $\bar{K}(h^{\varepsilon})=\frac{w(|p|)\sqrt{\mathbf{M}}}{\sqrt{J_{M}}}K_{\mathbf{M}}(\frac{\sqrt{J_{M}}}{w\sqrt{\mathbf{M}}}h^{\varepsilon})$. With the characteristic equations in \eqref{chara1}, the corresponding characteristic estimates in Lemma \ref{charc}, and the estimates for the kernel of the operator $\bar{K}$ in Lemma \ref{k12}, we can
 obtain the following proposition.

\begin{proposition}\label{linfty} For given $T\in [0,\frac{1}{2}|\ln\varepsilon|^{\frac{1}{3}}]$, assume the crucial bootstrap assumptions \eqref{assump}. Then, for sufficiently small $\varepsilon$, we have
\begin{align*}
\sup_{0\leq t\leq T}\Big(\varepsilon^{\frac{3}{2}}\|h^{\varepsilon}(t)\|_{\infty}\Big)\lesssim& \varepsilon^{\frac{3}{2}}\|h^{\varepsilon}(0)\|_{\infty}+\sup_{0\leq t\leq T}\Big(\varepsilon^{\frac{5}{2}}\|[E^{\varepsilon}_R(t), B^{\varepsilon}_R(t)]\|_{\infty}\Big)+\sup_{0\leq t\leq T}\|f^{\varepsilon}(t)\|+\varepsilon^{k+\frac{3}{2}}.
\end{align*}
\end{proposition}
\begin{proof} The proof is quite similar to the corresponding $L^{\infty}$ estimate of  $h^{\varepsilon}$ in subsection 5.1 of \cite{Guo-Jang-CMP-2010}. The main difference is that $\nabla_x\phi$ is replaced by $E+\hat{p} \times B$. In \cite{Guo-Jang-CMP-2010}, the electric potential $\phi$ satisfies the Poisson equation and $\|\nabla_x\phi\|_{\infty}$ can be controlled by $\|h^{\varepsilon}\|_{\infty}$. However, in our case, the $L^{\infty}$ norm of the electromagnetic field $\|[E^{\varepsilon}_R, B^{\varepsilon}_R]\|_{\infty}$ can't be bounded by $\|h^{\varepsilon}\|_{\infty}$. This is the very reason  that the term $\sup_{0\leq t\leq T}\Big(\varepsilon^{\frac{5}{2}}\|[E^{\varepsilon}_R(t), B^{\varepsilon}_R(t)]\|_{\infty}\Big)$ appears in the upper bound of $h^{\varepsilon}$. In fact, it arises from the estimate of the term $\big(E_R^{\varepsilon}+\hat{p} \times B_R^{\varepsilon}\big)\cdot\frac{w}{\sqrt{J_{M}}}\nabla_p\Big(\mathbf{M}+\sum_{i=1}^{2k-1}\varepsilon^iF_i \Big)$ in the equation of $h^{\varepsilon}$. More explicitly, corresponding to the seventh term in \cite[equation (5.9)]{Guo-Jang-CMP-2010}, we have
$$\big|\frac{w}{\sqrt{J_{M}}}\nabla_p\Big(\mathbf{M}+\sum_{i=1}^{2k-1}\varepsilon^iF_i \Big)\big|\lesssim 1+ {\mathcal I}_1(t)$$
by \eqref{growth0} in Theorem \ref{fn} and can further obtain
\begin{eqnarray*}
&&
\int_0^t\exp\Big\{-\frac{1}{\varepsilon}\int_s^t\nu(\tau)d\tau\Big\}\Big|\big(E_R^{\varepsilon}+\hat{p} \times B_R^{\varepsilon}\big)\cdot\frac{w}{\sqrt{J_{M}}}\nabla_p\Big(\mathbf{M}+\sum_{i=1}^{2k-1}\varepsilon^iF_i \Big)\Big|(s, X(s),P(s))\,ds\nonumber\\
&&\lesssim \int_0^t\exp\Big\{-\frac{\nu_0 (t-s)}{\varepsilon}\Big\}\big[1+ \varepsilon{\mathcal I}_1(s)\big]\,ds\sup_{0\leq s\leq t}\Big(\|[E^{\varepsilon}_R(s), B^{\varepsilon}_R(s)]\|_{\infty}\Big)\\
&&\lesssim \varepsilon\big[1+ {\mathcal I}_1(t)\big]\sup_{0\leq s\leq t}\Big(\|[E^{\varepsilon}_R(s), B^{\varepsilon}_R(s)]\|_{\infty}\Big),\nonumber
\end{eqnarray*}
 where we used $\nu(\tau)\geq \nu_0$ for some uniform positive constant $\nu_0$. We further apply similar arguments in \cite{Guo-Jang-CMP-2010} to complete the proof.

\end{proof}

\section{$W^{1,\infty}$ estimates for $(E^{\varepsilon}_R, B^{\varepsilon}_R)$}\label{w1eb}

\setcounter{equation}{0}
In this section, we use the well-known Glassey-Strauss Representation in \cite{Glassey-Strauss-ARMA-1986} to perform $W^{1,\infty}$ estimates for $(E^{\varepsilon}_R, B^{\varepsilon}_R)$.
\begin{proposition}\label{TWEB}
For $t\in [0, \frac{1}{2}|\ln\varepsilon|^{\frac{1}{3}}]$, remainders $(E^{\varepsilon}_R, B^{\varepsilon}_R)$ satisfy
\begin{align}
&\|[E^{\varepsilon}_R(t), B^{\varepsilon}_R(t)]\|_{W^{1,\infty}}\nonumber\\
&\lesssim
|\ln \varepsilon|^{\frac{1}{3}}\int_0^t\|[E^{\varepsilon}_R(\tau), B^{\varepsilon}_R(\tau)]\|_{W^{1,\infty}}\,d\tau+\frac{|\ln \varepsilon|^{\frac{1}{3}}}{\varepsilon}\int_0^t\|h^{\varepsilon}(\tau)\|_{W^{1,\infty}}\, d\tau\label{TWEB1}\\
 &+\varepsilon^{-\frac{17}{8}}|\ln \varepsilon|^2\big(\|[E^{\varepsilon}_R(0), B^{\varepsilon}_R(0)]\|_{W^{1,\infty}}+\|h^{\varepsilon}(0)\|_{\infty}+\varepsilon^k\big)+\varepsilon^{-\frac{29}{8}}|\ln \varepsilon|^2\sup_{0\leq s\leq t}\|f^{\varepsilon}(t)\|.\nonumber
\end{align}
\end{proposition}
\begin{proof}
Proposition \ref{TWEB} is a combination of Proposition \ref{LEBe} and Proposition \ref{WEBe}.
\end{proof}

\subsection{$L^{\infty}$ estimate for $(E^{\varepsilon}_R, B^{\varepsilon}_R)$} Define an operator $S$ by
$$S=\partial_t+\hat{p}\cdot\nabla_x.$$
The remainder of the electro-magnetic  field $(E^{\varepsilon}_R, B^{\varepsilon}_R)$ can be explicitly expressed in terms of $F^{\varepsilon}_R$ and $SF^{\varepsilon}_R$ as follows:
\begin{lemma}\cite{Glassey-Strauss-ARMA-1986}\label{LinftyEB} For $i=1,2,3$, one has
\begin{align}\label{LEB0}
&4\pi E_R^{\varepsilon,i}(t,x)=\big(E_R^{\varepsilon,i}\big)_0(t,x)+E_{R,T}^{\varepsilon,i}(t,x)+E_{R,S}^{\varepsilon,i}(t,x),\nonumber\\
&4\pi B_R^{\varepsilon,i}(t,x)=\big(B_R^{\varepsilon,i}\big)_0(t,x)+B_{R,T}^{\varepsilon,i}(t,x)+B_{R,S}^{\varepsilon,i}(t,x),
\end{align}
where $\big(E_R^{\varepsilon,i}\big)_0(t,x)$ and $\big(B_R^{\varepsilon,i}\big)_0(t,x)$ are initial data of $E^{\varepsilon,i}_R$ and $B^{\varepsilon,i}_R$, and
\begin{align*}
&E_{R,T}^{\varepsilon,i}(t,x)=-\iint\limits_{|x-y|\leq t}\frac{dpdy}{|y-x|^2}\frac{(\omega_i+\hat{p}_i)(1-|\hat{p}|^2)}{(1+\hat{p}\cdot\omega)^2}F^{\varepsilon}_R(t-|y-x|,y,p),\\
&E_{R,S}^{\varepsilon,i}(t,x)=-\iint\limits_{|x-y|\leq t}\frac{dpdy}{|y-x|}\frac{(\omega_i+\hat{p}_i)}{1+\hat{p}\cdot\omega}(SF^{\varepsilon}_R)(t-|y-x|,y,p),\\
&B_{R,T}^{\varepsilon,i}(t,x)=\iint\limits_{|x-y|\leq t}\frac{dpdy}{|y-x|^2}\frac{(\omega\times\hat{p})_i(1-|\hat{p}|^2)}{(1+\hat{p}\cdot\omega)^2}F^{\varepsilon}_R(t-|y-x|,y,p),\\
&B_{R,S}^{\varepsilon,i}(t,x)=\iint\limits_{|x-y|\leq t}\frac{dpdy}{|y-x|}\frac{(\omega\times\hat{p})_i}{1+\hat{p}\cdot\omega}(SF^{\varepsilon}_R)(t-|y-x|,y,p),
\end{align*}
with $\omega=\frac{x-y}{|x-y|}$.

\end{lemma}

From the kinetic equation of $F^{\varepsilon}_R$ in \eqref{remain}, we can obtain the upper bound of the electro-magnetic field.

\begin{proposition}\label{LEBe}
For $t\in [0, \frac{1}{2}|\ln\varepsilon|^{\frac{1}{3}}]$, the remainders $E^{\varepsilon}_R$ and $B^{\varepsilon}_R$ can be estimated as
\begin{align}
\|[E^{\varepsilon}_R(t), B^{\varepsilon}_R(t)]\|_{\infty}\lesssim &\varepsilon^{-\frac{9}{8}}|\ln \varepsilon|\big(\|[E^{\varepsilon}_R(0), B^{\varepsilon}_R(0)]\|_{\infty}+\|h^{\varepsilon}(0)\|_{\infty}+\varepsilon^k\big)\nonumber\\
&+\varepsilon^{-\frac{21}{8}}|\ln \varepsilon|\sup_{0\leq s\leq t}\|f^{\varepsilon}(t)\|.\label{LEB1}
\end{align}
\end{proposition}
\begin{proof}
Since the estimate of $B^{\varepsilon}_R$ can be done in a same way, we only deal with $E^{\varepsilon}_R$.
We first estimate $E_{R,T}^{\varepsilon,i}$ in \eqref{LEB0}.  Note that
\begin{align*}
&1-|\hat{p}|^2=(1+|\hat{p}|)(1-|\hat{p}|)\leq 2(1-|\hat{p}|)\lesssim (1+|p|)^{-2},\\
&(1+\hat{p}\cdot\omega)^{-m}\leq (1-|\hat{p}|)^{-m}\lesssim (1+|p|)^{2m},\\
&\Big|\nabla_p\Big(\frac{\omega_i+\hat{p}_i}{1+\hat{p}\cdot\omega}\Big)\Big|\lesssim \frac{|\nabla_p\hat{p}_i|}{1+\hat{p}\cdot\omega}+ \frac{|\nabla_p\hat{p}|}{(1+\hat{p}\cdot\omega)^2}\lesssim (1+|p|)^3
\end{align*}
 for any $m\geq0$ and $i=1, 2, 3$.  Then we have
\begin{equation*}
|E_{R,T}^{\varepsilon,i}|\lesssim t \sup_{s\in[0,t]}\|h^{\varepsilon}(s)\|_{\infty}\int_{\mathbb{R}^3}(1+|p|)^2\frac{\sqrt{J_{M}}}{w(|p|)}dp\lesssim t \sup_{s\in[0,t]}\|h^{\varepsilon}(s)\|_{\infty}.
\end{equation*}
For $E_{R,S}^{\varepsilon,i}$, we obtain from the definition of operator $S$ and \eqref{remain} that
\begin{align}
E_{R,S}^{\varepsilon,i}(t,x)=&-\iint\limits_{|x-y|\leq t}\frac{dpdy}{|y-x|}\frac{(\omega_i+\hat{p}_i)}{1+\hat{p}\cdot\omega}\Big[(E_R^{\varepsilon}+\hat{p} \times B_R^{\varepsilon} )\cdot\nabla_pF^{\varepsilon}\Big](t-|y-x|,y,p)\nonumber\\
&-\iint\limits_{|x-y|\leq t}\frac{dpdy}{|y-x|}\frac{(\omega_i+\hat{p}_i)}{1+\hat{p}\cdot\omega}\Big[\sum_{j=0}^{2k-1}(E_j+\hat{p} \times B_j )\cdot\nabla_pF_R^{\varepsilon}\Big](t-|y-x|,y,p)\label{LEB2}\\
 &-\frac{1}{\varepsilon}\iint\limits_{|x-y|\leq t}\frac{dpdy}{|y-x|}\frac{(\omega_i+\hat{p}_i)}{1+\hat{p}\cdot\omega}[Q(F_R^{\varepsilon},F^{\varepsilon})+Q(F^{\varepsilon},F_R^{\varepsilon})](t-|y-x|,y,p)\nonumber\\
 &-\varepsilon^{k}\iint\limits_{|x-y|\leq t}\frac{dpdy}{|y-x|}\frac{(\omega_i+\hat{p}_i)}{1+\hat{p}\cdot\omega}A(t-|y-x|,y,p)
.\nonumber
\end{align}
Now we estimate terms in \eqref{LEB2}. By the a priori assumptions \eqref{assump} and \eqref{growth0} in Theorem \ref{fn}, we have
\begin{align}\label{Fpl}
&\Big\|\Big|\nabla_p\Big(\frac{\omega_i+\hat{p}_i}{1+\hat{p}\cdot\omega}\Big)\Big|F ^{\varepsilon}\Big\|_{L^1_p}\nonumber\\
\lesssim &\int_{\mathbb R^3}  (1+|p|)^3 \Big(F_0+\sum_{j=1}^{2k-1}\varepsilon^j|F_j|+\varepsilon^k\frac{\sqrt{J_{M}}}{w}|h^{\varepsilon}|\Big)\, dp\\
\lesssim & 1+\sum_{j=1}^{2k-1}\varepsilon^j(1+t)^j+\varepsilon^k\sup_{t\in [0, \frac{1}{2}|\ln\varepsilon|^{\frac{1}{3}}]}\|h^{\varepsilon}(t)\|_{\infty}\lesssim 1  \nonumber.
\end{align}
Then, for the first term in the right hand side of \eqref{LEB2}, we integrate by parts w.r.t. $p$ and let $\tau=t-|y-x|$ to bound it by
\begin{align*}
&\int\limits_{|x-y|\leq t}\frac{dy}{|y-x|}\Big[|E_R^{\varepsilon}|+| B_R^{\varepsilon}| \Big\|\Big|\nabla_p\Big(\frac{\omega_i+\hat{p}_i}{1+\hat{p}\cdot\omega}\Big)\Big|F ^{\varepsilon}\Big\|_{L^1_p}\Big](t-|y-x|,y)\\
&\lesssim \int_0^t (t-\tau)\|[E_R^{\varepsilon}(\tau), B_R^{\varepsilon}(\tau)]\|_{\infty}\, d\tau .
\end{align*}
Similarly, the upper bound of the second line in  \eqref{LEB2} is
\begin{align*}
&\int\limits_{|x-y|\leq t}\frac{dy}{|y-x|}\Big[\sum_{j=0}^{2k-1}\varepsilon^j(|E_j|+| B_j| )| \Big\|\Big|\nabla_p\Big(\frac{\omega_i+\hat{p}_i}{1+\hat{p}\cdot\omega}\Big)\Big|\frac{\sqrt{J_{M}}}{w}\Big\|_{L^1_p}\|h^{\varepsilon}\|_{L^{\infty}_p}\Big](t-|y-x|,y)
\\
&\lesssim  \int_0^t\Big[(1+\tau)^{-\beta_0}+\sum_{j=1}^{2k-1}\varepsilon^j(1+\tau )^j\Big]  (t-\tau)\|h^{\varepsilon}(\tau)]\|_{\infty}\, d\tau
\lesssim  \int_0^t(t-\tau)\|h^{\varepsilon}(\tau)]\|_{\infty}\, d\tau
\end{align*}
by  \eqref{growth0} in Theorem \ref{fn} again.
Note  that
$$(1+|p|)\sqrt{\mathbf{M}}(p)\lesssim 1, \quad \sqrt{\mathbf{M}}(p'))\sqrt{\mathbf{M}}(q')\leq \sqrt{\mathbf{M}}(p), $$
by \eqref{conservations}, and
\begin{align}\label{Fpj}
 & \Big\|Q(\frac{F_R^{\varepsilon}}{\sqrt{\mathbf{M}}}, \frac{F^{\varepsilon}}{\sqrt{\mathbf{M}}})\|_{L^1_p}\lesssim \Big\|\frac{\sqrt{J_{M}} h^{\varepsilon}}{w\sqrt{\mathbf{M}}}\Big\|_{L^1_p} \Big\|\frac{F^{\varepsilon}}{\sqrt{\mathbf{M}}}\Big\|_{L^1_p}  \nonumber\\
 \lesssim & \| h^{\varepsilon}\|_{L^{\infty}_p} \Big(\|\sqrt{\mathbf{M}}\|_{L^1_p}+\sum_{j=0}^{2k-1}\varepsilon^j \Big\|\frac{F_j}{\sqrt{\mathbf{M}}}\Big\|_{L^1_p} +\varepsilon^k\Big\|\frac{\sqrt{J_{M}} h^{\varepsilon}}{w\sqrt{\mathbf{M}}}\Big\|_{L^1_p}\Big)\\
 \lesssim & \| h^{\varepsilon}\|_{L^{\infty}_p} \Big(1+\sum_{j=0}^{2k-1}\varepsilon^j (1+t)^j +\varepsilon^k\| h^{\varepsilon}\|_{\infty}\Big)\lesssim \| h^{\varepsilon}\|_{L^{\infty}_p} \nonumber
\end{align}
by the a priori assumptions \eqref{assump} and \eqref{growth0} in Theorem \ref{fn}.
The third line in  \eqref{LEB2} can be controlled by
 \begin{align*}
&\frac{1}{\varepsilon}\int\limits_{|x-y|\leq t}\frac{dy}{|y-x|}\sup_{p\in \mathbb R^3}\big[(1+|p|)^2\sqrt{\mathbf{M}}(p)\big]
\Big[\Big\|Q(\frac{F_R^{\varepsilon}}{\sqrt{\mathbf{M}}}, \frac{F^{\varepsilon}}{\sqrt{\mathbf{M}}})\Big\|_{L^1_p}\Big](t-|y-x|,y)\\
&\lesssim \frac{1}{\varepsilon}\int\limits_{|x-y|\leq t}\frac{dy}{|y-x|}\|h^{\varepsilon}(t-|y-x|,y)\|_{L^{\infty}_p}\lesssim \frac{1}{\varepsilon}\int_0^t(t-\tau)\|h^{\varepsilon}(\tau)]\|_{\infty}\, d\tau
\end{align*}
via letting $\tau=t-|y-x|$.
It is easy to obtain that the fourth line in \eqref{LEB2} can be estimated by $\varepsilon^k (1+t)^{2k+3}$.
We collect the estimates above to get
\begin{align*}
\|[E^{\varepsilon}_R(t), B^{\varepsilon}_R(t)]\|_{\infty}\lesssim &t \sup_{s\in[0,t]}\|h^{\varepsilon}(s)\|_{\infty}+\int_0^t (t-\tau)\|[E_R^{\varepsilon}(\tau), B_R^{\varepsilon}(\tau)]\|_{\infty}\, d\tau\nonumber\\
&+\frac{1}{\varepsilon}\int_0^t(t-\tau)\|h^{\varepsilon}(\tau)]\|_{\infty}\, d\tau+\varepsilon^k (1+t)^{2k+3}.
\end{align*}
Then, there is some constant $C_1>1$ such that for $t\in [0, \frac{1}{2}|\ln \varepsilon|^{\frac{1}{3}}]$, we can combine the above inequality and Proposition \ref{linfty} to have
\begin{align*}
&\sup_{0\leq s\leq t}\|[E^{\varepsilon}_R(s), B^{\varepsilon}_R(s)]\|_{\infty}\\
\leq &C_1\Big[\|[E^{\varepsilon}_R(0), B^{\varepsilon}_R(0)]\|_{\infty}+\Big(t+\frac{t^2}{\varepsilon}\Big)\|h^{\varepsilon}(0)\|_{\infty}+\varepsilon^{k-1}(1+ t)^2\\
&+\varepsilon^{-\frac{5}{2}}(1+ t)^2\sup_{0\leq s\leq t}\|f^{\varepsilon}(t)\|+t\int_0^t \sup_{0\leq \tau\leq s}\|[E_R^{\varepsilon}(\tau), B_R^{\varepsilon}(\tau)]\|_{\infty}\, d s\Big]\\
\leq  &C_1\Big[\|[E^{\varepsilon}_R(0), B^{\varepsilon}_R(0)]\|_{\infty}+\Big(t+\frac{t^2}{\varepsilon}\Big)\|h^{\varepsilon}(0)\|_{\infty}+\varepsilon^{k-1} (1+ t)^2\\
&+\varepsilon^{-\frac{5}{2}}(1+ t)^2\sup_{0\leq s\leq t}\|f^{\varepsilon}(t)\|\Big]+\frac{1}{4}|\ln \varepsilon|^{\frac{2}{3}}\int_0^t \sup_{0\leq \tau\leq s}\|[E_R^{\varepsilon}(\tau), B_R^{\varepsilon}(\tau)]\|_{\infty}\, d s,
\end{align*}
where we used $C_1\leq \frac{1}{2}|\ln \varepsilon|^{\frac{1}{3}}$ when $\varepsilon$ is sufficiently small. We further apply Gr\"{o}nwall's inequality to obtain that for $t\in [0, \frac{1}{2}|\ln \varepsilon|^{\frac{1}{3}}]$,
\begin{align*}
&\sup_{0\leq s\leq t}\|[E^{\varepsilon}_R(s), B^{\varepsilon}_R(s)]\|_{\infty}\\
\leq  &C_1\varepsilon^{-\frac{1}{8}}\Big[\|[E^{\varepsilon}_R(0), B^{\varepsilon}_R(0)]\|_{\infty}+\frac{|\ln \varepsilon|}{\varepsilon}\|h^{\varepsilon}(0)\|_{\infty}+\varepsilon^{k-1} |\ln \varepsilon|+\varepsilon^{-\frac{5}{2}}|\ln \varepsilon|\sup_{0\leq s\leq t}\|f^{\varepsilon}(t)\|\Big].
\end{align*}
\end{proof}

\subsection{$W^{1,\infty}$ estimates for $(E^{\varepsilon}_R, B^{\varepsilon}_R)$}
The gradient of the remainder of the electro-magnetic  field $(E^{\varepsilon}_R, B^{\varepsilon}_R)$ can be explicitly expressed in the following:
\begin{lemma}\cite{Glassey-Strauss-ARMA-1986}\label{W1EB} For $i,j=1,2,3$, one has
\begin{align}
&\partial_{x_j} E_R^{\varepsilon,i}(t,x)=\partial_{x_j}\big(E_R^{\varepsilon,i}\big)_0(t,x)+\iint\limits_{|x-y|\leq t}\frac{dpdy}{|y-x|^3}a(\omega,\hat{p})F^{\varepsilon}_R(t-|y-x|,y,p)\nonumber\\
&\hspace{0.5cm}+O(1)+\iint\limits_{|x-y|\leq t}\frac{dpdy}{|y-x|^2}b(\omega,\hat{p})(SF^{\varepsilon}_R)(t-|y-x|,y,p)\label{W1EB10}\\
&\hspace{0.5cm}+\iint\limits_{|x-y|\leq t}\frac{dpdy}{|y-x|}c(\omega,\hat{p})(S^2F^{\varepsilon}_R)(t-|y-x|,y,p),\nonumber
\end{align}
where $\partial_{x_j}\big(E_R^{\varepsilon,i}\big)_0(t,x)$ is a derivative of the initial datum. Moreover, the kernels $a, b, c$ above are smooth except at points where $1+\hat{p}\cdot\omega=0$, and satisfy
\begin{align}
&(i)\hspace{1.5cm} |a(\omega,\hat{p})|\lesssim (1+\hat{p}\cdot\omega)^{-4},\qquad \int_{|\omega|=1}d\omega a(\omega,\hat{p})=0,\nonumber\\
&(ii)\hspace{1.4cm} |b(\omega,\hat{p})|\lesssim (1+\hat{p}\cdot\omega)^{-3},\label{zero}\\
&(iii)\hspace{1.3cm} |c(\omega,\hat{p})|\lesssim (1+\hat{p}\cdot\omega)^{-2}\nonumber.
\end{align}
The gradient of the magnetic field $\partial_{x_j}B_R^{\varepsilon,i}$ admits a similar representation as \eqref{W1EB}, but
with slightly different kernels $a_B, b_B$ and $c_B$ for which the properties \eqref{W1EB10} and \eqref{zero} remain valid.
\end{lemma}

 Using again the definition of the operator $S$ and the equation satisfied by $F^{\varepsilon}_R$  in \eqref{remain}, we have the following gradient bound estimate.
\begin{proposition}\label{WEBe}
For $t\in [0, \frac{1}{2}|\ln\varepsilon|^{\frac{1}{3}}]$, the gradient of remainders $E^{\varepsilon}_R$ and $B^{\varepsilon}_R$ can be estimated as
\begin{align*}
&\|[\nabla_xE^{\varepsilon}_R(t), \nabla_xB^{\varepsilon}_R(t)]\|_{\infty}\nonumber\\
&\lesssim
|\ln \varepsilon|^{\frac{1}{3}}\int_0^t\|[E^{\varepsilon}_R(\tau), B^{\varepsilon}_R(\tau)]\|_{W^{1,\infty}}\,d\tau+\frac{|\ln \varepsilon|^{\frac{1}{3}}}{\varepsilon}\int_0^t\|h^{\varepsilon}(\tau)\|_{W^{1,\infty}}\, d\tau\\
 &+\varepsilon^{-\frac{17}{8}}|\ln \varepsilon|^2\big(\|[E^{\varepsilon}_R(0), B^{\varepsilon}_R(0)]\|_{\infty}+\|h^{\varepsilon}(0)\|_{\infty}+\varepsilon^k\big)+\varepsilon^{-\frac{29}{8}}|\ln \varepsilon|^2\sup_{0\leq s\leq t}\|f^{\varepsilon}(t)\|.\nonumber
\end{align*}
\end{proposition}
\begin{proof}
We estimate the three integrals in the right hand side of \eqref{W1EB10} one by one.
Noting that
$$\iint\limits_{|x-y|\leq t}\frac{dpdy}{|y-x|^3}a(\omega,\hat{p})F^{\varepsilon}_R(t-|y-x|,y,p)
=\int_0^t\int_{\mathbb S^2}\frac{dpd\omega d\tau}{t-\tau}a(\omega,\hat{p})F^{\varepsilon}_R(\tau, x+(t-\tau)\omega,p),$$
 we use $(i)$ in  \eqref{zero} and \eqref{Linfty} to have
\begin{align*}
&\iint\limits_{|x-y|\leq t}\frac{dpdy}{|y-x|^3}a(\omega,\hat{p})F^{\varepsilon}_R(t-|y-x|,y,p)\nonumber\\
&= \int_0^t\int_{\mathbb{R}^3}\int_{\mathbb S^2}\frac{dpd\omega d\tau}{t-\tau}a(\omega,\hat{p})\left[F^{\varepsilon}_R(\tau, x+(t-\tau)\omega,p)-F^{\varepsilon}_R(\tau, x, p)\right]\\
&\lesssim t \sup_{t\in[0,T]}\|\nabla_xh^{\varepsilon}(t)\|_{\infty}\int_{\mathbb{R}^3}dp (1+|p|)^{8}  (1+|p|)^{-\beta}\sqrt{J_{M}}\nonumber\\
&\lesssim t \sup_{t\in[0,T]}\|\nabla_xh^{\varepsilon}(t)\|_{\infty}.\nonumber
\end{align*}
For the second integral in the right hand side of \eqref{W1EB10}, it can be estimated in almost the same way as $E^{\varepsilon,i}_{R,S}$ in \eqref{LEB2}. The main difference is that we don't integrate by parts w.r.t. $p$. Then we use $(ii)$ in \eqref{zero} to have
\begin{align}
&\Big|\iint\limits_{|x-y|\leq t}\frac{dpdy}{|y-x|^2}b(\omega,\hat{p})(SF^{\varepsilon}_R)(t-|y-x|,y,p)\Big|\nonumber\\
&\lesssim\int_0^t \|[E_R^{\varepsilon}(\tau), B_R^{\varepsilon}(\tau)]\|_{\infty}\, d\tau+\frac{1}{\varepsilon}\int_0^t\|h^{\varepsilon}(\tau)]\|_{W^{1,\infty}}\, d\tau+\varepsilon^k (1+t)^{2k+2}.\nonumber
 \end{align}

Now we estimate the third integral in the right hand side of \eqref{W1EB10}. It holds that
\begin{eqnarray}
&&\iint\limits_{|x-y|\leq t}\frac{dpdy}{|y-x|}c(\omega,\hat{p})(S^2F^{\varepsilon}_R)(t-|y-x|,y,p)\nonumber\\
&&=\iint\limits_{|x-y|\leq t}\frac{dpdy}{|y-x|}c(\omega,\hat{p})(\partial_t+\hat{p}\cdot \nabla_y)\Big[(E_R^{\varepsilon}+\hat{p} \times B_R^{\varepsilon} )\cdot\nabla_pF^{\varepsilon}\Big](t-|y-x|,y,p)\nonumber\\
&&\hspace{0.5cm}+\iint\limits_{|x-y|\leq t}\frac{dpdy}{|y-x|}c(\omega,\hat{p})(\partial_t+\hat{p}\cdot \nabla_y)\Big[\sum_{i=0}^{2k-1}\varepsilon^i(E_i+\hat{p} \times B_i )\cdot\nabla_pF_R^{\varepsilon}\Big](t-|y-x|,y,p)\label{EBW}\\
 &&\hspace{0.5cm}+\frac{1}{\varepsilon}\iint\limits_{|x-y|\leq t}\frac{dpdy}{|y-x|}c(\omega,\hat{p})(\partial_t+\hat{p}\cdot \nabla_y)[Q(F_R^{\varepsilon},F^{\varepsilon})+Q(F^{\varepsilon},F_R^{\varepsilon})](t-|y-x|,y,p)\nonumber\\
& &\hspace{0.5cm}+\varepsilon^{k}\iint\limits_{|x-y|\leq t}\frac{dpdy}{|y-x|}c(\omega,\hat{p})(\partial_t+\hat{p}\cdot \nabla_y)A(t-|y-x|,y,p).\nonumber
\end{eqnarray}
 For the second line in \eqref{EBW}, we integrate w.r.t. $p$ and use the Maxwell system in \eqref{remain} to have
\begin{align}
&\Big|\iint\limits_{|x-y|\leq t}\frac{dpdy}{|y-x|}c(\omega,\hat{p})(\partial_t+\hat{p}\cdot \nabla_x)\Big[(E_R^{\varepsilon}+\hat{p} \times B_R^{\varepsilon} )\cdot\nabla_pF^{\varepsilon}\Big](t-|y-x|,y,p)\Big|\nonumber\\
&\leq\iint\limits_{|x-y|\leq t}\frac{dpdy}{|y-x|}\sum_{i=0}^1|\nabla_p^ic(\omega,\hat{p})|\Big[\big(|\nabla_xE_R^{\varepsilon}|+|\nabla_xB_R^{\varepsilon}|+\Big\|\frac{\sqrt{J_{M}} h^{\varepsilon}}{w}\Big\|_{L^1_p}\big)F^{\varepsilon}\Big](t-|y-x|,y,p)\nonumber\\
&\hspace{0.5cm}+\iint\limits_{|x-y|\leq t}\frac{dpdy}{|y-x|}\sum_{i=0}^1|\nabla_p^ic(\omega,\hat{p})|\Big[(|E_R^{\varepsilon}|+|B_R^{\varepsilon}| )\big(|(\partial_t+\hat{p}\cdot \nabla_x)F^{\varepsilon}|+|\nabla_xF^{\varepsilon}|\big)\Big](t-|y-x|,y,p)\nonumber\\
&\leq \int_0^t(t-\tau)\Big(\|[\nabla_xE^{\varepsilon}_R(\tau), \nabla_xB^{\varepsilon}_R(\tau)]\|_{\infty}+\|h^{\varepsilon}(\tau)\|_{\infty}\Big)\,d\tau \Big\|\sum_{i=0}^1|\nabla_p^ic(\omega,\hat{p})|F^{\varepsilon}\Big\|_{L^1_p}\nonumber\\
&\hspace{0.5cm}+\int_0^t(t-\tau)\|[E^{\varepsilon}_R(\tau), B^{\varepsilon}_R(\tau)]\|_{\infty}\,d\tau \Big\|\sum_{i=0}^1|\nabla_p^ic(\omega,\hat{p})||\nabla_xF^{\varepsilon}|\Big\|_{L^1_p}\nonumber\\
&\hspace{0.5cm}+\int_0^t(t-\tau)\|[E^{\varepsilon}_R(\tau), B^{\varepsilon}_R(\tau)]\|_{\infty}\,d\tau \Big\|\sum_{i=0}^1|\nabla_p^ic(\omega,\hat{p})|(\partial_t+\hat{p}\cdot \nabla_x)F^{\varepsilon}\Big\|_{L^1_p}\nonumber.
\end{align}
On the other hand, similar to \eqref{Fpl} and \eqref{Fpj}, we can use $(iii)$ in \eqref{zero} and the a priori assumptions \eqref{assump} to have
\begin{align}\label{Fplx}
&\Big\|\sum_{i=0}^1|\nabla_p^ic(\omega,\hat{p})|F^{\varepsilon}\Big\|_{L^1_p}\lesssim 1,\qquad\Big\|\sum_{i=0}^1|\nabla_p^ic(\omega,\hat{p})||\nabla_xF^{\varepsilon}|\Big\|_{L^1_p}\lesssim 1,\\
 & \Big\|\sum_{i=0}^1|\nabla_p^ic(\omega,\hat{p})|(\partial_t+\hat{p}\cdot \nabla_x)F^{\varepsilon}\Big\|_{L^1_p}\lesssim 1  \nonumber.
\end{align}
Then, we can further bound the second line in the right hand side of \eqref{EBW} by $$C\int_0^t(t-\tau)\Big(\|[E^{\varepsilon}_R(\tau), B^{\varepsilon}_R(\tau)]\|_{W^{1,\infty}}+\|h^{\varepsilon}(\tau)\|_{\infty}\Big)\,d\tau. $$
For the third line in \eqref{EBW}, we integrate by parts w.r.t. $p$ and use \eqref{decay} and \eqref{growth0} in Theorem \ref{fn} to have
\begin{align*}
&\iint\limits_{|x-y|\leq t}\frac{dpdy}{|y-x|}\sum_{j=0}^1|\nabla_p^jc(\omega,\hat{p})|\Big[\sum_{i=0}^{2k-1}\varepsilon^i
\sum_{j=0}^1\big(|\nabla_{t,y}^jE_i|+|\nabla_{t,y}^jB_i|\big )\nonumber\\
&\hspace{0.5cm}\times \big(|(\partial_t+\hat{p}\cdot \nabla_x)F^{\varepsilon}_R|+|\nabla_xF^{\varepsilon}_R|\big)\Big](t-|y-x|,y,p)\\
&\leq \iint\limits_{|x-y|\leq t}\frac{dpdy}{|y-x|}\sum_{j=0}^1|\nabla_p^jc(\omega,\hat{p})|\big[(1+t-|y-x|)^{-\beta_0}+\sum_{i=0}^{2k-1}\varepsilon^i
(1+t-|y-x|)^i\big]\nonumber\\
&\hspace{0.5cm}\times \Big[\big(|(\partial_t+\hat{p}\cdot \nabla_x)F^{\varepsilon}_R|+|\nabla_xF^{\varepsilon}_R|\big)\Big](t-|y-x|,y,p)\\
&\lesssim \iint\limits_{|x-y|\leq t}\frac{dpdy}{|y-x|}\sum_{j=0}^1|\nabla_p^jc(\omega,\hat{p})|\Big[\big(|(\partial_t+\hat{p}\cdot \nabla_x)F^{\varepsilon}_R|+|\nabla_xF^{\varepsilon}_R|\big)\Big](t-|y-x|,y,p)
\end{align*}
for $t\in [0, \frac{1}{2}|\ln\varepsilon|^{\frac{1}{3}}]$. Therefore, similar to the estimate of the first integral in the right hand side of \eqref{EBW}, we can further obtain the following upper bound of the third line in \eqref{EBW}:
$$C\Big[\int_0^t (t-\tau)\|[E_R^{\varepsilon}(\tau), B_R^{\varepsilon}(\tau)]\|_{\infty}\, d\tau+\frac{1}{\varepsilon}\int_0^t(t-\tau)\|h^{\varepsilon}(\tau)\|_{W^{1,\infty}}\, d\tau+\varepsilon^k (1+t)^{2k+3}\Big].$$
For the fourth line in \eqref{EBW}, we have
\begin{align}
&\frac{1}{\varepsilon}\iint\limits_{|x-y|\leq t}\frac{dpdy}{|y-x|}c(\omega,\hat{p})(\partial_t+\hat{p}\cdot \nabla_x)[Q(F_R^{\varepsilon},F^{\varepsilon})+Q(F^{\varepsilon},F_R^{\varepsilon})](t-|y-x|,y,p)\nonumber\\
=&\frac{1}{\varepsilon}\iint\limits_{|x-y|\leq t}\frac{dpdy}{|y-x|}c(\omega,\hat{p})\Big[Q(\partial_tF_R^{\varepsilon},F^{\varepsilon})+\hat{p}\cdot Q(\nabla_yF_R^{\varepsilon}, F^{\varepsilon})\Big](t-|y-x|,y,p)\nonumber\\
&+\frac{1}{\varepsilon}\iint\limits_{|x-y|\leq t}\frac{dpdy}{|y-x|}c(\omega,\hat{p})\Big[Q(F_R^{\varepsilon},\partial_tF^{\varepsilon})+\hat{p}\cdot Q(F_R^{\varepsilon}, \nabla_yF^{\varepsilon})\Big](t-|y-x|,y,p)\label{diff}\\
&+\frac{1}{\varepsilon}\iint\limits_{|x-y|\leq t}\frac{dpdy}{|y-x|}c(\omega,\hat{p})\Big[Q(\partial_tF^{\varepsilon},F_R^{\varepsilon})+\hat{p}\cdot Q(\nabla_yF^{\varepsilon}, F_R^{\varepsilon})\Big](t-|y-x|,y,p)\nonumber\\
&+\frac{1}{\varepsilon}\iint\limits_{|x-y|\leq t}\frac{dpdy}{|y-x|}c(\omega,\hat{p})\Big[Q(F^{\varepsilon},\partial_tF_R^{\varepsilon})+\hat{p}\cdot Q(F^{\varepsilon}, \nabla_yF_R^{\varepsilon})\Big](t-|y-x|,y,p)\nonumber\\
=&{\mathcal H}_1+{\mathcal H}_2+{\mathcal H}_3+{\mathcal H}_4.\nonumber
\end{align}
Here we denote ${\mathcal H}_i, 1\leq i\leq 4$  as the four integrals in the right hand side of \eqref{diff}.
As the estimation for the these terms are similar, we only deal with ${\mathcal H}_1$ for simplicity. For ${\mathcal H}_1$, we use the equation of \eqref{remain} to have
\begin{align}
{\mathcal H}_1=&\frac{1}{\varepsilon}\iint\limits_{|x-y|\leq t}\frac{dpdy}{|y-x|}c(\omega,\hat{p})\Big[\hat{p}\cdot Q(\nabla_yF_R^{\varepsilon},F^{\varepsilon})-Q(\hat{p}\cdot\nabla_yF_R^{\varepsilon},F^{\varepsilon})\Big](t-|y-x|,y,p)\nonumber\\
&+\frac{1}{\varepsilon}\iint\limits_{|x-y|\leq t}\frac{dpdy}{|y-x|}c(\omega,\hat{p})Q[(E_R^{\varepsilon}+\hat{p} \times B_R^{\varepsilon} )\cdot\nabla_pF^{\varepsilon}, F^{\varepsilon}](t-|y-x|,y,p)\nonumber\\
&+\frac{1}{\varepsilon}\iint\limits_{|x-y|\leq t}\frac{dpdy}{|y-x|}c(\omega,\hat{p})Q[\sum_{i=0}^{2k-1}\varepsilon^i(E_i+\hat{p} \times B_i )\cdot\nabla_pF_R^{\varepsilon}, F^{\varepsilon}](t-|y-x|,y,p)\label{diff1}\\
 &+\frac{1}{\varepsilon^2}\iint\limits_{|x-y|\leq t}\frac{dpdy}{|y-x|}c(\omega,\hat{p})Q[Q(F_R^{\varepsilon},F^{\varepsilon})+Q(F^{\varepsilon},F_R^{\varepsilon}), F^{\varepsilon}](t-|y-x|,y,p)\nonumber\\
 &+\varepsilon^{k-1}\iint\limits_{|x-y|\leq t}\frac{dpdy}{|y-x|}c(\hat{p})Q(A,F^{\varepsilon})(t-|y-x|,y,p).\nonumber
\end{align}
Similar to the estimate of the third line in \eqref{LEB2},
the first and third terms in  the right hand side of \eqref{diff1} can be controlled by $\frac{C}{\varepsilon}\int_0^t(t-\tau)\|h^{\varepsilon}(\tau)\|_{W^{1,\infty}}\, d\tau$.
Note that as in \eqref{Fpl} and \eqref{Fplx},
\begin{align*}
&\sum_{i=0}^1\Big\|\frac{1}{\sqrt{\mathbf{M}}}|\nabla_p^iF ^{\varepsilon}|\Big\|_{L^1_p}\nonumber\\
\lesssim &\int_{\mathbb R^3} \frac{1}{\sqrt{\mathbf{M}}} \sum_{i=0}^1\Big(|\nabla_p^iF_0|+\sum_{j=1}^{2k-1}\varepsilon^j|\nabla_p^iF_j|+\varepsilon^k
\Big|\nabla_p^i\Big(\frac{\sqrt{J_{M}}}{w}|h^{\varepsilon}\Big)\Big|\Big)\, dp\\
\lesssim & 1+\sum_{j=1}^{2k-1}\varepsilon^j(1+t)^j+\varepsilon^k\sup_{t\in [0, \frac{1}{2}|\ln\varepsilon|^{\frac{1}{3}}]}\|h^{\varepsilon}(t)\|_{W^{1,\infty}}\lesssim 1  \nonumber.
\end{align*}
We bound the second line in \eqref{diff1} by
\begin{align*}
&\frac{1}{\varepsilon}\int\limits_{|x-y|\leq t}\frac{dy}{|y-x|}\sup_{p\in\mathbb R^3} |c(\omega,\hat{p})\sqrt{\mathbf{M}}|\Big[\big(|E^{\varepsilon}_R|+|B^{\varepsilon}_R|\big )\Big\|\frac{Q(|\nabla_pF^{\varepsilon}|, F^{\varepsilon})}{\sqrt{\mathbf{M}}}\Big\|_{L^1_p}\Big](t-|y-x|,y,p)\\
&\lesssim \frac{1}{\varepsilon}\int\limits_{|x-y|\leq t}\frac{dy}{|y-x|}\Big[\big(|E^{\varepsilon}_R|+|B^{\varepsilon}_R|\big )\Big\|\frac{|\nabla_pF^{\varepsilon}|}{\sqrt{\mathbf{M}}}\Big\|_{L^1_p}\Big\|\frac{F^{\varepsilon}}{\sqrt{\mathbf{M}}}\Big\|_{L^1_p}\Big](t-|y-x|,y)\nonumber\\
&\lesssim \frac{1}{\varepsilon}\int_0^t(t-\tau)\|[E^{\varepsilon}_R(\tau), B^{\varepsilon}_R(\tau)]\|_{\infty}\,d\tau.
\end{align*}
Similarly, for the fourth line in \eqref{diff1}, we use \eqref{Fpj} to control it by
\begin{align*}
&\frac{1}{\varepsilon^2}\int\limits_{|x-y|\leq t}\frac{dy}{|y-x|}\sup_{p\in\mathbb R^3} |c(\omega,\hat{p})\sqrt{\mathbf{M}}|\\
&\hspace{0.5cm}\times\Big[\Big\|\frac{F^{\varepsilon}}{\sqrt{\mathbf{M}}}\Big\|_{L^1_p}\Big(\Big\|\frac{Q(F_R^{\varepsilon}, F^{\varepsilon})}{\sqrt{\mathbf{M}}}\Big\|_{L^1_p}+\Big\|\frac{Q(F^{\varepsilon}, F_R^{\varepsilon})}{\sqrt{\mathbf{M}}}\Big\|_{L^1_p}\Big)\Big](t-|y-x|,y,p)\\
\lesssim&\frac{1}{\varepsilon^2}\Big\|\frac{F^{\varepsilon}}{\sqrt{\mathbf{M}}}\Big\|^2_{L^1_p} \int\limits_{|x-y|\leq t}\frac{dy}{|y-x|}\|h^{\varepsilon}(t-|y-x|,y)\|_{L^{\infty}_p}\nonumber\\
\lesssim& \frac{1}{\varepsilon^2} \int_0^t(t-\tau)\|h^{\varepsilon}(\tau)\|_{\infty}\,d\tau.
\end{align*}
It is clear to see that the fifth line in \eqref{diff1} is bounded by $\varepsilon^{k-1} (1+t)^{2k+3}$. Then, the upper bound of ${\mathcal H}_1$ and the fourth line in \eqref{EBW} is
\begin{align*}
&C\Big[\frac{1}{\varepsilon}\int_0^t(t-\tau)\|h^{\varepsilon}(\tau)\|_{W^{1,\infty}}\, d\tau+\frac{1}{\varepsilon}\int_0^t(t-\tau)\|[E^{\varepsilon}_R(\tau), B^{\varepsilon}_R(\tau)]\|_{\infty}\,d\tau\Big]\nonumber\\
& +C\Big[\frac{1}{\varepsilon^2} \int_0^t(t-\tau)\|h^{\varepsilon}(\tau)\|_{\infty}\,d\tau+\varepsilon^{k-1} (1+t)^{2k+3}\Big].
\end{align*}
Finally, we straightly have the upper bound $\varepsilon^{k} (1+t)^{2k+3}$ for the fifth line in \eqref{EBW} and obtain that
\begin{eqnarray}
&&\Big|\iint\limits_{|x-y|\leq t}\frac{dpdy}{|y-x|}c(\omega,\hat{p})(S^2F^{\varepsilon}_R)(t-|y-x|,y,p)\Big|\nonumber\\
&&\lesssim \int_0^t(t-\tau)\|[E^{\varepsilon}_R(\tau), B^{\varepsilon}_R(\tau)]\|_{W^{1,\infty}}\,d\tau+\frac{1}{\varepsilon}\int_0^t(t-\tau)\|[E^{\varepsilon}_R(\tau), B^{\varepsilon}_R(\tau)]\|_{\infty}\,d\tau\nonumber\\
 &&\hspace{0.5cm}+\frac{1}{\varepsilon}\int_0^t(t-\tau)\|h^{\varepsilon}(\tau)\|_{W^{1,\infty}}\, d\tau+\frac{1}{\varepsilon^2} \int_0^t(t-\tau)\|h^{\varepsilon}(\tau)\|_{\infty}\,d\tau+\varepsilon^{k-1} (1+t)^{2k+3}.\nonumber
\end{eqnarray}

We collect the estimates above, and use Proposition \ref{linfty} and \eqref{LEB1} to get that for $t\in [0, \frac{1}{2}|\ln \varepsilon|^{\frac{1}{3}}]$,
\begin{align*}
&\|[\nabla_xE^{\varepsilon}_R(t), \nabla_xB^{\varepsilon}_R(t)]\|_{\infty}\nonumber\\
\lesssim &|\ln \varepsilon|^{\frac{1}{3}}\int_0^t\|[E^{\varepsilon}_R(\tau), B^{\varepsilon}_R(\tau)]\|_{W^{1,\infty}}\,d\tau+\frac{|\ln \varepsilon|^{\frac{1}{3}}}{\varepsilon}\int_0^t\|h^{\varepsilon}(\tau)\|_{W^{1,\infty}}\, d\tau\\
 &+\varepsilon^{-\frac{17}{8}}|\ln \varepsilon|^2\big(\|[E^{\varepsilon}_R(0), B^{\varepsilon}_R(0)]\|_{W^{1,\infty}}+\|h^{\varepsilon}(0)\|_{\infty}+\varepsilon^k\big)+\varepsilon^{-\frac{29}{8}}|\ln \varepsilon|^2\sup_{0\leq s\leq t}\|f^{\varepsilon}(t)\|.\nonumber
\end{align*}

\end{proof}

\section{$W^{1,\infty}$ estimates for $h^{\varepsilon}$}\label{w1h}

\setcounter{equation}{0}
In this section, we present the $W^{1,\infty}$ estimates for $h^{\varepsilon}$.

\begin{proposition}\label{W1infty} For given $T\in{ [0,\frac{1}{2}|\ln\varepsilon|^{\frac{1}{3}}]}$, assume the crucial bootstrap assumptions \eqref{assump}. Then, for sufficiently small $\varepsilon$, we have
\begin{align*}
\sup_{0\leq t\leq T}\Big(\varepsilon^{\frac{3}{2}}\|h^{\varepsilon}(t)\|_{W^{1,\infty}}\Big)\lesssim& \varepsilon^{\frac{3}{2}}\|h^{\varepsilon}(0)\|_{W^{1,\infty}}+\sup_{0\leq t\leq T}\Big(\varepsilon^{\frac{5}{2}}\|[E^{\varepsilon}_R(t), B^{\varepsilon}_R(t)]\|_{W^{1,\infty}}\Big)\\
&+\sup_{0\leq t\leq T}\|f^{\varepsilon}(t)\|_{H^1}+\varepsilon^{\frac{2k+3}{2}}.
\end{align*}
\end{proposition}
\begin{proof} With the basic estimates along the curved characteristics in Lemma \ref{charc} and the  nice estimates of $\bar{k}_1$ and $\bar{k}_2$ in Corollary \ref{bar}, the proof can be done in a similar way as the $L^{\infty}$ bound estimate in Subsection 5.1 of \cite{Guo-Jang-CMP-2010}. However, there are two main differences and we explain them now. Firstly, the $W^{1,\infty}$ norm of the electromagnetic field appears in the upper bound of the $W^{1,\infty}$ estimates of $h^{\varepsilon}$. The reason for this is the same as that in Proposition \ref{linfty} and we omit the details for brevity. Secondly,
different from the $W^{1,\infty}$ estimates in \cite{Guo-Jang-CMP-2010}, we don't apply integration by parts to avoid estimating the norms $\|\nabla_{x,p}f^{\varepsilon}(t)\|$ during treating the most difficult term $K_{M,w}(h^{\varepsilon})$, and we directly use the norms $\|\nabla_{x,p}f^{\varepsilon}(t)\|$  instead. Take the $L^{\infty}$ estimate of $D_xh^{\varepsilon}$ for example. In Subsection 5.2 of \cite{Guo-Jang-CMP-2010}, $D_xh^{\varepsilon}$ was written as
\begin{eqnarray}
&&D_xh^{\varepsilon}(t,x,v)\nonumber\\
&&\hspace{0.5cm}=\ldots-\frac{1}{\varepsilon}
\int_0^t\int_0^s\exp\Big\{-\frac{1}{\varepsilon}\int_s^t\nu(\tau)d\tau\Big\}(K_{M,w}D_xh^{\varepsilon})(s, X(s), V(s)) ds+\ldots\nonumber\\
&&\hspace{0.5cm}=\ldots+\frac{1}{\varepsilon^2}
\int_0^t\int_0^{s-\kappa\varepsilon}\exp\Big\{-\frac{1}{\varepsilon}\int_s^t\nu(\tau)d\tau-\frac{1}{\varepsilon}\int_{s_1}^s\nu(\tau)d\tau\Big\}\nonumber\\
&&\hspace{0.5cm}\times\int_{B}l_{N}(V(s),v')l_{N}(V(s_1),v'')D_xh^{\varepsilon}(s_1, X(s_1), v'')dv'dv''ds_1ds+\ldots\nonumber\\
&&\hspace{0.5cm}=\ldots+\frac{1}{\varepsilon^2}
\int_0^t\int_0^{s-\kappa\varepsilon}\exp\Big\{-\frac{1}{\varepsilon}\int_s^t\nu(\tau)d\tau-\frac{1}{\varepsilon}\int_{s_1}^s\nu(\tau)d\tau\Big\}\nonumber\\
&&\hspace{0.5cm}\times\int_{\hat{B}}l_{N}(V(s),v')l_{N}(V(s_1),v'')D_xh^{\varepsilon}(s_1, X(s_1), v'')|\frac{dv'}{dy}|dydv''ds_1ds+\ldots,\label{BB}
\end{eqnarray}
 where $B=\{|v'|\leq 2N, |v''|\leq 3N\}$. Noting that the volume of $B$ satisfies $|B|\lesssim N^6$  and
  $|\frac{dv'}{dy}|\lesssim \varepsilon^{-3}$, the upper bound of  the last expression in \eqref{BB} is
  $$\frac{C_N}{\varepsilon^{3/2}}\|D_xf^{\varepsilon}\|.$$

It should also be pointed out that in our case, we don't need extra weighted $L^{\infty}$ estimate for $h^{\varepsilon}$. In fact, for estimates related to $\bar{K}(h^{\varepsilon})$,  no extra weight arises due to Corollary \ref{bar}.  For a typical term $\frac{w(|p|)}{\sqrt{J_{M}}}D_xQ(\frac{\sqrt{J_{M}}}{w}h^{\varepsilon}, F_i)$,  extra weight for $h^{\varepsilon}$ is not necessary because $J_{M}$ is a global Maxwellian and  $\frac{D_xF_i}{\sqrt{J_{M}}}$ can also be controlled by a Maxwellian. While for
$\frac{w(|p|)}{\sqrt{J_{M}}}D_pQ(\frac{\sqrt{J_{M}}}{w}h^{\varepsilon}, F_i)$, it can be estimated similarly as Lemma 2 and Lemma 5 in \cite{Guo-Strain-CMP-2012} without extra weight for $h^{\varepsilon}$ arises either.
\end{proof}

\section{Auxiliary $H^1$ estimates for $f^{\varepsilon}$}\label{h1f}

\setcounter{equation}{0}
In this part, we continue to perform the $L^2$ energy estimates for the first order derivatives of the remainders  $(f^{\varepsilon}, E^{\varepsilon}_ R, B^{\varepsilon}_ R)$. We first deal with the space derivative estimate.
\begin{proposition}\label{H1x}
For the remainders $(f^{\varepsilon}, E^{\varepsilon}_ R, B^{\varepsilon}_ R)$, it holds that
\begin{equation}
\begin{split}
&\frac{d}{dt}\Big(\Big\|\sqrt{\frac{2T_0}{u^0}}\nabla_xf^{\varepsilon}(t)\Big\|^2+\|[\nabla_xE^{\varepsilon}_ R(t), \nabla_xB^{\varepsilon}_ R(t)]\|^2\Big)\\
&\hspace{1cm}+\Big(\frac{\delta_0T_M}{2\varepsilon}-C\varepsilon^{k-2}\|h^{\varepsilon}\|_{\infty}\Big)
\|\nabla_x{\{\bf I-P\}}f^{\varepsilon}\|^2\\
&\hspace{1cm}\lesssim \left[ (1+t)^{-\beta_0}+\varepsilon^{k}\|h^{\varepsilon}\|_{W^{1,\infty}}+\varepsilon  {\mathcal I}_1(t)\right]
\left(\|f^{\varepsilon}\|^2_{H^1}+\|(E^{\varepsilon}_ R, B^{\varepsilon}_ R)\|^2_{H^1}\right)\label{H1ener0}\\
& \hspace{1cm}+\Big(\frac{1}{\varepsilon}(1+t)^{-2\beta_0} +(1+t)^{-\beta_0}[\varepsilon^{k-1}
\|h^{\varepsilon}\|_{\infty}]^2\Big)\|f^{\varepsilon}\|^2\\
& \hspace{1cm}+ \varepsilon^{\frac{7}{4}}(1+t)^{-\beta_0}\|h^{\varepsilon}\|_{\infty}\|f^{\varepsilon}\|+\varepsilon^{k}{\mathcal I}_2(t)\|\nabla_xf^{\varepsilon}\|.
\end{split}
\end{equation}

\end{proposition}
\begin{proof} We take $D_x$ to \eqref{L20} and \eqref{L201}, and proceed similar derivation as \eqref{L202} to have
\begin{align}
& \frac{1}{2}\frac{d}{dt}\Big(\Big\|\sqrt{\frac{2T_0}{u^0}}D_xf^{\varepsilon}(t)\Big\|^2+\|[D_xE^{\varepsilon}_ R(t), D_xB^{\varepsilon}_ R(t)]\|^2\Big)+\frac{1}{\varepsilon}\langle D_xL{\{\bf I-P\}}f^{\varepsilon}, D_xf^{\varepsilon}\rangle\nonumber\\
&-\Big\langle D_x\Big[\Big(E_0+\hat{p} \times B_0 \Big)\cdot\nabla_pf^{\varepsilon}\Big],\frac{2T_0}{u^0}D_xf^{\varepsilon}\Big\rangle\nonumber\\
&=-\Big\langle\frac{D_xf^{\varepsilon}}{\sqrt{\mathbf{M}}}\Big[\partial_t+\hat{p}\cdot\nabla_x-\Big(E_0+\hat{p} \times B_0 \Big)\cdot\nabla_p\Big]\sqrt{\mathbf{M}},\frac{2T_0}{u^0}D_xf^{\varepsilon}\Big\rangle\nonumber\\
&-\Big\langle f^{\varepsilon}D_x\Big\{\frac{1}{\sqrt{\mathbf{M}}}\Big[\partial_t+\frac{p}{p^{0}}\cdot\nabla_x-(E_0+\hat{p} \times B_0 )\cdot\nabla_p\Big]\sqrt{\mathbf{M}}\Big\},\frac{2T_0}{u^0}D_xf^{\varepsilon}\Big\rangle\nonumber\\
&+\frac{1}{2}\Big\langle\Big[\Big(\partial_t+\hat{p}\cdot\nabla_x\Big)\Big(\frac{2T_0}{u^0}\Big)\Big]D_xf^{\varepsilon}, D_xf^{\varepsilon}\Big\rangle\nonumber\\
 &-\Big\langle D_x\Big[\frac{u^0}{T_0}\hat{p}\cdot E_R^{\varepsilon}\Big], \frac{T_0}{u^0}D_xf^{\varepsilon}\Big\rangle
 +\Big\langle D_x\Big[\frac{u\sqrt{\mathbf{M}}}{T_0} \cdot\Big(\hat{p} \times B_R^{\varepsilon} \Big)\Big],\frac{T_0}{u^0}D_xf^{\varepsilon}\Big\rangle\nonumber\\
 &+\Big\langle D_x\Big[(E_0+\hat{p} \times B_0) \Big]\cdot\nabla_pf^{\varepsilon},\frac{2T_0}{u^0}D_xf^{\varepsilon}\Big\rangle\nonumber\\
 &+\varepsilon^{k-1}\Big\langle D_x\Gamma(f^{\varepsilon},f^{\varepsilon}),\frac{2T_0}{u^0}D_xf^{\varepsilon}\Big\rangle\nonumber\\
 &+\sum_{i=1}^{2k-1}\varepsilon^{i-1}\Big\langle D_x\big[\Gamma( \frac{F_i}{\sqrt{\mathbf{M}}}, f^{\varepsilon})+\Gamma(f^{\varepsilon},  \frac{F_i}{\sqrt{\mathbf{M}}})\big],\frac{2T_0}{u^0}D_xf^{\varepsilon}\Big\rangle\nonumber\\
 &-\varepsilon^k\Big\langle D_x\Big[\frac{1}{T_0}\Big(u^0\hat{p}-u\Big)\cdot\Big(E_R^{\varepsilon}+\hat{p} \times B_R^{\varepsilon}\Big)f^{\varepsilon}\Big],\frac{2T_0}{u^0}D_xf^{\varepsilon}\Big\rangle\label{H100}\\
 &+\varepsilon^k\Big\langle D_x\Big[\Big(E_R^{\varepsilon}+\hat{p} \times B_R^{\varepsilon}\Big)\cdot \nabla_pf^{\varepsilon}\Big],\frac{2T_0}{u^0}D_xf^{\varepsilon}\Big\rangle\nonumber\\
 &+\sum_{i=1}^{2k-1}\varepsilon^i\Big\langle D_x\Big[\Big(E_i+\hat{p} \times B_i \Big)\cdot\nabla_pf^{\varepsilon}\Big],\frac{2T_0}{u^0}D_xf^{\varepsilon}\Big\rangle\nonumber\\
 &+\sum_{i=1}^{2k-1}\varepsilon^i\Big\langle D_x\Big[ \Big(E_R^{\varepsilon}+\hat{p} \times B_R^{\varepsilon} \Big)\cdot\frac{\nabla_pF_i}{\sqrt{\mathbf{M}}}\Big],\frac{2T_0}{u^0}D_xf^{\varepsilon}\Big\rangle\nonumber\\
 &-\sum_{i=1}^{2k-1}\varepsilon^i\Big\langle D_x\Big[\Big(E_i+\hat{p} \times B_i \Big)\cdot\frac{1}{2 T_0}\Big(u^0\hat{p}-u\Big)f^{\varepsilon}\Big],\frac{2T_0}{u^0}D_xf^{\varepsilon}\Big\rangle\nonumber\\
 &+\varepsilon^{k}\Big\langle D_x\bar{A},\frac{2T_0}{u^0}D_xf^{\varepsilon}\Big\rangle,
\nonumber
\end{align}
For brevity, we only estimate the second term in the left hand side of \eqref{H100}, the second, seventh, eighth terms in  the right hand side of \eqref{H100}. The rest of terms can be estimated similarly as these terms or the corresponding terms in Proposition \ref{L2ener}.

We first estimate the second term in the left hand side of \eqref{H100}. Noting
\begin{align*}
&-D_xL( {\{\bf I-P\}}f^{\varepsilon})\nonumber\\
&\hspace{0.5cm}= D_x\Big(\int_{\mathbb R^3}dq\int_{\mathbb S^2}d\omega \frac{s}{p^0q^0}\mathbb B(p,q,\omega)\sqrt{\mathbf{M}}(q)[{\{\bf I-P\}}f^{\varepsilon}(p')\sqrt{\mathbf{M}}(q')\nonumber\\
&\hspace{0.5cm}+{\{\bf I-P\}}f^{\varepsilon}(q')\sqrt{\mathbf{M}}(p')
-{\{\bf I-P\}}f^{\varepsilon}(p)\sqrt{\mathbf{M}}(q)-
{\{\bf I-P\}}f^{\varepsilon}(q)\sqrt{\mathbf{M}}(p)]\Big)\nonumber\\
&\hspace{0.5cm}=\int_{\mathbb R^3}dq\int_{\mathbb S^2}d\omega \frac{s}{p^0q^0}\mathbb B(p,q,\omega)D_x\sqrt{\mathbf{M}}(q)\Big({\{\bf I-P\}}f^{\varepsilon}(p')\sqrt{\mathbf{M}}(q')\nonumber\\
&\hspace{0.5cm}+{\{\bf I-P\}}f^{\varepsilon}(q')\sqrt{\mathbf{M}}(p')
-{\{\bf I-P\}}f^{\varepsilon}(p)\sqrt{\mathbf{M}}(q)-
{\{\bf I-P\}}f^{\varepsilon}(q)\sqrt{\mathbf{M}}(p)\Big)\nonumber\\
&\hspace{0.5cm}+\int_{\mathbb R^3}dq\int_{\mathbb S^2}d\omega \frac{s}{p^0q^0}\mathbb B(p,q,\omega)\sqrt{\mathbf{M}}(q)\Big({\{\bf I-P\}}f^{\varepsilon}(p')[D_x\sqrt{\mathbf{M}}(q')]\nonumber\\
&\hspace{0.5cm}+{\{\bf I-P\}}f^{\varepsilon}(q')[D_x\sqrt{\mathbf{M}}(p')]
-{\{\bf I-P\}}f^{\varepsilon}(p)[D_x\sqrt{\mathbf{M}}(q)]\nonumber\\
&\hspace{0.5cm}-
{\{\bf I-P\}}f^{\varepsilon}(q)[D_x\sqrt{\mathbf{M}}(p)]\Big)-L( D_x{\{\bf I-P\}}f^{\varepsilon})
\end{align*}
by the expression of the operator $L$, and
$$\|{\{\bf I-P\}}D_xf^{\varepsilon}-D_x{\{\bf I-P\}}f^{\varepsilon}\|=\|{\bf P}D_xf^{\varepsilon}-D_x{\bf P}f^{\varepsilon}\|\lesssim \left(\|\nabla_xn_0\|_{\infty}+\|\nabla_xu\|_{\infty}\right)\|{\bf P}f^{\varepsilon}\|,$$
we have
\begin{align*}
&\frac{1}{\varepsilon}\langle D_xL( {\{\bf I-P\}}f^{\varepsilon}), \frac{2T_0}{u^0}D_xf^{\varepsilon}\rangle\nonumber\\
&\hspace{0.5cm}\geq \frac{2\delta_0T_M}{u^0\varepsilon}\|D_x{\{\bf I-P\}}f^{\varepsilon}\|^2-\frac{C}{\varepsilon}\left(\|\nabla_xn_0\|_{\infty}+\|\nabla_xu\|_{\infty}\right)\|{\bf P}f^{\varepsilon}\|\|D_x{\{\bf I-P\}}f^{\varepsilon}\|\nonumber\\
&\hspace{0.5cm}-\frac{C}{\varepsilon}\left(\|\nabla_xn_0\|_{\infty}+\|\nabla_xu\|_{\infty}\right)\|f^{\varepsilon}\|(\|D_x{\{\bf I-P\}}f^{\varepsilon}\|+\|{\{\bf I-P\}}f^{\varepsilon}\|)\nonumber\\
&\hspace{0.5cm}\geq \frac{3\delta_0T_M}{2\varepsilon}\|D_x{\{\bf I-P\}}f^{\varepsilon}\|^2-\frac{C}{\varepsilon}(1+t)^{-2\beta_0}\|f^{\varepsilon}\|^2.\nonumber
\end{align*}
Here we used the integration by parts w.r.t. $x$ when the operator $D_x$ hits the local Maxwellian in the operator $L$.
Note that $D_x^2\{\frac{1}{\sqrt{\mathbf{M}}}[\partial_t+\frac{p}{p^{0}}\cdot\nabla_x-(E_0+\hat{p} \times B_0 )\cdot\nabla_p]\sqrt{\mathbf{M}}\}$ is a cubic polynomial of $p$, $(1+|p|)^3f^{\varepsilon}\leq (1+|p|)^{-5}h^{\varepsilon}$, and
$$\Big(\int_{|p|\geq\sqrt{\frac{\kappa}{\varepsilon}}}(1+|p|)^{-5\times2}dp\Big)^{1/2}\lesssim\Big(\frac{\varepsilon}{\kappa}\Big)^{\frac{7}{4}}.$$
Then, similar to the estimation of \eqref{L200}, the second term in the right hand side of \eqref{H100} can be controlled by
$$ (1+t)^{-\beta_0}\|f^{\varepsilon}\|^2_{H^1}+\frac{\kappa}{\varepsilon}\|{\{\bf I-P\}}D_xf^{\varepsilon}\|^2+\varepsilon^{\frac{7}{4}}(1+t)^{-\beta_0}\|h^{\varepsilon}\|_{\infty}\|f^{\varepsilon}\|.$$
Now we estimate the  seventh term in  the right hand side of \eqref{H100}.
Note that
\begin{align*}
&D_x\Gamma(f^{\varepsilon}, f^{\varepsilon})\nonumber\\
&=D_x\Big(\int_{\mathbb R^3}dq\int_{\mathbb S^2}d\omega \frac{s}{p^0q^0}\mathbb B(p,q,\omega)\sqrt{\mathbf{M}}(q)[f^{\varepsilon}(p')f^{\varepsilon}(q')-f^{\varepsilon}(p)f^{\varepsilon}(q)]\Big)\nonumber\\
&=\Gamma(f^{\varepsilon}, D_xf^{\varepsilon})+\Gamma(D_xf^{\varepsilon}, f^{\varepsilon})\\
&+\int_{\mathbb R^3}dq\int_{\mathbb S^2}d\omega \frac{s}{p^0q^0}\mathbb B(p,q,\omega)D_x\sqrt{\mathbf{M}}(q)[f^{\varepsilon}(p')f^{\varepsilon}(q')-f^{\varepsilon}(p)f^{\varepsilon}(q)].\nonumber
\end{align*}
Though ${\bf M}$ is a local Maxwellian, we can obtain from \eqref{temp} that
\begin{equation}\label{MJx}
\sqrt{\mathbf{M}}+|D_x\sqrt{\mathbf{M}}|\lesssim J_M^{\frac14},
\end{equation}
for $\bar{\varepsilon}_0$ sufficiently small in \eqref{decay}.
In Theorem 2 of \cite{Guo-Strain-CMP-2012}, energy estimates for derivatives of the nonlinear operator $\Gamma$ corresponding to a global Maxwellian were given.
Then, we can use \eqref{MJx}, Theorem 2 in \cite{Guo-Strain-CMP-2012}   to obtain
\begin{align}\label{Hx1000}
&\varepsilon^{k-1}\Big\langle D_x\Gamma(f^{\varepsilon}, f^{\varepsilon}),\frac{2T_0}{u^0}D_xf^{\varepsilon}\Big\rangle\nonumber\\
&\hspace{0.5cm}\leq \varepsilon^{k-1}\Big\langle [\Gamma(f^{\varepsilon}, D_xf^{\varepsilon})+\Gamma(D_xf^{\varepsilon}, f^{\varepsilon})],\frac{2T_0}{u^0}{\{\bf I-P\}}D_xf^{\varepsilon}\Big\rangle\nonumber\\
&\hspace{0.5cm}+C\varepsilon^{k-1}(\|\nabla_xn_0\|_{\infty}+\|\nabla_xu\|_{\infty})\|h^{\varepsilon}\|_{\infty}
\|f^{\varepsilon}\|\|D_xf^{\varepsilon}\|\\
&\hspace{0.5cm}\lesssim\varepsilon^{k-2}\|h^{\varepsilon}\|_{\infty}\|\nabla_x{\{\bf I-P\}}f^{\varepsilon}\|^2+\varepsilon^{k}\|h^{\varepsilon}\|_{\infty}
\|\nabla_xf^{\varepsilon}\|^2\nonumber\\
&\hspace{0.5cm}+(1+t)^{-\beta_0}\|D_xf^{\varepsilon}\|^2+(1+t)^{-\beta_0}[\varepsilon^{k-1}
\|h^{\varepsilon}\|_{\infty}]^2\|f^{\varepsilon}\|^2.\nonumber
\end{align}
Similarly, by  \eqref{growth0} in Theorem \ref{fn}, the eighth term in the right hand side of \eqref{H100} can be estimated by
\begin{align}\label{Hx1001}
 &\sum_{i=1}^{2k-1}\varepsilon^{i-1}\Big\langle [\Gamma(F_i, D_xf^{\varepsilon})+\Gamma(D_x\Gamma(f^{\varepsilon}, \frac{F_i}{\sqrt{\mathbf{M}}})\big],\frac{2T_0}{u^0}{\{\bf I-P\}}D_xf^{\varepsilon}\Big\rangle\nonumber\\
&\hspace{0.5cm}+\sum_{i=1}^{2k-1}\varepsilon^{i-1}\Big\langle [\Gamma(D_x \frac{F_i}{\sqrt{\mathbf{M}}}, f^{\varepsilon})+\Gamma(f^{\varepsilon}, D_xF_i)],\frac{2T_0}{u^0}{\{\bf I-P\}}D_xf^{\varepsilon}\Big\rangle\nonumber\\
&\hspace{0.5cm}+C(\|\nabla_xn_0\|_{\infty}+\|\nabla_xu\|_{\infty})\sum_{i=1}^{2k-1}\varepsilon^{i-1}\|F_i\|_{L^{\infty}_xL^2_p}
\|f^{\varepsilon}\|\|D_xf^{\varepsilon}\|\\
&\hspace{0.5cm}\lesssim\frac{\kappa}{\varepsilon}\|{\{\bf I-P\}}D_xf^{\varepsilon}\|^2+\varepsilon\Big(\sum_{i=1}^{2k-1}\varepsilon^{i-1} (1+t)^{i}\Big)^2
(\|f^{\varepsilon}\|^2+\|\nabla_xf^{\varepsilon}\|^2)\nonumber\\
&\hspace{0.5cm}+\frac{1}{\varepsilon}(1+t)^{-2\beta_0}\|f^{\varepsilon}\|^2.\nonumber
\end{align}
\end{proof}

Next we deal with the momentum derivative estimate for $f^{\varepsilon}$. Due to the Maxwellian structure in the Macroscopic part, it is
enough to perform momentum derivative estimate for ${\{\bf I-P\}}f^{\varepsilon}$. Moreover, as illustrated in the Introduction part, this technique is also necessary for us to make use of the strong dissipation term $\frac{Lf^{\varepsilon}}{\varepsilon}$ and close our energy estimates. Take microscopic projection onto \eqref{L20} to have
\begin{align}
& \partial_t{\{\bf I-P\}}f^{\varepsilon}+\hat{p}\cdot\nabla_x{\{\bf I-P\}}f^{\varepsilon}+{\{\bf I-P\}}\Big[\frac{\sqrt{\mathbf{M}} }{2 T_0}\Big(u^0\hat{p}-u\Big)\cdot\Big(E_R^{\varepsilon}+\hat{p} \times B_R^{\varepsilon} \Big)\Big]\nonumber\\
&-\Big(E_0+\hat{p} \times B_0 \Big)\cdot\nabla_p{\{\bf I-P\}}f^{\varepsilon}+\frac{Lf^{\varepsilon}}{\varepsilon}\nonumber\\
 &=\left({\bf P}\partial_tf^{\varepsilon}-\partial_t{\bf P}f^{\varepsilon}\right)-\frac{{\{\bf I-P\}}f^{\varepsilon}}{\sqrt{\mathbf{M}}}\Big[\partial_t+\hat{p}\cdot\nabla_x-\Big(E_0+\hat{p} \times B_0 \Big)\cdot\nabla_p\Big]\sqrt{\mathbf{M}}\nonumber\\
 &+\varepsilon^{k-1}\Gamma(f^{\varepsilon},f^{\varepsilon})+\sum_{i=1}^{2k-1}\varepsilon^{i-1}\big[\Gamma( \frac{F_i}{\sqrt{\mathbf{M}}}, f^{\varepsilon})+\Gamma(f^{\varepsilon}, \frac{F_i}{\sqrt{\mathbf{M}}})\big]+\varepsilon^k\Big(E_R^{\varepsilon}+\hat{p} \times B_R^{\varepsilon}\Big)\nonumber\\
 &\cdot\nabla_p{\{\bf I-P\}}f^{\varepsilon}-\varepsilon^k\frac{1 }{2 T_0}\Big(u^0\hat{p}-u\Big)\cdot\Big(E_R^{\varepsilon}+\hat{p} \times B_R^{\varepsilon}\Big){\{\bf I-P\}}f^{\varepsilon}\label{Hp0}\\
 &+\sum_{i=1}^{2k-1}\varepsilon^i\Big[\Big(E_i+\hat{p} \times B_i \Big)\cdot\nabla_p{\{\bf I-P\}}f^{\varepsilon}+{\{\bf I-P\}}\Big((E_R^{\varepsilon}+\hat{p} \times B_R^{\varepsilon} )\cdot\nabla_pF_i\Big)\Big]\nonumber\\
 &-\sum_{i=1}^{2k-1}\varepsilon^i\Big[\Big(E_i+\hat{p} \times B_i \Big)\cdot\frac{1}{2 T_0}\Big(u^0\hat{p}-u\Big){\{\bf I-P\}}f^{\varepsilon}\Big]\nonumber\\
 &+\varepsilon^{k}{\{\bf I-P\}}\bar{A}+[[{\bf P},\tau_{E,B}]]f^{\varepsilon},\nonumber
\end{align}
where $[[{\bf P},\tau_{E,B}]]={\bf P}\tau_{E,B}-\tau_{E,B}{\bf P}$ denotes the commutator of two operators ${\bf P}$ and $\tau_{E,B}$
given by
\begin{align}
\tau_{E,B}&=\hat{p}\cdot\nabla_x-\Big(E_0+\hat{p} \times B_0 \Big)\cdot\nabla_p\nonumber\\
&+\frac{1}{\sqrt{\mathbf{M}}}\Big[\partial_t+\hat{p}\cdot\nabla_x-\Big(E_0+\hat{p} \times B_0 \Big)\cdot\nabla_p\Big]\sqrt{\mathbf{M}}\nonumber\\
&-\varepsilon^k\Big(E_R^{\varepsilon}+\hat{p} \times B_R^{\varepsilon}\Big)\cdot\nabla_p+\frac{\varepsilon^k}{2 T_0}\Big(u^0\hat{p}-u\Big)\cdot\Big(E_R^{\varepsilon}+\hat{p} \times B_R^{\varepsilon}\Big)\nonumber\\
&-\sum_{i=1}^{2k-1}\varepsilon^i\Big(E_i+\hat{p} \times B_i \Big)\cdot\Big[\nabla_p-\frac{1 }{2 T_0}\Big(u^0\hat{p}-u\Big)\Big].\nonumber
\end{align}

The momentum derivative $\nabla_p{\{\bf I-P\}}f^{\varepsilon}$ can be estimated as follows:
\begin{proposition}\label{H1p}
For the remainders $(f^{\varepsilon}, E^{\varepsilon}_ R, B^{\varepsilon}_ R)$, it holds that
\begin{align}
&\frac{d}{dt}\|\nabla_p{\{\bf I-P\}}f^{\varepsilon}(t)\|^2+\Big(\frac{\delta_0}{2\varepsilon}-C\varepsilon^{k-2}\|h^{\varepsilon}\|_{\infty}\Big)\|\nabla_p{\{\bf I-P\}}f^{\varepsilon}\|^2\nonumber\\
&\hspace{1cm}\lesssim \left[(1+t)^{-\beta_0}+\varepsilon[1+{\mathcal I}_1(t)]+\varepsilon^{k}\|h^{\varepsilon}\|_{W^{1,\infty}}\right]
\left(\|f^{\varepsilon}\|^2_{H^1}+\|(E^{\varepsilon}_ R, B^{\varepsilon}_ R)\|^2_{H^1}\right)\label{H1p0}\\
& \hspace{1cm}+\frac{1}{\varepsilon}\|{\{\bf I-P\}}f^{\varepsilon}\|^2 +\varepsilon^{\frac{11}{4}}\|h^{\varepsilon}\|_{\infty}\|f^{\varepsilon}\|+\varepsilon^{k}{\mathcal I}_2(t)\|\nabla_pf^{\varepsilon}\|.\nonumber
\end{align}

\end{proposition}
\begin{proof}
 We apply $D_p$ to \eqref{Hp0} and take the $L^2$ inner product with $D_p{\{\bf I-P\}}f^{\varepsilon}$
on both sides to have
\begin{align}
& \frac{1}{2}\frac{d}{dt}\|D_p{\{\bf I-P\}}f^{\varepsilon}(t)\|^2+\frac{1}{\varepsilon}\langle D_p Lf^{\varepsilon}, D_p{\{\bf I-P\}}f^{\varepsilon}\rangle\nonumber\\
&+\langle (D_p \hat{p})\cdot\nabla_x{\{\bf I-P\}}f^{\varepsilon}, D_p{\{\bf I-P\}}f^{\varepsilon}\rangle\nonumber\\
&+\Big\langle D_p{\{\bf I-P\}}\Big[\frac{(u^0\hat{p}-u)}{2T_0}\cdot\Big(E_R^{\varepsilon}+\hat{p} \times B_R^{\varepsilon} \Big)\Big],D_p{\{\bf I-P\}}f^{\varepsilon}\Big\rangle\nonumber\\
&-\Big\langle D_p\Big[\Big(E_0+\hat{p} \times B_0 \Big)\cdot\nabla_p{\{\bf I-P\}}f^{\varepsilon}\Big],D_p{\{\bf I-P\}}f^{\varepsilon}\Big\rangle\nonumber\\
&=-\Big\langle D_p\Big[\frac{{\{\bf I-P\}}f^{\varepsilon}}{\sqrt{\mathbf{M}}}\Big(\partial_t+\hat{p}\cdot\nabla_x-(E_0+\hat{p} \times B_0 )\cdot\nabla_p\Big)\sqrt{\mathbf{M}}\Big],D_p{\{\bf I-P\}}f^{\varepsilon}\Big\rangle\nonumber\\
&+\Big\langle D_p\left({\bf P}\partial_tf^{\varepsilon}-\partial_t{\bf P}f^{\varepsilon}\right),D_p{\{\bf I-P\}}f^{\varepsilon}\Big\rangle+\varepsilon^{k-1}\Big\langle D_p\Gamma(f^{\varepsilon},f^{\varepsilon}),D_p{\{\bf I-P\}}f^{\varepsilon}\Big\rangle\nonumber\\
 &+\sum_{i=1}^{2k-1}\varepsilon^{i-1}\Big\langle D_p
 \big[\Gamma( \frac{F_i}{\sqrt{\mathbf{M}}}, f^{\varepsilon})+\Gamma(f^{\varepsilon},  \frac{F_i}{\sqrt{\mathbf{M}}})\big],D_p{\{\bf I-P\}}f^{\varepsilon}\Big\rangle\nonumber\\
 &+\varepsilon^k\Big\langle \frac{1 }{2 T_0}D_p\Big[(u^0\hat{p}-u)\cdot\Big(E_R^{\varepsilon}+\hat{p} \times B_R^{\varepsilon}\Big){\{\bf I-P\}}f^{\varepsilon}\Big],D_p{\{\bf I-P\}}f^{\varepsilon}\Big\rangle\label{H1p01}\\
 &+\varepsilon^k\Big\langle D_p\Big[\Big(E_R^{\varepsilon}+\hat{p} \times B_R^{\varepsilon}\Big)\cdot\nabla_p{\{\bf I-P\}}f^{\varepsilon}\Big],D_p{\{\bf I-P\}}f^{\varepsilon}\Big\rangle\nonumber\\
 &+\sum_{i=1}^{2k-1}\varepsilon^i\Big\langle D_p\Big[\Big(E_i+\hat{p} \times B_i \Big)\cdot\nabla_p{\{\bf I-P\}}f^{\varepsilon}\Big],D_p{\{\bf I-P\}}f^{\varepsilon}\Big\rangle\nonumber\\
 &+\sum_{i=1}^{2k-1}\varepsilon^i\Big\langle D_p{\{\bf I-P\}}\Big[ \Big(E_R^{\varepsilon}+\hat{p} \times B_R^{\varepsilon} \Big)\cdot\frac{\nabla_pF_i}{\sqrt{\mathbf{M}}}\Big],D_p{\{\bf I-P\}}f^{\varepsilon}\Big\rangle\nonumber\\
 &-\sum_{i=1}^{2k-1}\varepsilon^i\Big\langle D_p\Big[\Big(E_i+\hat{p} \times B_i \Big)\cdot\frac{1}{2 T_0}\Big(u^0\hat{p}-u\Big){\{\bf I-P\}}f^{\varepsilon}\Big],D_p{\{\bf I-P\}}f^{\varepsilon}\Big\rangle\nonumber\\
 &+\varepsilon^{k}\Big\langle D_p{\{\bf I-P\}}\bar{A}, D_p{\{\bf I-P\}}f^{\varepsilon}\Big\rangle+\Big\langle D_p{\{\bf I-P\}}\left([[{\bf P},\tau_{E,B}]]f^{\varepsilon}\right), D_p{\{\bf I-P\}}f^{\varepsilon}\Big\rangle.
\nonumber
\end{align}
For brevity, we only prove the second, fourth terms in the left hand side of \eqref{H1p01} and the first, second, third, fourth terms and the last term in the right hand side.
 We first estimate the second term in the left hand side of \eqref{H1p01}. Similar to \eqref{MJx}, one has
\begin{equation}\label{MJp}
\sqrt{\mathbf{M}}+|D_p\sqrt{\mathbf{M}}|\lesssim J_M^{\frac14}.
\end{equation}
Note that energy estimates for the momentum derivative of the collision frequency $\nu$ and operator $K$, which correspond to a global Maxwellian were  proved in Propositions $7, 8$ \cite{Guo-Strain-CMP-2012} via a splitting technique and different expressions of the relativistic collision operator in the Glassey-Strauss frame and the center of mass frame. We combine \eqref{MJp} and Propositions $7, 8$ in \cite{Guo-Strain-CMP-2012} to estimate  the third term in the left hand side of \eqref{H1p01} as
$$\frac{1}{\varepsilon}\langle D_p Lf^{\varepsilon}, D_p{\{\bf I-P\}}f^{\varepsilon}\rangle\geq\frac{7\delta_0}{8\varepsilon}\|D_p{\{\bf I-P\}}f^{\varepsilon}\|^2-\frac{C}{\varepsilon}\|{\{\bf I-P\}}f^{\varepsilon}\|^2.$$
For the fourth term in the left hand side, it can be estimated by
\begin{align*}
\frac{\delta_0}{8\varepsilon}\|D_p{\{\bf I-P\}}f^{\varepsilon}\|^2+C\varepsilon\left(\|E_R^{\varepsilon}\|^2+\|B_R^{\varepsilon}\|^2\right).
\end{align*}
For the first term in the right hand side, we integrate by parts w.r.t. $p$ when $D_p$ doesn't hit $f^{\varepsilon}$ and bound it by
$$
\frac{\kappa}{\varepsilon}\|D_p{\{\bf I-P\}}f^{\varepsilon}\|^2_{H^1}+\varepsilon^{\frac{11}{4}}\|h^{\varepsilon}\|_{\infty}\|f^{\varepsilon}\|.$$
Noting
$$\|{\bf P}\partial_tf^{\varepsilon}-\partial_t{\bf P}f^{\varepsilon}\|\lesssim (\|\partial_tn_0\|_{\infty}+\|\partial_tu\|_{\infty})\|{\bf P}f^{\varepsilon}\|,$$
we can bound the second term in the right hand side of \eqref{H1p01} by
$$\frac{\delta_0}{8\varepsilon}\|D_p{\{\bf I-P\}}f^{\varepsilon}\|^2+C\varepsilon(1+t)^{-2\beta_0}\|f^{\varepsilon}\|^2.$$
By Theorem 2 in \cite{Guo-Strain-CMP-2012} and \eqref{MJp}, we can bound the third term in the right hand side  by
\begin{align*}
&\varepsilon^{k-2}\|h^{\varepsilon}\|_{\infty}\|\nabla_p{\{\bf I-P\}}f^{\varepsilon}\|^2+C\varepsilon^{k}\|h^{\varepsilon}\|_{\infty}
\|f^{\varepsilon}\|^2_{H^1}.\nonumber
\end{align*}
Similarly, we control the fourth term in the right hand side by
\begin{align*}
&\frac{\delta_0}{8\varepsilon}\|D_p{\{\bf I-P\}}f^{\varepsilon}\|^2+C\varepsilon\Big[\sum_{i=1}^{2k-1}\varepsilon^{i-1}
\left(\Big\|\frac{F_i}{\bf M}\Big\|_{L^{\infty}_xL^2_p}+\Big\|\frac{D_pF_i}{\bf M}\Big\|_{L^{\infty}_xL^2_p}\right)\Big]^2\|f^{\varepsilon}\|^2_{H^1}\\
&\hspace{0.5cm} \leq
\frac{\delta_0}{8\varepsilon}\|D_p{\{\bf I-P\}}f^{\varepsilon}\|^2+C\varepsilon\Big(\sum_{i=1}^{2k-1}\varepsilon^{i-1} (1+t)^{i}\Big)^2
\|f^{\varepsilon}\|^2_{H^1}.
\end{align*}
Similar to the above estimates, the upper bound of the  last term in the right hand side is
$$\frac{\delta_0}{8\varepsilon}\|D_p{\{\bf I-P\}}f^{\varepsilon}\|^2+\Big[\varepsilon[1+{\mathcal I}_1(t)]+\varepsilon^{k}\|h^{\varepsilon}\|_{\infty}\Big]
\left(\|f^{\varepsilon}\|^2_{H^1}+\|(E^{\varepsilon}_ R, B^{\varepsilon}_ R)\|^2_{H^1}\right).$$
\end{proof}
Finally, we multiply \eqref{L2ener0} in Proposition \ref{L2ener} by a sufficiently large constant $C_0$, combine it and
\eqref{H1ener0} in Proposition \ref{H1x}, \eqref{H1p0} in Proposition \ref{H1p}, and use \eqref{assump} to obtain
\begin{proposition}\label{TH1}
Let $0\leq t\leq T\leq \varepsilon^{-\frac{1}{2}}$. Under the assumptions \eqref{assump}, one has
\begin{align}
&\frac{d}{dt}\Big[C_0\Big(\Big\|\sqrt{\frac{2T_0}{u^0}}f^{\varepsilon}(t)\Big\|^2+\|[E^{\varepsilon}_ R(t), B^{\varepsilon}_ R(t)]\|^2\Big)+\Big\|\sqrt{\frac{2T_0}{u^0}}\nabla_xf^{\varepsilon}(t)\Big\|^2\nonumber\\
&\hspace{0.5cm}+\|[\nabla_xE^{\varepsilon}_ R(t), \nabla_xB^{\varepsilon}_ R(t)]\|^2+\|\nabla_p{\{\bf I-P\}}f^{\varepsilon}(t)\|^2\Big]\nonumber\\
&\hspace{0.5cm}\lesssim \left[(1+t)^{-\beta_0}+\varepsilon{\mathcal I}_1(t)+\varepsilon^{k}\|h^{\varepsilon}\|_{W^{1,\infty}}\right]
\left(\|f^{\varepsilon}\|^2_{H^1}+\|(E^{\varepsilon}_ R, B^{\varepsilon}_ R)\|^2_{H^1}\right)\label{H1total}\\
&\hspace{0.5cm}+\frac{1}{\varepsilon}(1+t)^{-2\beta_0}\|f^{\varepsilon}\|^2 +(1+t)^{-\beta_0}[\varepsilon^{k-1}\|h^{\varepsilon}\|_{\infty}]^2\|f^{\varepsilon}\|^2\nonumber\\
&\hspace{0.5cm}+[\varepsilon^{\frac{7}{4}}(1+t)^{-\beta_0}+\varepsilon^{\frac{11}{4}}]\|h^{\varepsilon}\|_{\infty}\|f^{\varepsilon}\|+\varepsilon^{k}
{\mathcal I}_2(t)\|f^{\varepsilon}\|_{H^1}.\nonumber
\end{align}

\end{proposition}

\section{Proof of the main result}\label{mr}

\setcounter{equation}{0}
This section is devoted to the proof of  Theorem \ref{result}. \\
\begin{proof}
 For $0\leq t\leq T\leq \frac{1}{2}|\ln\varepsilon|^{\frac{1}{3}}$, one has
 \begin{equation}\label{time}
 \begin{split}
 {\mathcal I}_1(t)&\lesssim\sum_{1\leq i\leq 2k-1}(1+t)^i\varepsilon^{i-1}\Big(\sum_{1\leq i\leq 2k-1}(1+t)^i\varepsilon^{i-1}+1\Big)\lesssim (1+t)^2\lesssim |\ln\varepsilon| ,\\
 \varepsilon^k{\mathcal I}_2(t)&\lesssim\varepsilon^k\sum_{1+2k\leq i+j\leq 4k-2}\varepsilon^{-1-2k}[(1+t)\varepsilon]^{i+j}\\
 &\lesssim\sum_{1+2k\leq i+j\leq 4k-2}\varepsilon^{i+j-k-1}|\ln\varepsilon|^{\frac{i+j}{3}}\lesssim \varepsilon^{k-1} .
 \end{split}
 \end{equation}

 We combine the a priori assumptions \eqref{assump}, \eqref{L2ener0} and Proposition \ref{linfty} to have
\begin{align}
&\frac{d}{dt}\Big(\Big\|\sqrt{\frac{2T_0}{u^0}}f^{\varepsilon}(t)\Big\|^2+\|[E^{\varepsilon}_ R(t), B^{\varepsilon}_ R(t)]\|^2\Big)+\frac{\delta_0T_M}{4\varepsilon}\|{\{\bf I-P\}}f^{\varepsilon}\|^2\nonumber\\
&\hspace{0.5cm}\lesssim \left[(1+t)^{-\beta_0}+\varepsilon|\ln\varepsilon|\right]
\left(\|f^{\varepsilon}\|^2+\|(E^{\varepsilon}_ R, B^{\varepsilon}_ R)\|^2\right)+\varepsilon\|f^{\varepsilon}\|\nonumber.
\end{align}

Then, for $T\leq\frac{1}{2}|\ln\varepsilon|^{\frac{1}{3}}$, Gr\"{o}nwall's inequality yields
\begin{align}\label{ulf}
&\|f^{\varepsilon}(t)\|+\|[E^{\varepsilon}_ R(t), B^{\varepsilon}_ R(t)]\|+1\lesssim \|f^{\varepsilon}(0)\|+\|[E^{\varepsilon}_ R(0), B^{\varepsilon}_ R(0)]\|+1.
\end{align}
Here we used $\int_0^t(1+s)^{-\beta_0}ds\lesssim1$.

Combining  \eqref{H1total} and \eqref{ulf}, we use the a priori assumptions \eqref{assump} again to have
\begin{align}
&\frac{d}{dt}\Big[C_0\Big(\Big\|\sqrt{\frac{2T_0}{u^0}}f^{\varepsilon}(t)\Big\|^2+\|[E^{\varepsilon}_ R(t), B^{\varepsilon}_ R(t)]\|^2\Big)+\Big\|\sqrt{\frac{2T_0}{u^0}}\nabla_xf^{\varepsilon}(t)\Big\|^2\nonumber\\
&\hspace{0.5cm}+\|[\nabla_xE^{\varepsilon}_ R(t), \nabla_xB^{\varepsilon}_ R(t)]\|^2+\|\nabla_p{\{\bf I-P\}}f^{\varepsilon}(t)\|^2\Big]\nonumber\\
&\hspace{0.5cm}\lesssim \left[(1+t)^{-\beta_0}+\varepsilon|\ln \varepsilon|\right]\left(\|f^{\varepsilon}\|^2_{H^1}+\|(E^{\varepsilon}_ R, B^{\varepsilon}_ R)\|^2_{H^1}+\frac{1}{\varepsilon}\right),\nonumber
\end{align}

Noting
\begin{align*}
&\|f^{\varepsilon}\|^2_{H^1}+\|(E^{\varepsilon}_ R, B^{\varepsilon}_ R)\|^2_{H^1}\approx \Big\|\sqrt{\frac{2T_0}{u^0}}f^{\varepsilon}\Big\|^2+\|(E^{\varepsilon}_ R, B^{\varepsilon}_ R)\|^2\nonumber\\
&\hspace{0.5cm}+\Big\|\sqrt{\frac{2T_0}{u^0}}\nabla_xf^{\varepsilon}\Big\|^2+\|(\nabla_xE^{\varepsilon}_ R, \nabla_xB^{\varepsilon}_ R)\|^2+\|\nabla_p{\{\bf I-P\}}f^{\varepsilon}\|^2,
\end{align*}
we further apply Gr\"{o}nwall's inequality to obtain that for $t\leq \frac{1}{2}|\ln\varepsilon|^{\frac{1}{3}}$,
\begin{align}
&\|f^{\varepsilon}(t)\|_{H^1}
+\|[E^{\varepsilon}_ R(t), B^{\varepsilon}_ R(t)]\|_{H^1}+\frac{1}{\sqrt{\varepsilon}}\lesssim
\|f^{\varepsilon}(0)\|_{H^1}
+\|[E^{\varepsilon}_ R(0), B^{\varepsilon}_ R(0)]\|_{H^1}+\frac{1}{\sqrt{\varepsilon}}.\label{result12}
\end{align}
With the uniform $L^2$ estimates \eqref{ulf} and \eqref{result12} in hands, now we turn to the $L^{\infty}$ estimates of the remainders $h^{\varepsilon}, E_R^{\varepsilon}, B_R^{\varepsilon}$. From \eqref{LEB1} and \eqref{ulf}, we have
\begin{align}
\|[E^{\varepsilon}_R(t), B^{\varepsilon}_R(t)]\|_{\infty}\lesssim &\varepsilon^{-\frac{9}{8}}|\ln \varepsilon|\big(\|[E^{\varepsilon}_R(0), B^{\varepsilon}_R(0)]\|_{\infty}+\|h^{\varepsilon}(0)\|_{\infty}\big)\nonumber\\
&+\varepsilon^{-\frac{21}{8}}|\ln \varepsilon|\big( \|f^{\varepsilon}(0)\|+\|[E^{\varepsilon}_ R(0), B^{\varepsilon}_ R(0)]\|+1\big),\label{uEB1}
\end{align}
for $t\in [0,\frac{1}{2}|\ln\varepsilon|^{\frac{1}{3}}]$. Applying \eqref{ulf} and \eqref{uEB1} in Proposition \ref{linfty} gives
\begin{align}
\sup_{0\leq t\leq \frac{1}{2}|\ln\varepsilon|^{\frac{1}{3}}}\|h^{\varepsilon}(t)\|_{\infty}\lesssim& \varepsilon^{-\frac{1}{8}}|\ln \varepsilon|\big(\|[E^{\varepsilon}_R(0), B^{\varepsilon}_R(0)]\|_{\infty}+\|h^{\varepsilon}(0)\|_{\infty}\big)\nonumber\\
&+\varepsilon^{-\frac{13}{8}}|\ln \varepsilon|\big( \|f^{\varepsilon}(0)\|+\|[E^{\varepsilon}_ R(0), B^{\varepsilon}_ R(0)]\|+1\big).\nonumber
\end{align}
Next we further derive the uniform $W^{1,\infty}$ norm estimates for the electromagnetic field. Collecting  the estimate in Proposition \ref{W1infty} and \eqref{result12} in  \eqref{TWEB1} gives
\begin{align*}
&\|[E^{\varepsilon}_R(t), B^{\varepsilon}_R(t)]\|_{W^{1,\infty}}\nonumber\\
&\leq C_2\Big[
|\ln \varepsilon|^{\frac{1}{3}}\int_0^t\sup_{0\leq\tau\leq s}\|[E^{\varepsilon}_R(\tau), B^{\varepsilon}_R(\tau)]\|_{W^{1,\infty}}\,ds\nonumber\\
&\hspace{0.5cm}+\varepsilon^{-\frac{17}{8}}|\ln \varepsilon|^2\big(\|[E^{\varepsilon}_R(0), B^{\varepsilon}_R(0)]\|_{W^{1,\infty}}+\|h^{\varepsilon}(0)\|_{W^{1,\infty}}\big)\\
 &\hspace{0.5cm}+\varepsilon^{-\frac{5}{2}}|\ln \varepsilon|^{\frac{1}{3}}\big(\|f^{\varepsilon}(0)\|_{H^1}
+\|[E^{\varepsilon}_ R(0), B^{\varepsilon}_ R(0)]\|_{H^1}+\varepsilon^{-\frac{1}{2}}\big)\\
 &\hspace{0.5cm}+\varepsilon^{-\frac{29}{8}}|\ln \varepsilon|^2\big(\|f^{\varepsilon}(0)\|
+\|[E^{\varepsilon}_ R(0), B^{\varepsilon}_ R(0)]\|+1\big)\Big]\nonumber
\end{align*}
for some constant $C_2>1$ and $t\in [0,\frac{1}{2}|\ln\varepsilon|^{\frac{1}{3}}]$. Then, for $\varepsilon$
sufficiently small, we have $C_2\leq \frac{1}{2}|\ln\varepsilon|^{\frac{1}{3}}$ and further apply Gr\"{o}nwall's inequality to get
\begin{align*}
&\sup_{0\leq t\leq \frac{1}{2}|\ln\varepsilon|^{\frac{1}{3}}}\|[E^{\varepsilon}_R(t), B^{\varepsilon}_R(t)]\|_{W^{1,\infty}}\nonumber\\
&\leq C_2\Big[\varepsilon^{-\frac{9}{4}}|\ln \varepsilon|^2\big(\|[E^{\varepsilon}_R(0), B^{\varepsilon}_R(0)]\|_{W^{1,\infty}}+\|h^{\varepsilon}(0)\|_{W^{1,\infty}}\big)\\
 &\hspace{0.5cm}+\varepsilon^{-\frac{21}{8}}|\ln \varepsilon|\big(\|f^{\varepsilon}(0)\|_{H^1}
+\|[E^{\varepsilon}_ R(0), B^{\varepsilon}_ R(0)]\|_{H^1}\big)\\
&\hspace{0.5cm}+\varepsilon^{-\frac{15}{4}}|\ln \varepsilon|^2\big(\|f^{\varepsilon}(0)\|
+\|[E^{\varepsilon}_ R(0), B^{\varepsilon}_ R(0)]\|+1\big)\Big]\nonumber
\end{align*}
Finally, we use this estimate in Proposition \ref{W1infty} to have
\begin{align*}
\sup_{0\leq t\leq \frac{1}{2}|\ln\varepsilon|^{\frac{1}{3}}}\|h^{\varepsilon}(t)\|_{W^{1,\infty}}\lesssim& \varepsilon^{-\frac{5}{4}}|\ln \varepsilon|^2\big(\|[E^{\varepsilon}_R(0), B^{\varepsilon}_R(0)]\|_{W^{1,\infty}}+\|h^{\varepsilon}(0)\|_{W^{1,\infty}}\big)\\
 &+\varepsilon^{-\frac{13}{8}}|\ln \varepsilon|\big(\|f^{\varepsilon}(0)\|_{H^1}
+\|[E^{\varepsilon}_ R(0), B^{\varepsilon}_ R(0)]\|_{H^1}\big)\\
&+\varepsilon^{-\frac{11}{4}}|\ln \varepsilon|^2\big(\|f^{\varepsilon}(0)\|
+\|[E^{\varepsilon}_ R(0), B^{\varepsilon}_ R(0)]\|+1\big).
\end{align*}
We collect  above estimates together
to obtain \eqref{theo1}. Note that \eqref{theo1} implies the assumptions in \eqref{assump} and these assumptions can be verified by a continuity argument for all time $t\leq T\leq \frac{1}{2}|\ln\varepsilon|^{\frac{1}{3}}$.
\end{proof}

%
%
\section {Appendix}\label{App}

\setcounter{equation}{0}
Our Appendix is about three problem: the derivation of the kernel of operator  $K_{\mathbf{M}}$ in \eqref{Koper}, estimates of the kernels related to $K_1$ and $K_2$, and the construction and regularity estimates of the coefficients in the Hilbert expansion \eqref{expan}.

\subsection {Appendix 1: Derivation of the kernel $k_2$}\label{App1}

In this part, we derive the kernel $k_2$ of the operator $K_2$. For $K_i, (i=1, 2)$ given in \eqref{Koper}, their kernels are defined as
$$K_i(f)=\int_{\mathbb{R}^3}dq k_i(p,q)f(q),\qquad i=1,2.$$
Following the derivation of $k_2(p,q)$ for a global Maxwellian case in the Appendix of \cite{Strain-CMP-2010}, we can obtain the following expression of $k_2(p,q)$ for our case:
\begin{equation}\label{k2}
k_2(p,q)=\frac{C_1s^{\frac{3}{2}}}{gp^0q^0}U_1(p,q)\exp\{-U_2(p,q)\},
\end{equation}
where $C_1$ is some positive constant,  $U_2(p,q)$ and $U_1(p,q)$ are smooth functions satisfying
\begin{align*}
U_2(p,q)&=[1+O(1)|u|]\frac{\sqrt{s}|p-q|}{2T_0g},\\
U_1(p,q)&=\frac{1}{U_2(p,q)}+\frac{(p^0+q^0)[1+O(1)|u|]}{2[U_2(p,q)]^2}+\frac{(p^0+q^0)[1+O(1)|u|]}{2[U_2(p,q)]^3}.
\end{align*}
  The  proof of \eqref{k2} can be proceeded similarly as in the Appendix \cite{Strain-CMP-2010} except for the estimation of additional terms w.r.t. the velocity $u$, which arises due to the local Maxwellian $\bf{M}$.

As in \cite{Strain-CMP-2010}, we first write the operator $K_2$ as
\begin{align*}
K_2(f)=\frac{1}{p^0}\int_{\mathbb R^3}\frac{dq}{q^0}\int_{\mathbb R^3}\frac{dq'}{q'^{0}}\int_{\mathbb R^3}\frac{dp'}{p'^{0}}W\sqrt{\mathbf{M}}(q)[\sqrt{\mathbf{M}}(q')f(p')+\sqrt{\mathbf{M}}(p')f(q')].
  \end{align*}
After the same exchanges of  variables $q, p', q'$ and changes of integration variables as in \cite{Strain-CMP-2010}, we obtain the following form of kernel $k_2$
\begin{align*}
 k_2(p,q)=\frac{C}{p^0q^{0}}\int_{\mathbb R^3}\frac{d\bar{p}}{\bar{p}^0}\delta((q^{\mu}-p^{\mu})\bar{p}_{\mu})\bar{s}
  \exp\Big\{\frac{u^{\mu}\bar{p}_{\mu} }{2T_0}\Big\},
  \end{align*}
where the integration variable $\bar{p}_{\mu}=p'_{\mu}+q'_{\mu}$ and $\bar{s}=\bar{g}^2+4$ with
$$\bar{g}^2=g^2+\frac{1}{2}(p^{\mu}+q^{\mu})(p_{\mu}+q_{\mu}-\bar{p}_{\mu}).$$
 Introduce a Lorentz transformation $\Lambda$ which maps into the center-of-momentum system:
  $$A^{\mu}=\Lambda^{\mu}_{\nu}(p^{\nu}+q^{\nu})=(\sqrt{s},0,0,0), \quad B^{\mu}=-\Lambda^{\mu}_{\nu}(p^{\nu}-q^{\nu})=(0,0,0,g).$$
  Here $\Lambda$ is the following Lorentz transform derived in \cite{Strain-Thesis-2005}:
  \begin{equation*}
\Lambda=(\Lambda^{\mu}_{\nu})=\left(
\begin{array}{cccc}
\frac{p^0+q^0}{\sqrt{s}} & -\frac{p_1+q_1}{\sqrt{s}} & -\frac{p_2+q_2}{\sqrt{s}} &-\frac{p_3+q_3}{\sqrt{s}}\\
\Lambda^{1}_0 & \Lambda^{1}_{1} & \Lambda^{1}_{2} & \Lambda^{1}_3 \\
0 & \frac{(p\times q)_1}{|p\times q|} & \frac{(p\times q)_2}{|p\times q|} & \frac{(p\times q)_3}{|p\times q|}\\
\frac{p^0-q^0}{g} & -\frac{p_1-q_1}{g} & -\frac{p_2-q_2}{g} &-\frac{p_3-q_3}{g}
\end{array}
\right),
\end{equation*}
 where  $\Lambda^{1}_0=\frac{2|p\times q|}{g\sqrt{s}}$ and
 $$\Lambda^{1}_i=\frac{2[p_i(p^0+q^0p^{\mu}q_{\mu})+q_i(q^0+p^0p^{\mu}q_{\mu})]}{g\sqrt{s}|p\times q|}, \quad i=1,2,3.$$
 Define $\bar{u}^\mu=\Lambda^{\mu}_{\nu}u^{\nu}$:
 \begin{align*}
 \bar{u}^0&=\Lambda^{0}_{\nu}u^{\nu}=\frac{(p^0+q^0)u^0}{\sqrt{s}} -\frac{(p+q)\cdot u}{\sqrt{s}},\\
 \bar{u}^1&=\Lambda^{1}_{\nu}u^{\nu}=\frac{2|p\times q|u^0}{g\sqrt{s}} +\frac{2[p(p^0+q^0p^{\mu}q_{\mu})+q(q^0+p^0p^{\mu}q_{\mu})]\cdot u}{g\sqrt{s}|p\times q|},\\
 \bar{u}^2&=\Lambda^{2}_{\nu}u^{\nu}=\frac{(p\times q)\cdot u}{|p\times q|},\\
 \bar{u}^3&=\Lambda^{3}_{\nu}u^{\nu}=\frac{(p^0-q^0)u^0}{g} -\frac{(p-q)\cdot u}{g}.
 \end{align*}
 Then we have
 $$\int_{\mathbb R^3}\frac{d\bar{p}}{\bar{p}^0}\delta((q^{\mu}-p^{\mu})\bar{p}_{\mu})\bar{s}
  \exp\Big\{\frac{u^{\mu}\bar{p}_{\mu} }{2T_0}\Big\}=\int_{\mathbb R^3}\frac{d\bar{p}}{\bar{p}^0}\delta(B^{\mu}\bar{p}_{\mu})\bar{s}_{\Lambda}
  \exp\Big\{\frac{u^{\mu}\bar{p}_{\mu} }{2T_0}\Big\},$$
 where $B^{\mu}\bar{p}_{\mu}=\bar{p}_3g,$
 \begin{align*}
  \bar{s}_{\Lambda}=&\bar{g}^2+4=g^2+\frac{1}{2}A^{\mu}(A_{\mu}-\bar{p}_{\mu})=g^2+\frac{1}{2}\sqrt{s}(\bar{p}^0-\sqrt{s}),\\
 \bar{u}^{\mu}\bar{p}_{\mu}=& \bar{p}^0\Big(-\frac{(p^0+q^0)u^0}{\sqrt{s}} +\frac{(p+q)\cdot u}{\sqrt{s}}\Big)\\
 &+\bar{p}_1\Big(\frac{2|p\times q|u^0}{g\sqrt{s}} +\frac{2[p(p^0+q^0p^{\mu}q_{\mu})+q(q^0+p^0p^{\mu}q_{\mu})]\cdot u}{g\sqrt{s}|p\times q|}\Big)\\
 &+\bar{p}_2\frac{(p\times q)\cdot u}{|p\times q|}+\bar{p}_3\Big(\frac{(p^0-q^0)u^0}{g} -\frac{(p-q)\cdot u}{g}\Big).
 \end{align*}
 Switch to polar coordinates
 $$d\bar{p}=|\bar{p}|^2d|\bar{p}|\sin\phi d\phi d\theta,\quad \bar{p}=|\bar{p}|(\sin\phi \cos\theta, \sin\phi\sin\theta, \cos\phi),$$
and use $\cos\phi=0$ for $\phi=\frac{\pi}{2}$ to rewrite $ k_2(p,q)$ as
 \begin{align*}
&\frac{C}{gp^0q^0}\int_{0}^{2\pi}d\phi\int_0^{\infty}\frac{|\bar{p}|d\bar{p}}{\bar{p}^0}\bar{s}_{\Lambda}
  \exp\Big\{\frac{1}{2T_0}\Big[\bar{p}^0\Big(-\frac{(p^0+q^0)u^0}{\sqrt{s}} +\frac{(p+q)\cdot u}{\sqrt{s}}\Big)\\
&+\Big(\frac{2|p\times q|u^0}{g\sqrt{s}} +\frac{2[p(p^0+q^0p^{\mu}q_{\mu})+q(q^0+p^0p^{\mu}q_{\mu})]\cdot u}{g\sqrt{s}|p\times q|}\Big)|\bar{p}|\cos\phi\\
&+\frac{(p\times q)\cdot u}{|p\times q|}|\bar{p}|\sin\phi\Big]  \Big\}\\
&=\frac{C}{gp^0q^0}\int_{0}^{\infty}\frac{|\bar{p}|d\bar{p}}{\bar{p}^0}\bar{s}_{\Lambda}
  \exp\Big\{\frac{\bar{p}^0}{2T_0}\Big(-\frac{(p^0+q^0)u^0}{\sqrt{s}} +\frac{(p+q)\cdot u}{\sqrt{s}}\Big)\Big\}\\
&\times I_0\Big(\frac{\sqrt{{g^2s|(p\times q)\cdot u|^2+4[|p\times q|^2u^0+(p(p^0+q^0p^{\mu}q_{\mu})+q(q^0+p^0p^{\mu}q_{\mu}))\cdot u]^2}}}{2T_0g\sqrt{s}|p\times q|}|\bar{p}|\Big),
 \end{align*}
 where $I_0$ is the first kind modified Bessel function of index zero:
 $$I_0(y)=\frac{1}{2\pi}\int_0^{2\pi}e^{y\cos\phi}d\phi.$$
 By further changing variables of integration and applying some known integrals as in the Appendix \cite{Strain-CMP-2010}, we obtain \eqref{k2} with the following exact form of $U_2(p,q)$:
 \begin{align}
 4T_0^2sU_2^2=&|(p^0+q^0)u^0-(p+q)\cdot u|^2-\frac{s|(p\times q)\cdot u|^2}{|p\times q|^2}\nonumber\\
 &-\frac{4[|p\times q|^2u^0+(p(p^0+q^0p^{\mu}q_{\mu})+q(q^0+p^0p^{\mu}q_{\mu}))\cdot u]^2}{g^2|p\times q|^2}\nonumber\\
 =&\frac{g^2(p^0+q^0)^2-4|p\times q|^2}{g^2}(u^0)^2+|(p+q)\cdot u|^2-\frac{s|(p\times q)\cdot u|^2}{|p\times q|^2}\label{u2}\\
 &-\frac{2u^0[g^2(p^0+q^0)(p+q)+4(p(p^0+q^0p^{\mu}q_{\mu})+q(q^0+p^0p^{\mu}q_{\mu}))]\cdot u}{g^2}\nonumber\\
 &-\frac{4|(p(p^0+q^0p^{\mu}q_{\mu})+q(q^0+p^0p^{\mu}q_{\mu}))\cdot u|^2}{g^2|p\times q|^2}.\nonumber
 \end{align}
Now we estimate the terms in the right hand side of the second equality in \eqref{u2}. For the first term, we have
\begin{align*}
\begin{aligned}
&g^2(p^0+q^0)^2-4|p\times q|^2\\
&\hspace{0.5cm}=2[(2p^0q^0+2+|p|^2+|q|^2)(p^0q^0-p\cdot q-1)-4|p|^2|q|^2+4(p\cdot q)^2\\
&\hspace{0.5cm}=4[(|p|^2+|q|^2)-(p^0q^0+1)p\cdot q]\\
&\hspace{0.5cm}+2(|p|^2+|q|^2)(p^0q^0-p\cdot q-1)+4(p\cdot q)^2\\
&\hspace{0.5cm}=2(|p|^2+|q|^2)(p^0q^0-p\cdot q+1)-4p\cdot q(p^0q^0-p\cdot q+1)\\
&\hspace{0.5cm}=s|p-q|^2.
\end{aligned}
\end{align*}
It is straightforward to see that the upper bound of the second term is $s|u|^2$.
Noting
\begin{align*}
&g^2(p^0+q^0)(p+q)+4(p(p^0+q^0p^{\mu}q_{\mu})+q(q^0+p^0p^{\mu}q_{\mu}))\\
&\hspace{0.5cm}=2p[(p^0+q^0)(p^0q^0-p\cdot q-1)-2p^0|q|^2+2q^0p\cdot q)]\\
&\hspace{0.5cm}+2q[(p^0+q^0)(p^0q^0-p\cdot q-1)-2q^0|p|^2+2p^0p\cdot q)]\\
&\hspace{0.5cm}=2p(p^0-q^0)(p^0q^0-p\cdot q+1)-2q(p^0-q^0)(p^0q^0-p\cdot q+1)\\
&\hspace{0.5cm}=s(p^0-q^0)(p-q),
\end{align*}
 the fourth term can be bounded by
 $$\frac{2u^0|u|s|p-q|^2}{g^2}.$$
For the corresponding  fifth term, we rewrite its numerator as
\begin{align*}
&4|p(p^0+q^0p^{\mu}q_{\mu})+q(q^0+p^0p^{\mu}q_{\mu})|^2\\
&\hspace{0.5cm}=4|p|^2[(p^0)^2|q|^4-2p^0q^0|q|^2p\cdot q+(q^0)^2(p\cdot q)^2]\\
&\hspace{0.5cm}+4|q|^2[(q^0)^2|p|^4-2p^0q^0|p|^2p\cdot q+(p^0)^2(p\cdot q)^2]\\
&\hspace{0.5cm}+2p\cdot q[p^0q^0|p|^2|q|^2-((p^0)^2|q|^2+(q^0)^2|p|^2)p\cdot q+p^0q^0(p\cdot q)^2]\\
&\hspace{0.5cm}=4|p|^2|q|^2\left\{[(p^0)^2|q|^2+(q^0)^2|p|^2+2p^0q^0p\cdot q]-4p^0q^0p\cdot q\right\}\\
&\hspace{0.5cm}+[|p|^2(q^0)^2+|q|^2(p^0)^2-2(|q|^2(p^0)^2+|p|^2(q^0)^2)](p\cdot q)^2+2p^0q^0(p\cdot q)^3\\
&\hspace{0.5cm}=4|p|^2|q|^2|p^0q-q^0p|^2-(p\cdot q)^2[(p^0)^2|q|^2+(q^0)^2|p|^2-2p^0q^0(p\cdot q)]\\
&\hspace{0.5cm}=4|p\times q|^2|p^0q-q^0p|^2.
\end{align*}
Then it can be bounded by
 $$\frac{4|u|^2|p\times q|^2|p^0q-q^0p|^2}{g^2|p\times q|^2}=\frac{4|u|^2|p^0q-q^0p|^2}{g^2}.$$
On the other hand, we have
\begin{align*}
&s|p- q|^2-|p^0q-q^0p|^2\\
&\hspace{0.5cm}=2(p^0q^0-p\cdot q+1)(|p|^2+|q|^2-2p\cdot q)-|q|^2(p^0)^2-|p|^2(q^0)^2+2p^0q^0p\cdot q\\
&\hspace{0.5cm}=4(p\cdot q)^2-2[(p^0)^2+(q^0)^2+p^0q^0]p\cdot q\\
&\hspace{0.5cm}+2(|p|^2+|q|^2)(p^0q^0+1)-[(p^0)^2|q|^2+(q^0)^2|p|^2]\\
&\hspace{0.5cm}\geq4|p|^2|q|^2-(2|p|^2+2|q|^2+2p^0q^0+4)|p||q|\\
&\hspace{0.5cm}+2(|p|^2+|q|^2)p^0q^0+|p|^2+|q|^2-2|p|^2|q|^2\\
&\hspace{0.5cm}=2(|p|^2+|q|^2-|p||q|)(p^0q^0-|p||q|)+(|p|^2+|q|^2-4|p||q|)\\
&\hspace{0.5cm}\geq3(|p|-|q|)^2
\end{align*}
since $p^0q^0-|p||q|\geq 1.$
We combine the above estimates in \eqref{u2} to get
\begin{align}
 4T_0^2U_2^2(p,q)=[1+O(1)|u|]^2\frac{s|p-q|^2}{g^2}.\nonumber
 \end{align}

%
%
\subsection{Appendix 2: Estimates of kernels}\label{App2}

We first prove some important estimates for kernels $k_1$ and $k_2$. As a corollary, estimates for corresponding kernels of $\bar{K}(h^{\varepsilon})=\frac{w(|p|)\sqrt{{\bf M}}}{\sqrt{J_{M}}}K(\frac{\sqrt{J_{M}}}{w\sqrt{{\bf M}}}h^{\varepsilon})$ will also be established.

\begin{lemma}\label{k12} For $k_2$ given in \eqref{k2} and $i=1,2,3$, it holds that
\begin{equation}\label{kb}
k_2(p,q)\lesssim \frac{1}{p^0|p-q|}\exp\Big\{-\frac{\sqrt{s}|p-q|}{8T_0g}\Big\},
\end{equation}
and
\begin{align}
&\int_{\mathbb{R}^3}dq k_2(p,q)\lesssim \frac{1}{p^0},\nonumber\\
&\int_{\mathbb{R}^3}dq \partial_{x_i} k_2(p,q)\lesssim \frac{1}{p^0},\label{0xp}\\
&\int_{\mathbb{R}^3}dq \partial_{p_i} k_2(p,q)\lesssim \frac{1}{p^0}.\nonumber
\end{align}
It is straightforward to verify that $k_1(p,q)$ also satisfies the same estimates.
\end{lemma}
\begin{proof} We first prove \eqref{kb}.
Note that, by Lemma 3.1 in \cite{Glassey-Strauss-PRIMS-1993}, $s=g^2+4\leq 4p^0q^0+4$ and
$$\frac{|p-q|}{\sqrt{p^0q^0}}\leq g\leq |p-q|.$$
We use the smallness assumption of $|u|$  in \eqref{decay} to  have
\begin{align*}
U_2(p,q)\geq&\frac{\sqrt{s}|p-q|}{4T_0g},\qquad
U_1(p,q)\lesssim \frac{p^0+q^0}{2[U_2(p,q)]^2}.
\end{align*}
Then we can further estimate the kernel $k_2$ as
\begin{align}
 k_2(p,q)\lesssim& \frac{s^{\frac{3}{2}}}{gp^0q^0}(p^0+q^0)\times \frac{g^2}{s|p-q|^2}\exp\{-\frac{\sqrt{s}|p-q|}{4T_0g}\}\nonumber\\
=& \frac{(p^0+q^0)g^2}{p^0q^0|p-q|^3}\frac{\sqrt{s}|p-q|}{g}\exp\{-\frac{\sqrt{s}|p-q|}{4T_0g}\}\nonumber\\
\lesssim& \frac{1+|p-q|+q^0}{p^0q^0|p-q|}\exp\{-\frac{|p-q|}{6T_0}\}\nonumber\\
\lesssim& \frac{1}{p^0|p-q|}\exp\{-\frac{|p-q|}{8T_0}\}\nonumber.
\end{align}
Now we prove \eqref{0xp}. Noting that
$$|p|^2\leq 2|p-q|^2+2|q|^2,\qquad |q|^2\leq 2|p-q|^2+2|p|^2,$$
we get from \eqref{kb} that
\begin{align}
\int_{\mathbb{R}^3}dq k_2(p,q)\lesssim& \int_{\mathbb{R}^3}dq\frac{1}{p^0|p-q|}\exp\Big\{-\frac{|p-q|}{8T_0}\Big\}\lesssim\frac{1}{p^0}.\nonumber
\end{align}
 For the second inequality in \eqref{0xp}, again from \eqref{kb}, we obtain
\begin{align*}
&\int_{\mathbb{R}^3}dq \partial_{x_i} k_2(p,q)\nonumber\\
&\lesssim \int_{\mathbb{R}^3}dq  k_2(p,q) \frac{\sqrt{s}|p-q|}{g} \exp\{-\frac{\sqrt{s}|p-q|}{4T_0g}\}\\
&\lesssim \int_{\mathbb{R}^3}dq  k_2(p,q) \exp\{-\frac{\sqrt{s}|p-q|}{8T_0g}\}\nonumber\\
&\lesssim \frac{1}{p^0}.
\end{align*}
Now we prove the third inequality in \eqref{0xp}. Note that
\begin{align*}
\partial_{p_i}U_2(p,q)&=[1+O(1)|u|]\partial_{p_i}\Big(\frac{\sqrt{s}|p-q|}{2T_0g}\Big)\\
&=[1+O(1)|u|]\Big(\frac{|p-q|\partial_{p_i}s}{4T_0g\sqrt{s}}+\frac{\sqrt{s}(p_i-q_i)}{2T_0g|p-q|}
-\frac{\sqrt{s}|p-q|\partial_{p_i}g}{2T_0g^2}\Big)\\
&=[1+O(1)|u|]\Big[\frac{-4|p-q|}{2T_0g^2\sqrt{s}}\frac{q^0}{g}\Big(\frac{p_i}{p^0}-\frac{q_i}{q^0}\Big)
+\frac{\sqrt{s}(p_i-q_i)}{2T_0g|p-q|}\Big]\\
&\lesssim \frac{s^{\frac{3}{2}}|p-q|^3}{g^3} \frac{q^0}{s^2|p-q|\min\{p_0,q_0\}}+\frac{\sqrt{s}|p-q|}{2T_0g}\frac{1}{|p-q|}\\
&\lesssim \frac{s^{\frac{3}{2}}|p-q|^3}{g^3} \frac{\min\{p_0,q_0\}+|p-q|}{|p-q|\min\{p_0,q_0\}}+\frac{\sqrt{s}|p-q|}{2T_0g}\frac{1}{|p-q|},
\end{align*}
where we have used the following inequality:
$$\frac{p_i}{p^0}-\frac{q_i}{q^0}\leq\frac{|p-q|}{\min\{p_0,q_0\}}.$$
Then we use \eqref{kb} to  obtain
\begin{align*}
&\int_{\mathbb{R}^3}dq \partial_{p_i} k_2(p,q)\\
&\hspace{0.5cm}\lesssim\int_{\mathbb{R}^3}dq \Big(\frac{s^{\frac{3}{2}}|p-q|^3}{g^3} \frac{\min\{p_0,q_0\}+|p-q|}{|p-q|\min\{p_0,q_0\}}+\frac{\sqrt{s}|p-q|}{2T_0g}\frac{1}{|p-q|}\Big)k_2(p,q)\\
&\hspace{0.5cm}\lesssim\int_{\mathbb{R}^3}dq \frac{1}{p^0}\Big(\frac{s^{\frac{3}{2}}|p-q|^3}{g^3}+1\Big) \Big(1+\frac{1}{|p-q|^2}\Big)\exp\Big\{-\frac{\sqrt{s}|p-q|}{8T_0g}\Big\}\\
&\hspace{0.5cm}\lesssim\int_{\mathbb{R}^3}dq  \frac{1}{p^0}\Big(1+\frac{1}{|p-q|^2}\Big)\exp\Big\{-\frac{\sqrt{s}|p-q|}{16T_0g}\Big\}\\
&\hspace{0.5cm}\lesssim \frac{1}{p^0}.
\end{align*}
\end{proof}
According to the definition of $h^{\varepsilon}$ in \eqref{Linfty}, the operator $K$ corresponding to the equation satisfied by $h^{\varepsilon}$ is $\bar{K}(h^{\varepsilon})=\frac{w(|p|)\sqrt{{\bf M}}}{\sqrt{J_{M}}}K_{\bf M}(\frac{\sqrt{J_{M}}}{w\sqrt{{\bf M}}}h^{\varepsilon})$. Correspondingly, we can also define $\bar{K}_i, (i=1,2)$ and kernels $\bar{k}_i$ as follows:
$$\bar{K}_i(f)=\int_{\mathbb{R}^3}dq \bar{k}_i(p,q)f(q)=\int_{\mathbb{R}^3}dq \frac{w(|p|)\sqrt{J_{M}(q)}\sqrt{{\bf M}(p)}}{w(|q|)\sqrt{J_{M}(p)}\sqrt{{\bf M}(q)}}k_i(p,q)f(q),\qquad i=1,2.$$
Namely, $\bar{k}_i(p,q)= \frac{w(|p|)\sqrt{J_{M}(q)}\sqrt{{\bf M}(p)}}{w(|q|)\sqrt{J_{M}(p)}\sqrt{{\bf M}(q)}}k_i(p,q), i=1, 2$. Then $\bar{k}_i(p,q)$ also satisfy the same estimates in Lemma \ref{k12}.
\begin{corollary}\label{bar} For  $i=1,2,3$, it holds that
\begin{equation}\label{kbarb}
\bar{k}_2(p,q)\lesssim \frac{1}{p^0|p-q|}\exp\Big\{-\frac{c_0\sqrt{s}|p-q|}{T_0g}\Big\},
\end{equation}
for some small constant $c_0$, and
\begin{align}
&\int_{\mathbb{R}^3}dq \bar{k}_2(p,q)\lesssim \frac{1}{p^0},\nonumber\\
&\int_{\mathbb{R}^3}dq \partial_{x_i} \bar{k}_2(p,q)\lesssim \frac{1}{p^0},\label{kbar}\\
&\int_{\mathbb{R}^3}dq \partial_{p_i} \bar{k}_2(p,q)\lesssim \frac{1}{p^0}.\nonumber
\end{align}
$\bar{k}_1(p,q)$ also satisfies the same estimates.
\end{corollary}
\begin{proof} In fact, we only need to note that
\begin{align*}
\bar{k}_2(p,q)\leq&\frac{(1+|p|)^{\beta}}{(1+|q|)^{\beta}}\frac{s^{\frac{3}{2}}}{gp^0q^0}U_1(p,q)\\
&\times
\exp\Big\{-[1+O(1)|u|]\frac{\sqrt{s}|p-q|}{2T_0g}+\frac{(T_0-T_M)(p^0-q^0)}{2T_MT_0}\Big\}\\
\lesssim& \frac{s^{\frac{3}{2}}}{gp^0q^0}U_1(p,q) (1+|p-q|)^{\beta}\exp\Big\{[(T_0-2T_M)+O(1)|u|]\frac{\sqrt{s}|p-q|}{2T_0T_Mg}\}\\
\lesssim& \frac{s^{\frac{3}{2}}}{gp^0q^0}U_1(p,q) \exp\Big\{-c_0\frac{\sqrt{s}|p-q|}{T_0T_Mg}\}
\end{align*}
by \eqref{temp}. Then, similar to the proof of Lemma \ref{k12}, we can obtain \eqref{kbarb} and \eqref{kbar}.
\end{proof}

%
%
\subsection{Appendix 3: Construction and estimates of coefficients}\label{App3}

We will present the existence of coefficients $F_n, (1\leq n\leq 2k-1)$, and their momentum and time regularities estimates. For $i\in[1, 2k-1]$, we decompose  $\frac{F_n}{\sqrt{\mathbf{M}}}$ as the sum of macroscopic and microscopic parts:
\begin{equation*}\label{decom}
\begin{aligned}
\frac{F_n}{\sqrt{\mathbf{M}}}=&{\bf P}\Big(\frac{F_n}{\sqrt{\mathbf{M}}}\Big)+{\{\bf I-P\}}\Big(\frac{F_n}{\sqrt{\mathbf{M}}}\Big)\\
=&[a_n(t,x)+b_n(t,x)\cdot p+c_n(t,x) p^0]\sqrt{\mathbf{M}}+{\{\bf I-P\}}\Big(\frac{F_n}{\sqrt{\mathbf{M}}}\Big).
\end{aligned}
\end{equation*}

To obtain the linear system satisfied by the abstract functions $a_{n}(t,x), b_{n}(t,x), c_{n}(t,x)$ and $E_{n}(t,x)$, $B_{n}(t,x)$, we  first derive the explicit expression of the third momentum $T^{\alpha \beta \gamma}[\mathbf{M}]$, $(\alpha,\beta,\gamma\in\{0, 1, 2, 3\})$ :
\begin{equation*}\label{third}
T^{\alpha \beta\gamma}[\mathbf{M}] =
 \int_{\mathbb R^3}     \frac{ p^\alpha
p^\beta p^{\gamma} }{p^0} \mathbf{M}\,   d p.
\end{equation*}
We will first get the expression of $T^{\alpha \beta\gamma}[\mathbf{M}]$ in the rest frame where $(u^0, u^1, u^2, u^3)=(1, 0, 0, 0)$.
For convenience, we denote $T^{\alpha \beta\gamma}[\mathbf{M}]$ as $\bar{T}^{\alpha \beta\gamma}$  in the rest frame, $p$ as $\bar{p}$ in the rest frame, and $-v$ as the velocity of general reference frame relative to the rest frame. Then, the corresponding boost matrix $\bar{\Lambda}$ is
  \begin{equation*}
\bar{\Lambda}=(\bar{\Lambda}^{\mu}_{\nu})=\left(
\begin{array}{cccc}
\displaystyle \tilde{r} & \tilde{r}v_1 & \tilde{r}v_2 & \tilde{r}v_3\\
\displaystyle \tilde{r}v_1 & 1+(\tilde{r}-1)\frac{v_1^2}{|v|^2} & (\tilde{r}-1)\frac{v_1v_2}{|v|^2} & (\tilde{r}-1)\frac{v_1v_3}{|v|^2}\\
\displaystyle \tilde{r}v_2 & (\tilde{r}-1)\frac{v_1v_2}{|v|^2} & 1+(\tilde{r}-1)\frac{v_2^2}{|v|^2} & (\tilde{r}-1)\frac{v_2v_3}{|v|^2}\\
\displaystyle \tilde{r}v_3 & (\tilde{r}-1)\frac{v_1v_3}{|v|^2} & (\tilde{r}-1)\frac{v_2v_3}{|v|^2} &1+(\tilde{r}-1)\frac{v_3^2}{|v|^2}
\end{array}
\right),
\end{equation*}
where $\tilde{r}=u^0, v_i=\frac{u_i}{u^0}$. Noting
 $$p^{\mu}= \bar{\Lambda}^{\mu}_{\nu}\bar{p}^{\nu}, \qquad \frac{dp}{p^0}=\frac{d\bar{p}}{\bar{p}^0},$$
 we have
\begin{equation}\label{abg}
T^{\alpha \beta\gamma}[\mathbf{M}]=\bar{\Lambda}^{\alpha}_{\alpha'}\bar{\Lambda}^{\beta}_{\beta'}\bar{\Lambda}^{\gamma}_{\gamma'}\bar{T}^{\alpha' \beta'\gamma'}.
\end{equation}
Then, we can obtain  the expression of $T^{\alpha \beta\gamma}[\mathbf{M}]$  by the expression of $\bar{T}^{\alpha \beta\gamma}$  and \eqref{abg}. Now we give the expression of $\bar{T}^{\alpha \beta\gamma}$ as follows:

\begin{lemma}\label{rest}
Let $i, j, k\in\{1, 2, 3\}$. For the third momentum $\bar{T}^{\alpha \beta\gamma}$ which corresponds to $T^{\alpha \beta\gamma}[\mathbf{M}]$ in the rest frame, we have
\begin{align}
&\bar{T}^{000}=
\frac{n_0[3K_3(\gamma)+\gamma K_2(\gamma)]}{\gamma K_2(\gamma)},\label{rest000}\\
&\bar{T}^{0ii}=\bar{T}^{ii0}=\bar{T}^{i0i}=
\frac{n_0K_3(\gamma)}{\gamma K_2(\gamma)},\label{rest0ii}\\
&\bar{T}^{\alpha \beta\gamma}=0, \qquad\mbox{if} ~(\alpha,\beta,\gamma)\neq (0,0,0), (0,i,i), (i,i,0), (i,0,i). \label{rest0ij}
\end{align}
\end{lemma}

\begin{proof}
It is straightforward to verify \eqref{rest0ij}  by the symmetry of the variable of integration.  Now we prove \eqref{rest000}. It holds that
\begin{align*}\label{000}
&\bar{T}^{000}= \int_{\mathbb R^3}    (p^0)^2 \frac{n_0 \gamma}{4 \pi  K_2(\gamma)} \exp\{- \gamma\bar{p}_{0}  \} \,   d p.
\end{align*}
As the proof of Proposition 3.3 in \cite{Speck-Strain-CMP-2011}, we let $y=\gamma\bar{p}_{0}$ to get $\bar{p}^0=\frac{y}{\gamma}$, $|\bar{p}|=\frac{\sqrt{y^2-\gamma^2}}{\gamma}$, and
$d|\bar{p}|=\frac{1 }{\gamma}\frac{y dy}{ \sqrt{y^2-\gamma^2}}$. Then we have
\begin{align*}
\bar{T}^{000}=& \int_{\gamma}^{\infty}   \frac{1}{\gamma^2}y^2 \frac{n_0
\gamma}{ K_2(\gamma)} e^{- y} \frac{1}{\gamma^2}(y^2-\gamma^2)\frac{1 }{\gamma}\frac{y dy}{ \sqrt{y^2-\gamma^2}}\\
=& \frac{n_0}{\gamma^4K_2(\gamma)}\int_{\gamma}^{\infty}[(y^2-\gamma^2)^{\frac{3}{2}}+\gamma^2\sqrt{y^2-\gamma^2}]ye^{- y} dy\\
=&\frac{n_0[3K_3(\gamma)+\gamma K_2(\gamma)]}{\gamma K_2(\gamma)}.
\end{align*}
Similarly, \eqref{rest0ii} can be proved as follows:
\begin{align*}
\bar{T}^{0ii}=&c \int_{\mathbb R^3}    \frac{|p|^2}{3} \frac{n_0
 \gamma}{4 \pi K_2(\gamma)} \exp\{- \gamma\bar{p}_{0}  \} \,   d p\\
=& \int_{\gamma}^{\infty}   \frac{1}{3\gamma^2}(y^2-\gamma^2) \frac{n_0 \gamma}{  K_2(\gamma)} e^{- y} \frac{1}{\gamma^2}(y^2-\gamma^2)\frac{1 }{\gamma}\frac{y dy}{ \sqrt{y^2-\gamma^2}}\\
=& \frac{n_0}{3\gamma^4K_2(\gamma)}\int_{\gamma}^{\infty}(y^2-\gamma^2)^{\frac{3}{2}}ye^{- y} dy\\
=& \frac{n_0K_3(\gamma)}{\gamma K_2(\gamma)}.
\end{align*}

\end{proof}

Now we can use Lemma \ref{rest} and \eqref{abg} to derive the explicit expression of $T^{\alpha \beta\gamma}[\mathbf{M}]$.

\begin{lemma}\label{general}
For $i, j, k\in\{1, 2, 3\}$, we have
\begin{equation}\label{inert}
\begin{split}
&T^{000}[\mathbf{M}]=
\frac{n_0}{\gamma K_2(\gamma)}[(3K_3(\gamma)+\gamma K_2(\gamma))(u^0)^3+3K_3(\gamma)u^0|u|^2],\\
&T^{00i}[\mathbf{M}]=
\frac{n_0}{\gamma K_2(\gamma)}\left[(5K_3(\gamma)+\gamma K_2(\gamma))(u^0)^2u_i+K_3(\gamma)|u|^2u_i\right],\\
&T^{0ij}[\mathbf{M}]=\frac{n_0}{\gamma K_2(\gamma)}\left[(6K_3(\gamma)+\gamma K_2(\gamma))u^0u_iu_j+\delta_{ij} K_3(\gamma)u^0\right],\\
&T^{ijk}[\mathbf{M}]=\frac{n_0}{\gamma K_2(\gamma)}(6K_3(\gamma)+\gamma K_2(\gamma))u_iu_ju_k \\
&\hspace{2cm}+\frac{n_0K_3(\gamma)}{\gamma K_2(\gamma)}(u_i\delta_{jk}+u_j\delta_{ik}+u_k\delta_{ij}).
\end{split}
\end{equation}
\end{lemma}

\begin{proof}
We first prove $T^{000}[\mathbf{M}]$ in \eqref{inert}. By Lemma \ref{rest} and \eqref{abg}, we have
\begin{align*}
T^{000}[\mathbf{M}]=& \bar{\Lambda}^0_{\alpha}\bar{\Lambda}^0_{\beta}\bar{\Lambda}^0_{\gamma}\bar{T}^{\alpha \beta\gamma}\\
=&\sum_{\alpha=0}\tilde{r}\Big(\tilde{r}^2\frac{n_0[3K_3(\gamma)+\gamma K_2(\gamma)]}{\gamma K_2(\gamma)}+\tilde{r}^2|v|^2\frac{n_0K_3(\gamma)}{\gamma K_2(\gamma)}\Big)\\
&+\sum_{\alpha=1}^3\tilde{r}v_{\alpha}\times 2\tilde{r}^2v_{\alpha}\frac{n_0K_3(\gamma)}{\gamma K_2(\gamma)}\\
=& \frac{n_0}{\gamma K_2(\gamma)}\Big([3K_3(\gamma)+\gamma K_2(\gamma)](u^0)^3+K_3(\gamma)u^0|u|^2\Big)+\frac{2n_0K_3(\gamma)}{\gamma K_2(\gamma)}u^0|u|^2\\
=&\frac{n_0}{\gamma K_2(\gamma)}[(3K_3(\gamma)+\gamma K_2(\gamma))(u^0)^3+3K_3(\gamma)u^0|u|^2].
\end{align*}
For $T^{00i}[\mathbf{M}]$, we have
\begin{align*}
T^{00i}[\mathbf{M}]=& \bar{\Lambda}^0_{\alpha}\bar{\Lambda}^0_{\beta}\bar{\Lambda}^i_{\gamma}\bar{T}^{\alpha \beta\gamma}\\
=&\sum_{\alpha=0}\tilde{r}\Big(\tilde{r}^2v_{i}\frac{n_0[3K_3(\gamma)+\gamma K_2(\gamma)]}{\gamma K_2(\gamma)}\\
&+\sum_{j=1}^3\tilde{r}v_{j}\Big(\frac{\tilde{r}-1}{|v|^2}v_iv_j+\delta_{ij}\Big)\frac{n_0K_3(\gamma)}{\gamma K_2(\gamma)}\Big)\\
&+\sum_{\alpha=1}^3\tilde{r}v_{\alpha}\Big[\tilde{r}^2v_{\alpha}v_i
+\tilde{r}\Big(\frac{\tilde{r}-1}{|v|^2}v_iv_{\alpha}+\delta_{i\alpha}\Big)
\Big]\frac{n_0K_3(\gamma)}{\gamma K_2(\gamma)}\\
=& \frac{n_0}{\gamma K_2(\gamma)}[4K_3(\gamma)+\gamma K_2(\gamma)](u^0)^2u_i+\frac{n_0K_3(\gamma)}{\gamma K_2(\gamma)}u_i[|u|^2+(u^0)^2]\\
=&\frac{n_0}{\gamma K_2(\gamma)}\left[(5K_3(\gamma)+\gamma K_2(\gamma))(u^0)^2u_i+K_3(\gamma)|u|^2u_i\right].
\end{align*}
Similarly, we obtain
\begin{align*}
T^{0ij}[\mathbf{M}]=& \bar{\Lambda}^0_{\alpha}\bar{\Lambda}^i_{\beta}\bar{\Lambda}^j_{\gamma}\bar{T}^{\alpha \beta\gamma}\\
=&\sum_{\alpha=0}\tilde{r}\Big(\tilde{r}^2v_{i}v_{j}\frac{n_0[3K_3(\gamma)+\gamma K_2(\gamma)]}{\gamma K_2(\gamma)}\\
&+\sum_{k=1}^3\Big(\frac{\tilde{r}-1}{|v|^2}v_iv_k+\delta_{ik}\Big)\Big(\frac{\tilde{r}-1}{|v|^2}v_jv_k+\delta_{jk}\Big)\frac{n_0K_3(\gamma)}{\gamma K_2(\gamma)}\Big)\\
&+\sum_{\alpha=1}^3\tilde{r}v_{\alpha}\Big[\tilde{r}\frac{v_{i}}{c}\Big(\frac{\tilde{r}-1}{|v|^2}v_jv_{\alpha}+\delta_{j\alpha}\Big)
+\tilde{r}v_{j}\Big(\frac{\tilde{r}-1}{|v|^2}v_iv_{\alpha}+\delta_{i\alpha}\Big)\Big]\frac{n_0K_3(\gamma)}{\gamma K_2(\gamma)}\\
=& \frac{n_0}{\gamma K_2(\gamma)}\{[4K_3(\gamma)+\gamma K_2(\gamma)]u^0u_iu_j+\delta_{ij}K_3(\gamma)u^0\}+2\frac{n_0K_3(\gamma)}{\gamma K_2(\gamma)}u^0u_iu_j\\
=&\frac{n_0}{\gamma K_2(\gamma)}\left[(6K_3(\gamma)+\gamma K_2(\gamma))u^0u_iu_j+\delta_{ij}K_3(\gamma)u^0\right],
\end{align*}
and
\begin{align*}
T^{ijk}[\mathbf{M}]=& \bar{\Lambda}^i_{\alpha}\bar{\Lambda}^j_{\beta}\bar{\Lambda}^k_{\gamma}\bar{T}^{\alpha \beta\gamma}\\
=&\sum_{\alpha=0}\tilde{r}v_{i}\Big(\tilde{r}^2v_{j}v_{k}\frac{n_0[3K_3(\gamma)+\gamma K_2(\gamma)]}{\gamma K_2(\gamma)}\\
&+\sum_{l=1}^3\Big(\frac{\tilde{r}-1}{|v|^2}v_jv_l+\delta_{jl}\Big)\Big(\frac{\tilde{r}-1}{|v|^2}v_kv_l+\delta_{kl}\Big)
\frac{n_0K_3(\gamma)}{\gamma K_2(\gamma)}\Big)\\
&+\sum_{\alpha=1}^3\Big(\frac{\tilde{r}-1}{|v|^2}v_iv_{\alpha}+\delta_{i\alpha}\Big)
\Big[\tilde{r}v_{j}\Big(\frac{\tilde{r}-1}{|v|^2}v_kv_{\alpha}+\delta_{k\alpha}\Big)\\
&+\tilde{r}v_{k}\Big(\frac{\tilde{r}-1}{|v|^2}v_jv_{\alpha}+\delta_{j\alpha}\Big)\Big]\frac{n_0K_3(\gamma)}{\gamma K_2(\gamma)}\\
=& \frac{n_0}{\gamma K_2(\gamma)}\{[3K_3(\gamma)+\gamma K_2(\gamma)]u_iu_ju_k+u_i(u_ju_k+\delta_{jk})K_3(\gamma)\}\\
&+\frac{n_0K_3(\gamma)}{\gamma K_2(\gamma)}[u_j(u_iu_k+\delta_{ik})+u_k(u_iu_j+\delta_{ij})]\\
=&\frac{n_0}{\gamma K_2(\gamma)}[6K_3(\gamma)+\gamma K_2(\gamma)]u_iu_ju_k \\
&+\frac{n_0K_3(\gamma)}{\gamma K_2(\gamma)}(u_i\delta_{jk}+u_j\delta_{ik}+u_k\delta_{ij}).
\end{align*}

\end{proof}

With the preparation, now we construct the coefficients $(F_n, E_n, B_n), 1\leq n\leq 2k-1$ in a conductive way,  and estimate their regularities.

\begin{theorem}\label{fn} For any $n\in[0,2k-2]$, assume that $(F_i, E_i, B_i)$ have been constructed for all $0\leq i\leq n$. Then  the  microscopic part ${\{\bf I-P\}}\Big(\frac{F_{n+1}}{\sqrt{\mathbf{M}}}\Big)$ can be written as:
\begin{align*}
\begin{aligned}
{\{\bf I-P\}}\Big(\frac{F_{n+1}}{\sqrt{\mathbf{M}}}\Big)=&L^{-1}\Big[-\frac{1}{\sqrt{\mathbf{M}}}\Big(\partial_tF_n
+\hat{p}\cdot \nabla_xF_n-\sum_{\substack{i+j=n+1\\i,j\geq1}}Q(F_i,F_j)\\
 &\hspace{1cm}-\sum_{\substack{i+j=n\\i,j\geq0}}\Big(E_i+\hat{p} \times B_i \Big)\cdot\nabla_pF_j\Big)\Big].
\end{aligned}
\end{align*}
And $a_{n+1}(t,x), b_{n+1}(t,x), c_{n+1}(t,x)$, $E_{n+1}(t,x), B_{n+1}(t,x)$ satisfy the following system:
\begin{align}\label{number}
\begin{aligned}
& \partial_t\Big(n_0u^0a_{n+1}+(e_0+P_0)u^0(u\cdot b_{n+1})+[e_0(u^0)^2+P_0|u|^2]c_{n+1}\Big)\\
&+\nabla_x\cdot\Big(n_0u a_{n+1}+(e_0+P_0)u (u\cdot b_{n+1})+P_0 b_{n+1}+(e_0+P_0)u^0 u c_{n+1}\Big)\\
&+\nabla_x\cdot\int_{\mathbb R^3}  \frac{p}{p^0}\sqrt{\mathbf{M}}{\{\bf I-P\}}\Big(\frac{F_{n+1}}{\sqrt{\mathbf{M}}}\Big)\,dp=0,
\end{aligned}
\end{align}

\begin{align}
& \partial_t\Big((e_0+P_0)u^0u_j a_{n+1}+\frac{n_0}{\gamma K_2(\gamma)}[(6K_3(\gamma)+\gamma K_2(\gamma))u^0 u_j (u\cdot b_{n+1})\nonumber\\
&+K_3(\gamma)u^0 b_{n+1,j}]+\frac{n_0}{\gamma K_2(\gamma)}[(5K_3(\gamma)+\gamma K_2(\gamma))(u^0)^2  +K_3(\gamma)|u|^2]u_jc_{n+1}\Big)\nonumber\\
&+\nabla_x\cdot\Big((e_0+P_0)u_j u a_{n+1}+\frac{n_0}{\gamma K_2(\gamma)}(6K_3(\gamma)+\gamma K_2(\gamma))u_j  u [(u\cdot b_{n+1})+u^0c_{n+1}]\Big)\nonumber\\
&+\partial_{x_j}(P_0a_{n+1})+\nabla_x\cdot\Big[\frac{n_0K_3(\gamma)}{\gamma K_2(\gamma)}(ub_{n+1,j}+u_jb_{n+1})\Big]\nonumber\\
&+ \partial_{x_j}\Big(\frac{n_0K_3(\gamma)}{\gamma K_2(\gamma)}[(u\cdot b_{n+1})+u^0c_{n+1}]\Big)\label{moment}\\
&+E_{0,j}\Big(n_0u^0a_{n+1}+(e_0+P_0)u^0(u\cdot b_{n+1})+[e_0(u^0)^2+P_0|u|^2]c_{n+1}\Big)\nonumber\\
&+\Big[\Big(n_0u a_{n+1}+(e_0+P_0)u (u\cdot b_{n+1})+P_0 b_{n+1}+(e_0+P_0)u^0 u c_{n+1}\Big)\times B_0\Big]_j\nonumber\\
&+n_0u^0E_{n+1,j}+\Big(n_0u\times B_{n+1}\Big)_j\nonumber\\
&+\sum_{\substack{k+l=n+1\\k,l\geq1}}E_{k,j}\Big(n_ku^0_ka_{l}+(e_k+P_k)u^0_k(u_k\cdot b_{l})+P_kb_{l}+[e_k(u^0_k)^2+P_k|u_k|^2]c_{l}\Big)\nonumber\\
&+\sum_{\substack{k+l=n+1\\k,l\geq1}}\Big[\Big(n_ku_k^0 a_{l}+(e_k+P_k)u_k (u_k\cdot b_{l})+(e_k+P_k)u^0_k u_k c_{l}\Big)\times B_k\Big]_j\nonumber\\
&+\nabla_x\cdot\int_{\mathbb R^3}  \frac{p_jp}{p^0}\sqrt{\mathbf{M}}{\{\bf I-P\}}\Big(\frac{F_{n+1}}{\sqrt{\mathbf{M}}}\Big)\,dp+\Big[\int_{\mathbb R^3}  \hat{p}\times B_0\sqrt{\mathbf{M}}{\{\bf I-P\}}\Big(\frac{F_{n+1}}{\sqrt{\mathbf{M}}}\Big)\,dp\Big]_j\nonumber\\
&+\sum_{\substack{k+l=n+1\\k,l\geq1}}\Big[\int_{\mathbb R^3}  \hat{p}\times B_k\sqrt{\mathbf{M}}{\{\bf I-P\}}\Big(\frac{F_{l}}{\sqrt{\mathbf{M}}}\Big)\,dp\Big]_j=0,\nonumber
\end{align}
for $j=1, 2, 3$ with $b_{n+1}=(b_{n+1,1}, b_{n+1,2}, b_{n+1,3})$, $E_{n+1}=(E_{n+1,1}, E_{n+1,2}, E_{n+1,3})$,

\begin{align}
& \partial_t\Big([e_0(u^0)^2+P_0|u|^2] a_{n+1}+\frac{n_0}{\gamma K_2(\gamma)}[(5K_3(\gamma)+\gamma K_2(\gamma)) (u^0)^2+K_3(\gamma)|u|^2]\nonumber\\
&\times (u\cdot b_{n+1})+\frac{n_0}{\gamma K_2(\gamma)}[(3K_3(\gamma)+\gamma K_2(\gamma))(u^0)^2  +3K_3(\gamma)|u|^2]u^0c_{n+1}\Big)\nonumber\\
&+\nabla_x\cdot\Big((e_0+P_0)u^0 u a_{n+1}+\frac{n_0}{\gamma K_2(\gamma)}(6K_3(\gamma)+\gamma K_2(\gamma))u^0 u(u\cdot b_{n+1})\label{energy}\\
&+\frac{n_0K_3(\gamma)}{\gamma K_2(\gamma)}u^0b_{n+1}+\frac{n_0}{\gamma K_2(\gamma)}[(5K_3(\gamma)+\gamma K_2(\gamma))(u^0)^2\nonumber\\
&+K_3(\gamma)|u|^2]uc_{n+1} \Big)+n_0u\cdot E_{n+1}+ n_0u\cdot E_0 a_{n+1}\nonumber\\
&+(e_0+P_0)(u\cdot b_{n+1})(u\cdot E_0)+P_0E_0\cdot b_{n+1}\nonumber\\
&+(e_0+P_0)u^0(u\cdot E_0)c_{n+1}+\int_{\mathbb R^3}  \hat{p}\sqrt{\mathbf{M}}{\{\bf I-P\}}\Big(\frac{F_{n+1}}{\sqrt{\mathbf{M}}}\Big)\,dp \cdot E_0\nonumber\\
&+ \sum_{\substack{k+l=n+1\\k,l\geq1}}\Big[n_ku_k\cdot E_k a_{l}+(e_k+P_k)(u_k\cdot b_{l})(u\cdot E_k)+P_0E_k\cdot b_{l}\nonumber\\
&+(e_k+P_k)u^0_k(u_k\cdot E_k)c_{l}+\int_{\mathbb R^3}  \hat{p}\sqrt{\mathbf{M}}{\{\bf I-P\}}\Big(\frac{F_{l}}{\sqrt{\mathbf{M}}}\Big)\,dp \cdot E_k\Big]=0,\nonumber
\end{align}

\begin{align}
&\partial_tE_{n+1}-\nabla_x \times B_{n+1} \nonumber\\
&\hspace{1cm}=n_0u  a_{n+1}+(e_0+P_0)u (u\cdot b_{n+1})\nonumber\\
&\hspace{1cm}+(e_0+P_0)u^0 u c_{n+1}+\int_{\mathbb R^3}  \hat{p}\sqrt{\mathbf{M}}{\{\bf I-P\}}\Big(\frac{F_{n+1}}{\sqrt{\mathbf{M}}}\Big)\,dp, \nonumber\\
 &\partial_t B_{n+1}+ \nabla_x \times E_{n+1}=0,\label{EM}\\
& \nabla_x\cdot E_{n+1}=-\Big(n_0u^0a_{n+1}+(e_0+P_0)u^0(u\cdot b_{n+1})\nonumber\\
&\hspace{2cm}+[e_0(u^0)^2+P_0|u|^2]c_{n+1}\Big), \nonumber\\
& \nabla_x\cdot  B_{n+1}=0.\nonumber
\end{align}
Furthermore, assume $a_{n+1}(0,x), b_{n+1}(0,x), c_{n+1}(0,x), E_{n+1}(0,x)$, $B_{n+1}(0,x)\in H^N, N\geq0$
be given initial data to the system consisted of equations \eqref{number}, \eqref{moment}, \eqref{energy} and \eqref{EM}. Then the
linear system is well-posed in $C^0([0,\infty);H^N)$. Moreover, it holds that
\begin{align}
&|F_{n+1}|\lesssim (1+t)^{n+1}\mathbf{M}^{1_-},\qquad |\nabla_pF_{n+1}|\lesssim (1+t)^{n+1}\mathbf{M}^{1_-},\nonumber\\
&|\nabla_xF_{n+1}|\lesssim (1+t)^{n}\mathbf{M}^{1_-},\qquad |\nabla_p^2F_{n+1}|\lesssim (1+t)^{n+1}\mathbf{M}^{1_-},\nonumber\\
&|\nabla_x\nabla_pF_{n+1}|\lesssim (1+t)^{n+1}\mathbf{M}^{1_-},\label{growth0}\\
&|E_{n+1}|+|B_{n+1}|+|\nabla_xE_{n+1}|+|\nabla_xE_{n+1}|\lesssim(1+t)^{n+1}.\nonumber
\end{align}

\end{theorem}

\begin{proof}
From the equation of $F_n$ in \eqref{expan2},  the microscopic part ${\{\bf I-P\}}\Big(\frac{F_{n+1}}{\sqrt{\mathbf{M}}}\Big)$ can be written as
\begin{align*}
\begin{aligned}
{\{\bf I-P\}}\Big(\frac{F_{n+1}}{\sqrt{\mathbf{M}}}\Big)=&L^{-1}\Big[-\frac{1}{\sqrt{\mathbf{M}}}\Big(\partial_tF_n+\hat{p}\cdot \nabla_xF_n-\sum_{\substack{i+j=n+1\\i,j\geq1}}Q(F_i,F_j)\\
 &\hspace{1cm}-\sum_{\substack{i+j=n\\i,j\geq0}}\Big(E_i+\hat{p} \times B_i \Big)\cdot\nabla_pF_j\Big)\Big].
\end{aligned}
\end{align*}
We first prove the equation \eqref{number}.
Note that, from \eqref{moments3},
\begin{align}
\int_{\mathbb R^3}    F_{n+1} dp&=\int_{\mathbb R^3}  [a_{n+1}+b_{n+1}\cdot p+c_{n+1} p_0]\mathbf{M} dp\nonumber\\
&=n_0u^0a_{n+1}+(e_0+P_0)u^0(u\cdot b_{n+1})+[e_0(u^0)^2+P_0|u|^2]c_{n+1},\nonumber\\
\int_{\mathbb R^3}  \hat{p}_jF_{n+1}\,dp&=n_0u_j a_{n+1}+(e_0+P_0)u_j (u\cdot b_{n+1})+P_0b_{n+1,j}\label{n01}\\
&+(e_0+P_0)u^0 u_j c_{n+1}+\int_{\mathbb R^3}  \hat{p}_j\sqrt{\mathbf{M}}{\{\bf I-P\}}\Big(\frac{F_{n+1}}{\sqrt{\mathbf{M}}}\Big)\,dp,\nonumber
\end{align}
for $j=1, 2, 3,$ and
$$\int_{\mathbb R^3}  \Big(E_{n+1}+\hat{p} \times B_{n+1} \Big)\cdot\nabla_pF_0dp=\int_{\mathbb R^3}\Big(E_0+\hat{p} \times B_0 \Big)\cdot\nabla_pF_{n+1} dp=0.$$

Then, we integrate the equation of $F_{n+1}$  in \eqref{expan2} w.r.t. $p$ to get \eqref{number}.

Now we prove \eqref{moment}. From \eqref{moments3} and Lemma \ref{general}, it holds that
\begin{align*}
&\int_{\mathbb R^3}  p_j  F_{n+1} dp\\
&\hspace{0.5cm}=\int_{\mathbb R^3} p_j [a_{n+1}+b_{n+1}\cdot p+c_{n+1} p_0]\mathbf{M} dp\\
&\hspace{0.5cm}=(e_0+P_0)u^0u_j a_{n+1}+\frac{n_0}{\gamma K_2(\gamma)}[(6K_3(\gamma)+\gamma K_2(\gamma))u^0 u_j (u\cdot b_{n+1})\\
&\hspace{0.5cm}+K_3(\gamma)u^0 b_{n+1,j}]+\frac{n_0}{\gamma K_2(\gamma)}[(5K_3(\gamma)+\gamma K_2(\gamma))(u^0)^2    +K_3(\gamma)|u|^2]u_jc_{n+1},
\end{align*}
\begin{align*}
&\int_{\mathbb R^3}  \frac{p_jp}{p^0}  F_{n+1} dp\\
&\hspace{0.5cm}=\int_{\mathbb R^3} \frac{p_jp}{p^0} [a_{n+1}+b_{n+1}\cdot p+c_{n+1} p_0]\mathbf{M} dp+\int_{\mathbb R^3}  \frac{p_jp}{p^0}\sqrt{\mathbf{M}}{\{\bf I-P\}}\Big(\frac{F_{n+1}}{\sqrt{\mathbf{M}}}\Big)\,dp\\
&\hspace{0.5cm}=(e_0+P_0)u_j u a_{n+1}+\frac{n_0}{\gamma K_2(\gamma)}(6K_3(\gamma)+\gamma K_2(\gamma))u_j  u [(u\cdot b_{n+1})+u^0c_{n+1}]\nonumber\\
&\hspace{0.5cm}+e_j P_0a_{n+1}+\frac{n_0K_3(\gamma)}{\gamma K_2(\gamma)}(ub_{n+1,j}+u_jb_{n+1})\nonumber\\
&\hspace{0.5cm}+ e_j\frac{n_0K_3(\gamma)}{\gamma K_2(\gamma)}[(u\cdot b_{n+1})+u^0c_{n+1}]+\int_{\mathbb R^3}  \frac{p_jp}{p^0}\sqrt{\mathbf{M}}{\{\bf I-P\}}\Big(\frac{F_{n+1}}{\sqrt{\mathbf{M}}}\Big)\,dp,
\end{align*}
where  $e_j, (j=1,2,3),$ are the unit base vectors in $\mathbb R^3$, and
\begin{align*}
&-\int_{\mathbb R^3}  p_j\Big(E_{n+1}+\hat{p} \times B_{n+1} \Big)\cdot\nabla_pF_0dp\\
&\hspace{0.5cm}=\int_{\mathbb R^3}E_{n+1,j}F_{0}dp+\int_{\mathbb R^3}\Big(\hat{p} \times B_{n+1} \Big)_jF_{0} dp \\
&\hspace{0.5cm} =n_0u^0E_{n+1,j}+\Big(n_0u\times B_{n+1}\Big)_j,\\
&-\int_{\mathbb R^3}p_j\Big(E_0+\hat{p} \times B_0 \Big)\cdot\nabla_pF_{n+1} dp\\
&\hspace{0.5cm} =\int_{\mathbb R^3}E_{0,j}F_{n+1}dp+\int_{\mathbb R^3}\Big(\hat{p} \times B_0 \Big)_jF_{n+1} dp\\
&\hspace{0.5cm} =E_{0,j}\Big(n_0u^0a_{n+1}+(e_0+P_0)u^0(u\cdot b_{n+1})+[e_0(u^0)^2+P_0|u|^2]c_{n+1}\Big)\nonumber\\
&\hspace{0.5cm}+\Big[\Big(n_0u a_{n+1}+(e_0+P_0)u (u\cdot b_{n+1})+P_0b_{n+1}+(e_0+P_0)u^0 u c_{n+1}\Big)\times B_0\Big]_j\nonumber\\
&\hspace{0.5cm}+\int_{\mathbb R^3}\Big(\hat{p} \times B_0 \Big)_j\sqrt{\mathbf{M}}{\{\bf I-P\}}\Big(\frac{F_{n+1}}{\sqrt{\mathbf{M}}}\Big) dp .
\end{align*}
Then, we multiply the equation of $F_{n+1}$  in \eqref{expan2} by $p_j$ and integrate the resulting equation w.r.t. $p$ to get \eqref{moment}.

Next, we show that \eqref{energy} holds. By \eqref{moments3} and Lemma \ref{general}, one has
\begin{align*}
&\int_{\mathbb R^3}  p_0  F_{n+1} dp\\
&\hspace{0.5cm}=\int_{\mathbb R^3} p_0 [a_{n+1}+b_{n+1}\cdot p+c_{n+1} p_0]\mathbf{M} dp\\
&\hspace{0.5cm}=[e_0(u^0)^2+P_0|u|^2] a_{n+1}+\frac{n_0}{\gamma K_2(\gamma)}[(5K_3(\gamma)+\gamma K_2(\gamma)) (u^0)^2+K_3(\gamma)|u|^2]\\
&\hspace{0.5cm}\times (u\cdot b_{n+1})+\frac{n_0}{\gamma K_2(\gamma)}[(3K_3(\gamma)+\gamma K_2(\gamma))(u^0)^2  +3K_3(\gamma)|u|^2]u^0c_{n+1},
\end{align*}
and
\begin{align*}
&-\int_{\mathbb R^3}  p^0\Big(E_{n+1}+\hat{p} \times B_{n+1} \Big)\cdot\nabla_pF_0dp\\
&\hspace{0.5cm} =\int_{\mathbb R^3}E_{n+1}\cdot \hat{p}F_{0}dp=n_0u\cdot E_{n+1},\\
&-\int_{\mathbb R^3}  p^0\Big(E_{0}+\hat{p} \times B_0 \Big)\cdot\nabla_pF_{n+1}dp\\
&\hspace{0.5cm} =\int_{\mathbb R^3} E_{0}\cdot \hat{p}F_{n+1}dp\\
&\hspace{0.5cm} = n_0u\cdot E_0 a_{n+1}+(e_0+P_0)(u\cdot b_{n+1})(u\cdot E_0)+P_0E_0\cdot b_{n+1}\nonumber\\
&\hspace{0.5cm}+(e_0+P_0)u^0(u\cdot E_0)c_{n+1}+\int_{\mathbb R^3}  \hat{p}\sqrt{\mathbf{M}}{\{\bf I-P\}}\Big(\frac{F_{n+1}}{\sqrt{\mathbf{M}}}\Big)\,dp \cdot E_0.
\end{align*}
We integrate the equation of $F_n$  in \eqref{expan2} with $p^0$ over $\mathbb R^3_p$ to obtain \eqref{energy}.
Finally, it is straightforward to obtain the Maxwell system \eqref{EM} of $E_{n+1}, B_{n+1}$  from \eqref{n01}.

Now we prove the well-posedness of the system \eqref{number}, \eqref{moment}, \eqref{energy} and \eqref{EM}.  By conditions \eqref{cdn} and the equation $ \partial_t(n_0 u^0) + \nabla_x\cdot(n_0 u) =0$ in \eqref{order01}, we simplify equations \eqref{number}, \eqref{moment}, \eqref{energy} as follows:
\begin{align}\label{number01}
\begin{aligned}
& n_0u^0\partial_t\Big(a_{n+1}+h(u\cdot b_{n+1})+hu^0c_{n+1}\Big)-\partial_t(P_0c_{n+1})\\
&+n_0u\cdot\nabla_x\Big( a_{n+1}+h (u\cdot b_{n+1})+hu^0  c_{n+1}\Big)+\nabla_x\cdot(P_0b_{n+1})\\
&+\nabla_x\cdot\int_{\mathbb R^3}  \hat{p}\sqrt{\mathbf{M}}{\{\bf I-P\}}\Big(\frac{F_{n+1}}{\sqrt{\mathbf{M}}}\Big)\,dp=0,
\end{aligned}
\end{align}

\begin{align}
& n_0u^0\partial_t\Big(hu_j a_{n+1}+\frac{1}{\gamma K_2(\gamma)}[(6K_3(\gamma)+\gamma K_2(\gamma)) u_j (u\cdot b_{n+1}+u^0  c_{n+1})\nonumber\\
&+K_3(\gamma) b_{n+1,j}]\Big)-\partial_t\Big(\frac{n_0K_3(\gamma)}{\gamma K_2(\gamma)}u_jc_{n+1}\Big)\nonumber\\
&+n_0u\cdot\nabla_x\Big(hu_j a_{n+1}+\frac{1}{\gamma K_2(\gamma)}[(6K_3(\gamma)+\gamma K_2(\gamma))u_j[(u\cdot b_{n+1})+u^0c_{n+1}]\nonumber\\
&+K_3(\gamma) b_{n+1,j}]\Big)+\partial_{x_j}(P_0a_{n+1})+\nabla_x\cdot\Big[\frac{n_0K_3(\gamma)}{\gamma K_2(\gamma)}u_jb_{n+1}\Big]\nonumber\\
&+ \partial_{x_j}\Big(\frac{n_0K_3(\gamma)}{\gamma K_2(\gamma)}[(u\cdot b_{n+1})+u^0c_{n+1}]\Big)\label{moment01}\\
&+E_{0,j}\Big(n_0u^0a_{n+1}+(e_0+P_0)u^0(u\cdot b_{n+1})+[e_0(u^0)^2+P_0|u|^2]c_{n+1}\Big)\nonumber\\
&+\Big[\Big(n_0u a_{n+1}+(e_0+P_0)u (u\cdot b_{n+1})+P_0b_{n+1}+(e_0+P_0)u^0 u c_{n+1}\Big)\times B_0\Big]_j\nonumber\\
&+n_0u^0E_{n+1,j}+\Big(n_0u^0\times B_{n+1}\Big)_j\nonumber\\
&+\sum_{\substack{k+l=n+1\\k,l\geq1}}E_{k,j}\Big(n_0u^0a_{l}+(e_0+P_0)u^0(u\cdot b_{l})+[e_0(u^0)^2+P_0|u|^2]c_{l}\Big)\nonumber\\
&+\sum_{\substack{k+l=n+1\\k,l\geq1}}\Big[\Big(n_0u a_{l}+(e_0+P_0)u (u\cdot b_{l})+P_0b_l+(e_0+P_0)u^0 u c_{l}\Big)\times B_k\Big]_j\nonumber\\
&+\nabla_x\cdot\int_{\mathbb R^3}  \frac{p_jp}{p^0}\sqrt{\mathbf{M}}{\{\bf I-P\}}\Big(\frac{F_{n+1}}{\sqrt{\mathbf{M}}}\Big)\,dp+\Big[\int_{\mathbb R^3}  \hat{p}\times B_0\sqrt{\mathbf{M}}{\{\bf I-P\}}\Big(\frac{F_{n+1}}{\sqrt{\mathbf{M}}}\Big)\,dp\Big]_j\nonumber\\
&+\sum_{\substack{k+l=n+1\\k,l\geq1}}\Big[\int_{\mathbb R^3}  \hat{p}\times B_k\sqrt{\mathbf{M}}{\{\bf I-P\}}\Big(\frac{F_{l}}{\sqrt{\mathbf{M}}}\Big)\,dp\Big]_j=0,\nonumber
\end{align}
and
\begin{align}
& n_0u^0\partial_t\Big(hu^0 a_{n+1}+\frac{1}{\gamma K_2(\gamma)}[(6K_3(\gamma)+\gamma K_2(\gamma)) u^0](u\cdot b_{n+1})\nonumber\\
&+\frac{1}{\gamma K_2(\gamma)}[(3K_3(\gamma)+\gamma K_2(\gamma))(u^0)^2 +3K_3(\gamma)|u|^2]c_{n+1}\Big)\nonumber\\
&-\partial_t(P_0a_{n+1})-\partial_t\Big(\frac{n_0K_3(\gamma)}{\gamma K_2(\gamma)}(u\cdot b_{n+1})\Big)\nonumber\\
&+n_0u\cdot\nabla_x\Big(hu^0  a_{n+1}+\frac{1}{\gamma K_2(\gamma)}(6K_3(\gamma)+\gamma K_2(\gamma))u^0(u\cdot b_{n+1})\label{energy01}\\
&+\frac{1}{\gamma K_2(\gamma)}[(3K_3(\gamma)+\gamma K_2(\gamma))(u^0)^2+3K_3(\gamma)|u|^2]c_{n+1} \Big)\nonumber\\
&+\nabla_x\cdot\Big(\frac{n_0K_3(\gamma)}{\gamma K_2(\gamma)}u^0b_{n+1}\Big)+\nabla_x\cdot\Big(\frac{2n_0K_3(\gamma)}{\gamma K_2(\gamma)}uc_{n+1}\Big)\nonumber\\
&+n_0u\cdot E_{n+1}+ n_0u\cdot E_0 a_{n+1}+(e_0+P_0)(u\cdot b_{n+1})(u\cdot E_0)\nonumber\\
&+P_0E_0\cdot b_{n+1}+(e_0+P_0)u^0(u\cdot E_0)c_{n+1}\nonumber\\
&+\int_{\mathbb R^3}  \hat{p}\sqrt{\mathbf{M}}{\{\bf I-P\}}\Big(\frac{F_{n+1}}{\sqrt{\mathbf{M}}}\Big)\,dp \cdot E_0\nonumber\\
&+ \sum_{\substack{k+l=n+1\\k,l\geq1}}\Big[n_0u\cdot E_k a_{l}+(e_0+P_0)(u\cdot b_{l})(u\cdot E_k)+P_0E_k\cdot b_{l}\nonumber\\
&+(e_0+P_0)u^0(u\cdot E_k)c_{l}+\int_{\mathbb R^3}  \hat{p}\sqrt{\mathbf{M}}{\{\bf I-P\}}\Big(\frac{F_{l}}{\sqrt{\mathbf{M}}}\Big)\,dp \cdot E_k\Big]=0.\nonumber
\end{align}
Equations \eqref{number01}, \eqref{moment01}, \eqref{energy01} can be further written as:
\begin{align}\label{number02}
& n_0u^0\Big(\partial_ta_{n+1}+h(u\cdot\partial_t b_{n+1})+hu^0\partial_tc_{n+1}\Big)-P_0\partial_tc_{n+1}\nonumber\\
&+n_0u^0\Big(\partial_t\Big(hu\Big)\cdot b_{n+1})+\partial_t(hu^0)c_{n+1}\Big)-(\partial_tP_0)c_{n+1}\nonumber\\
&+n_0u\cdot\Big( \nabla_xa_{n+1}+h \nabla_xb_{n+1}\cdot u+hu^0 \nabla_x c_{n+1}\Big)+P_0\nabla_x\cdot b_{n+1}\\
&+n_0u\cdot\Big( \nabla_x\Big(h u\Big)\cdot b_{n+1})+\nabla_x\Big(h u^0\Big) c_{n+1}\Big)
+b_{n+1}\cdot\nabla_xP_0\nonumber\\
&+\nabla_x\cdot\int_{\mathbb R^3}  \hat{p}\sqrt{\mathbf{M}}{\{\bf I-P\}}\Big(\frac{F_{n+1}}{\sqrt{\mathbf{M}}}\Big)\,dp=0,\nonumber
\end{align}

\begin{align}
& n_0u^0\Big(hu_j \partial_ta_{n+1}+\frac{1}{\gamma K_2(\gamma)}[(6K_3(\gamma)+\gamma K_2(\gamma)) u_j (u\cdot \partial_tb_{n+1}+u^0  \partial_tc_{n+1})\nonumber\\
& +K_3(\gamma) \partial_tb_{n+1,j}]\Big)-\frac{n_0K_3(\gamma)}{\gamma K_2(\gamma)}u_j\partial_tc_{n+1}-\partial_t\Big(\frac{n_0K_3(\gamma)}{\gamma K_2(\gamma)}u_j\Big)c_{n+1}\nonumber\\
&+n_0u^0\Big[\partial_t\Big(hu_j\Big) a_{n+1}+\partial_t\Big(\frac{m^2(6K_3(\gamma)+\gamma K_2(\gamma))}{\gamma K_2(\gamma)} u_j u\Big)\cdot b_{n+1}\nonumber\\
&+\partial_t\Big(\frac{1}{\gamma K_2(\gamma)}(6K_3(\gamma)+\gamma K_2(\gamma)) u_ju^0\Big) c_{n+1}+\partial_t\Big(\frac{K_3(\gamma)}{\gamma K_2(\gamma)}\Big)b_{n+1,j}\Big]\nonumber\\
&+n_0u\cdot\Big(hu_j \nabla_xa_{n+1}+\frac{1}{\gamma K_2(\gamma)}(6K_3(\gamma)+\gamma K_2(\gamma))u_j[( \nabla_xb_{n+1}\cdot u)+u^0\nabla_xc_{n+1}]\nonumber\\
&+K_3(\gamma) \nabla_xb_{n+1,j}]\Big)+P_0\partial_{x_j}a_{n+1}+\partial_{x_j}P_0a_{n+1}\nonumber\\
&+n_0u\cdot\Big[\nabla_x\Big(hu_j\Big) a_{n+1}+\nabla_x\Big(\frac{1}{\gamma K_2(\gamma)}(6K_3(\gamma)+\gamma K_2(\gamma))u_ju\Big)\cdot b_{n+1}\nonumber\\
&+\nabla_x\Big(\frac{1}{\gamma K_2(\gamma)}(6K_3(\gamma)+\gamma K_2(\gamma))u_ju^0\Big)c_{n+1}\nonumber\\
&+\nabla_x\Big(\frac{K_3(\gamma)}{\gamma K_2(\gamma)}\Big) b_{n+1,j}\Big]+\frac{n_0K_3(\gamma)}{\gamma K_2(\gamma)}u_j\nabla_x\cdot b_{n+1}\nonumber\\
&+ \frac{n_0K_3(\gamma)}{\gamma K_2(\gamma)}[(u\cdot \partial_{x_j}b_{n+1})+u^0\partial_{x_j}c_{n+1}]\nonumber\\
&+\nabla_x\cdot\Big(\frac{n_0K_3(\gamma)}{\gamma K_2(\gamma)}u\Big)b_{n+1,j}+\nabla_x\Big(\frac{n_0K_3(\gamma)}{\gamma K_2(\gamma)}u_j\Big)\cdot b_{n+1}\nonumber\\
&+ \partial_{x_j}\Big(\frac{n_0K_3(\gamma)}{\gamma K_2(\gamma)}u\Big)\cdot b_{n+1}+\partial_{x_j}\Big(\frac{n_0K_3(\gamma)}{\gamma K_2(\gamma)}u^0\Big)c_{n+1}\label{moment02}\\
&+E_{0,j}\Big(n_0u^0a_{n+1}+(e_0+P_0)u^0(u\cdot b_{n+1})+[e_0(u^0)^2+P_0|u|^2]c_{n+1}\Big)\nonumber\\
&+\Big[\Big(n_0u a_{n+1}+(e_0+P_0)u (u\cdot b_{n+1})+(e_0+P_0)u^0 u c_{n+1}\Big)\times B_0\Big]_j\nonumber\\
&+n_0u^0E_{n+1,j}+\Big(n_0u^0\times B_{n+1}\Big)_j\nonumber\\
&+\sum_{\substack{k+l=n+1\\k,l\geq1}}E_{k,j}\Big(n_0u^0a_{l}+(e_0+P_0)u^0(u\cdot b_{l})+[e_0(u^0)^2+P_0|u|^2]c_{l}\Big)\nonumber\\
&+\sum_{\substack{k+l=n+1\\k,l\geq1}}\Big[\Big(n_0u a_{l}+(e_0+P_0)u (u\cdot b_{l})+(e_0+P_0)u^0 u c_{l}\Big)\times B_k\Big]_j\nonumber\\
&+\nabla_x\cdot\int_{\mathbb R^3}  \frac{p_jp}{p^0}\sqrt{\mathbf{M}}{\{\bf I-P\}}\Big(\frac{F_{n+1}}{\sqrt{\mathbf{M}}}\Big)\,dp\nonumber\\
&+\Big[\int_{\mathbb R^3}  \hat{p}\times B_0\sqrt{\mathbf{M}}{\{\bf I-P\}}\Big(\frac{F_{n+1}}{\sqrt{\mathbf{M}}}\Big)\,dp\Big]_j\nonumber\\
&+\sum_{\substack{k+l=n+1\\k,l\geq1}}\Big[\int_{\mathbb R^3}  \hat{p}\times B_k\sqrt{\mathbf{M}}{\{\bf I-P\}}\Big(\frac{F_{l}}{\sqrt{\mathbf{M}}}\Big)\,dp\Big]_j=0,\nonumber
\end{align}
and
\begin{align}
& n_0u^0\Big(hu^0 \partial_ta_{n+1}+\frac{1}{\gamma K_2(\gamma)}[(6K_3(\gamma)+\gamma K_2(\gamma)) u^0](u\cdot \partial_tb_{n+1})\nonumber\\
&+\frac{1}{\gamma K_2(\gamma)}[(3K_3(\gamma)+\gamma K_2(\gamma))(u^0)^2 +3K_3(\gamma)|u|^2]\partial_tc_{n+1}\Big)\nonumber\\
&+ n_0u^0\Big[\partial_t(hu^0)a_{n+1}+\partial_t\Big(\frac{1}{\gamma K_2(\gamma)}[(6K_3(\gamma)+\gamma K_2(\gamma)) u^0]u\Big)\cdot b_{n+1}\nonumber\\
&+\partial_t\Big(\frac{1}{\gamma K_2(\gamma)}[(3K_3(\gamma)+\gamma K_2(\gamma))(u^0)^2 +3K_3(\gamma)|u|^2]\Big)c_{n+1}\Big]\nonumber\\
&-P_0\partial_ta_{n+1}-\frac{n_0K_3(\gamma)}{\gamma K_2(\gamma)}(u\cdot \partial_tb_{n+1})\nonumber\\
&-\partial_tP_0a_{n+1}-\partial_t\Big(\frac{n_0K_3(\gamma)}{\gamma K_2(\gamma)}u\Big)\cdot b_{n+1}\nonumber\\
&+n_0u\cdot\Big(hu^0 \nabla_x a_{n+1}+\frac{1}{\gamma K_2(\gamma)}(6K_3(\gamma)+\gamma K_2(\gamma))u^0( \nabla_xb_{n+1}\cdot u)\label{energy02}\\
&+\frac{1}{\gamma K_2(\gamma)}[(3K_3(\gamma)+\gamma K_2(\gamma))(u^0)^2+3K_3(\gamma)|u|^2]\nabla_xc_{n+1} \Big)\nonumber\\
&+n_0u\cdot\Big[\nabla_x\Big(hu^0\Big)  a_{n+1}+\nabla_x\Big(\frac{1}{\gamma K_2(\gamma)}(6K_3(\gamma)+\gamma K_2(\gamma))u^0u\Big)\cdot b_{n+1}\nonumber\\
&+\nabla_x\Big(\frac{1}{\gamma K_2(\gamma)}[(3K_3(\gamma)+\gamma K_2(\gamma))(u^0)^2+3K_3(\gamma)|u|^2]\Big)c_{n+1} \Big]\nonumber\\
&+\Big(\frac{n_0K_3(\gamma)}{\gamma K_2(\gamma)}u^0\Big)\nabla_x\cdot b_{n+1}+\Big(\frac{2n_0K_3(\gamma)}{\gamma K_2(\gamma)}u\Big)\cdot\nabla_x c_{n+1}\nonumber\\
&+\nabla_x\Big(\frac{n_0K_3(\gamma)}{\gamma K_2(\gamma)}u^0\Big)\cdot b_{n+1}+\nabla_x\cdot\Big(\frac{2n_0K_3(\gamma)}{\gamma K_2(\gamma)}u\Big)c_{n+1}+n_0u\cdot E_{n+1}\nonumber\\
&+ n_0u\cdot E_0 a_{n+1}+(e_0+P_0)(u\cdot b_{n+1})(u\cdot E_0)+P_0E_0\cdot b_{n+1}\nonumber\\
&+(e_0+P_0)u^0(u\cdot E_0)c_{n+1}+\int_{\mathbb R^3}  \hat{p}\sqrt{\mathbf{M}}{\{\bf I-P\}}\Big(\frac{F_{n+1}}{\sqrt{\mathbf{M}}}\Big)\,dp \cdot E_0\nonumber\\
&+\sum_{\substack{k+l=n+1\\k,l\geq1}}\Big[n_0u\cdot E_k a_{l}+(e_0+P_0)(u\cdot b_{l})(u\cdot E_k)+P_0E_k\cdot b_{l}\nonumber\\
&+(e_0+P_0)u^0(u\cdot E_k)c_{l}+\int_{\mathbb R^3}  \hat{p}\sqrt{\mathbf{M}}{\{\bf I-P\}}\Big(\frac{F_{l}}{\sqrt{\mathbf{M}}}\Big)\,dp \cdot E_k\Big]=0.\nonumber
\end{align}
Now we can write the equations \eqref{number02}, \eqref{moment02}, \eqref{energy02} as a linear symmetric hyperbolic system:
\begin{equation}\label{lsh}
{\bf A}_0\partial_t U+\sum_{i=1}^3{\bf A}_i\partial_i U +{\bf B}_1 U+{\bf B}_2\bar{U}=S,
\end{equation}
where $U$ and $\bar{U}$ are
$$U=\left(
\begin{array}{c}
\displaystyle a_{n+1} \\
\displaystyle b_{n+1}\\
\displaystyle c_{n+1}
\end{array}
\right),\qquad \bar{U}=\left(
\begin{array}{c}
\displaystyle E_{n+1} \\
\displaystyle B_{n+1}
\end{array}
\right).$$
For simplicity, denote
$$h_1=h_1(t,x)\equiv\frac{n_0}{\gamma K_2(\gamma)} (6K_3(\gamma)+\gamma K_2(\gamma)), \quad h_2=h_2(t,x)\equiv\frac{n_0K_3(\gamma)}{\gamma K_2(\gamma)} .$$
$5\times5$ Matrixes ${\bf A}_0$, ${\bf A}_i, (i=1,2,3)$ in \eqref{lsh} are
\begin{equation*}
{\bf A}_0=\left(
\begin{array}{ccc}
\displaystyle n_0u^0 & n_0u^0hu^t & [e_0(u^0)^2+P_0|u|^2]\vspace{0.5ex}\\
\displaystyle n_0u^0h u & (h_1u\otimes u+h_2{\bf I})u^0 & (h_1(u^0)^2-h_2)u\vspace{0.5ex}\\
\displaystyle [e_0(u^0)^2+P_0|u|^2] & (h_1(u^0)^2-h_2 )u^t & (h_1(u^0)^2-3h_2 )u^0
\end{array}
\right),
\end{equation*}
and
\begin{equation*}
{\bf A}_i=\left(
\begin{array}{ccc}
\displaystyle n_0u_i & n_0hu_iu^t+P_0 e_i^t & n_0hu^0u_i\vspace{0.5ex}\\
\displaystyle n_0hu_iu+P_0 e_i & h_1u_iu\otimes u+h_2(u_i {\bf I}+\tilde{\bf A}_i) & (h_1u_iu+h_2 e_i)u^0\vspace{0.5ex}\\
\displaystyle n_0hu^0u_i & (h_1u_iu^t+h_2 e_i^t)u^0 & (h_1(u^0)^2-h_2 )u_i
\end{array}
\right),
\end{equation*}
where $(\cdot)^t$ denotes the transpose of a vector in ${\mathbb R^3}$, ${\bf I}$ is the $3\times3$ identity matrix and $\tilde{\bf A}_i$ is a $3\times3$ matrix with its components
$$(\tilde{\bf A}_i)_{jk}=\delta_{ij}u_k+\delta_{ik}u_j,\qquad 1\leq j,k\leq 3.$$
The components of matrixes ${\bf B}_1$ and ${\bf B}_2$, and the remainder terms $S$ can be written explicitly as  functions of $F_i, E_i, B_i,~ (0\leq i\leq n$) and their first order derivatives.
For the matrix ${\bf A}_0$, its determinant is
\begin{align*}
&n_0^5\left|
\begin{array}{ccc}
\displaystyle u_0 & u^0u^th & h(u^0)^2-\frac{P_0}{n_0}\vspace{0.5ex}\\
\displaystyle {\bf0} & u^0[h_3u\otimes u+h_4{\bf I}] & (u^0)^2h_3u\vspace{0.5ex}\\
\displaystyle 0 & (u^0)^2h_3u^t &(u^0)^3h_3-u^0h_4+\frac{1}{\gamma^2u^0}
\end{array}
\right|\\
=&n_0^5(u^0)^6\left|
\begin{array}{cc}
\displaystyle  h_3u\otimes u+h_4{\bf I} & h_3u\vspace{0.5ex}\\
\displaystyle  h_3u^0u^t & u^0h_3-\frac{1}{u^0}h_4+\frac{1}{\gamma^2(u^0)^3}
\end{array}
\right|\vspace{0.5ex}\\
=&n_0^5(u^0)^6\left|
\begin{array}{cc}
\displaystyle h_4{\bf I} &  h_3u\vspace{0.5ex}\\
\displaystyle  \frac{u^t}{ u^0}\Big[h_4-\frac{1}{\gamma^2(u^0)^2}\Big]& u^0h_3-\frac{1}{u^0}h_4+\frac{1}{\gamma^2(u^0)^3}
\end{array}
\right|\vspace{0.5ex}\\
=&n_0^5(u^0)^6\left|
\begin{array}{cc}
\displaystyle h_4{\bf I} &  h_3u\vspace{0.5ex}\\
\displaystyle  {\bf 0}^t& \Big[\frac{-|u|^2}{ u^0h_4}\Big(h_4-\frac{1}{\gamma^2(u^0)^2}\Big)+u^0\Big]h_3-\frac{1}{u^0}h_4+\frac{1}{\gamma^2(u^0)^3}
\end{array}
\right|\vspace{0.5ex}\\
=&n_0^5(u^0)^6h_4^3\Big\{\Big[\frac{-|u|^2}{ u^0h_4}\Big(h_4-\frac{1}{\gamma^2(u^0)^2}\Big)+u^0\Big]h_3-\frac{1}{u^0}h_4+\frac{c^4}{\gamma^2(u^0)^3}\Big\},
\end{align*}
where ${\bf 0}$ denotes the column  vector $(0,0,0)^t$, $h_3$ and $h_4$ are functions with the following forms:
$$h_3=-\Big(\frac{K_1(\gamma)}{K_2(\gamma)}\Big)^2-\frac{2}{\gamma}\frac{K_1(\gamma)}{K_2(\gamma)}+1+\frac{8}{\gamma^2},\quad h_4=\frac{K_1(\gamma)}{\gamma K_2(\gamma)}+\frac{4}{\gamma^2}.$$
Noting
$$-\Big(\frac{K_1(\gamma)}{K_2(\gamma)}\Big)^2-\frac{3}{\gamma}\frac{K_1(\gamma)}{K_2(\gamma)}
           +1+\frac{3}{\gamma^2}>0$$
by Proposition $6.3$ in Appendix 3 \cite{Ruggeri-Xiao-Zhao-2019},
we can further obtain
\begin{align*}
|{\bf A}_0|>&n_0^5(u^0)^6h_4^3\frac{1}{u^0}(h_3-h_4)\\
           =&n_0^5(u^0)^5h_4^3\Big[-\Big(\frac{K_1(\gamma)}{K_2(\gamma)}\Big)^2-\frac{3}{\gamma}\frac{K_1(\gamma)}{K_2(\gamma)}
           +1+\frac{4}{\gamma^2}\Big]\\
           >&\frac{n_0^5(u^0)^5h_4^3}{\gamma^2}.
\end{align*}

On the other hand, the system \eqref{EM} can also be written as a linear symmetric hyperbolic system of $(E_{n+1}, B_{n+1})$:
\begin{equation}\label{emn1}
\partial_t \bar{U}+\sum_{i=1}^3{\bf \bar{A}}_i\partial_i \bar{U} +{\bf B}U=0,
\end{equation}
where ${\bf B}$ is a $6\times 5$ matrix whose components are functions of $n_0, u$. Denote ${\bf O}$ as a $3\times 3$ matrix with all components be $0$, and define matrixes:
\begin{align*}
{\bf \bar{A}}_{11}=\left(
\begin{array}{ccc}
\displaystyle 0 & 0 & 0\\
\displaystyle 0 & 0 & 1\\
\displaystyle 0 & -1 & 0
\end{array}
\right), \quad {\bf \bar{A}}_{12}=\left(
\begin{array}{ccc}
\displaystyle 0 & 0 & 0\\
\displaystyle 0 & 0 & -1\\
\displaystyle 0 & 1 & 0
\end{array}
\right),\quad
 {\bf \bar{A}}_{21}=\left(
\begin{array}{ccc}
\displaystyle 0 & 0 & -1\\
\displaystyle 0 & 0 & 0\\
\displaystyle 1 & 0 & 0
\end{array}
\right), \\
 {\bf \bar{A}_{22}}=\left(
\begin{array}{ccc}
\displaystyle 0 & 0 & 1\\
\displaystyle 0 & 0 & 0\\
\displaystyle -1 & 0 & 0
\end{array}
\right),\quad {\bf \bar{A}}_{31}=\left(
\begin{array}{ccc}
\displaystyle 0 & 1 & 0\\
\displaystyle -1 & 0 & 0\\
\displaystyle 0 & 0 & 0
\end{array}
\right),\quad{\bf \bar{A}}_{32}=\left(
\begin{array}{ccc}
\displaystyle 0 & -1 & 0\\
\displaystyle 1 & 0 & 0\\
\displaystyle 0 & 0 & 0
\end{array}
\right).
\end{align*}
Then ${\bf \bar{A}}_i$ in \eqref{emn1} can be expressed as
 $${\bf \bar{A}}_i=\left(\begin{array}{cc}
\displaystyle {\bf O} &  {\bf \bar{A}}_{i1}\\
\displaystyle   {\bf \bar{A}}_{i2} & {\bf O}
\end{array}\right).$$
 Combining \eqref{lsh} and \eqref{emn1}, we can obtain the wellposedness of $(a_{n+1}, b_{n+1}, c_{n+1}, E_{n+1}$, $B_{n+1}) \in C^0([0,\infty);H^N) $ with initial data
$a_{n+1}(0,x), b_{n+1}(0,x), c_{n+1}(0,x), E_{n+1}(0,x), B_{n+1}(0,x)$ $\in H^N, N\geq 0$ by the standard theorem of  linear symmetric system \cite{Barhouri-Chemin-Dachin-2011} (Chapter $4.2$). Moreover, for $n=0$, one has
$$\|\partial_t{\bf A}_0\|_{\infty}+\sum_{i=1}^{3}\|\nabla_x{\bf A}_i\|_{\infty}+\sum_{i=1}^{2}\|{\bf B}_i\|_{\infty}+\|S\|_{\infty}\lesssim (1+t)^{-\beta_0},\qquad \|S\|_{H^N}\lesssim 1.$$
Therefore, standard energy estimates on systems \eqref{lsh} and \eqref{emn1} yield
\begin{align*}
&\frac{d}{dt}\left(\|(a_1(t),b_1(t),c_1(t))\|_{H^N}^2+\|[E_1(t),B_1(t)]\|_{H^N}^2\right)\\
&\hspace{0.5cm}\lesssim(1+t)^{-\beta_0}\left(\|(a_1(t),b_1(t),c_1(t))\|_{H^N}^2+\|[E_1(t),B_1(t)]\|_{H^N}^2\right)\\
&\hspace{0.5cm}
+\left(\|(a_1(t),b_1(t),c_1(t))\|_{H^N}+\|[E_1(t),B_1(t)]\|_{H^N}\right).
\end{align*}
Namely, \begin{equation}\label{abc1}
\|(a_1(t),b_1(t),c_1(t))\|_{H^N}+\|[E_1(t),B_1(t)]\|_{H^N}\lesssim 1+t
\end{equation}
 by Gr\"{o}nwall's inequality.

 On the other hand, based on the second equation in \eqref{expan2}, we can modify the corresponding arguments in \cite[Lemma A.2]{Guo-CPAM-2006} to obtain
$${\{\bf I-P\}}\Big(\frac{F_{1}}{\sqrt{\mathbf{M}}}\Big)\lesssim \mathbf{M}^{1_-}.$$
It is straightforward to obtain \eqref{growth0} by the structure of $F_1$ and \eqref{abc1}.
Assuming \eqref{growth0} holds for $1\leq i\leq n$, the case $i=n+1$ for \eqref{growth0} holds again by the structure of the equation for $F_{n+1}$ in \eqref{expan2},  the induction assumption and
 similar analysis as the case $i=1$ ( $n=0$).

\end{proof}

\end{document}